\documentclass[11pt,reqno]{amsart}
\usepackage{amssymb}
\usepackage[usenames,dvipsnames]{color}

\newtheorem{theorem}{Theorem}[section]
\newtheorem*{theorem*}{Theorem}
\newtheorem*{lemma*}{Lemma}

\newtheorem{lemma}[theorem]{Lemma}
\newtheorem{cor}[theorem]{Corollary}

\theoremstyle{definition}
\newtheorem{remark}[theorem]{Remark}
\newcounter{tenumerate}

\newcommand{\lm}{\lambda}

\renewcommand{\epsilon}{\varepsilon}

\DeclareMathOperator{\var}{Var} \DeclareMathOperator{\Cov}{Cov}

\DeclareMathOperator{\Corr}{Corr}

\newcommand{\R}{{\mathbb R}}
\newcommand{\C}{{\mathbb C}}
\newcommand{\N}{{\mathbb N}}

\newcommand{\remove}[1]{}

\renewcommand{\ge}{\geqslant}
\renewcommand{\leq}{\leqslant}
\renewcommand{\geq}{\geqslant}

\def\XXint#1#2#3{{\setbox0=\hbox{$#1{#2#3}{\int}$}
\vcenter{\hbox{$#2#3$}}\kern-.5\wd0}}

\numberwithin{equation}{section}
\allowdisplaybreaks
\newcommand{\abbr}[1]{{\sc\lowercase{#1}}}

\begin{document}
\title[Persistence for AR]{Persistence versus stability \\
for auto-regressive processes}
\author[A.\ Dembo]{Amir Dembo$^\star$}
\author[J.\ Ding]{Jian Ding$^\dagger$}
\author[J.\ Yan]{Jun Yan$^\ddagger$}
\address{$^\star$Department of Mathematics, Stanford University, Building
380, Sloan Hall, Stanford, CA 94305, USA}
\address{$^\dagger$Department of Statistics, Wharton, 3730 Walnut St,
Philadelphia, PA 19104, USA}
\address{$^\ddagger$Department of Statistics, Stanford University, Sequoia
Hall, Stanford, CA 94305, USA}
\thanks{$^\star$\,$^\ddagger$Research partially supported by NSF grant
DMS-1613091.}
\thanks{$^\dagger$Research partially supported by NSF grant DMS-1313596.}
\date{\today }
\maketitle

\begin{abstract}
The stability of an Auto-Regressive (\textsc{\lowercase{ar}}) time sequence
of finite order $L$, is determined by the maximal modulus $r^\star$ among
all zeros of its generating polynomial. If $r^\star<1$ then the effect of
input and initial conditions decays rapidly in time, whereas for $r^\star>1$
it is exponentially magnified (with constant or polynomially growing
oscillations when $r^\star=1$). Persistence of such \textsc{\lowercase{ar}}
sequence (namely staying non-negative throughout $[0,N]$) with decent probability, requires the
largest positive zero of the generating polynomial to have the largest
multiplicity among all zeros of modulus $r^\star$. These objects are behind
the rich spectrum of persistence probability decay for \textsc{\lowercase{ar}%
}$_L$ with zero initial conditions and i.i.d. Gaussian input, all the way
from bounded below to exponential decay in $N$, with intermediate regimes of
polynomial and stretched exponential decay. In particular, for \textsc{%
\lowercase{ar}}$_3$ the persistence decay power is expressed via the tail
probability for Brownian motion to stay in a cone, exhibiting the discontinuity of such power decay
between the \textsc{\lowercase{ar}}$_3$ whose generating polynomial has complex
zeros of rational versus irrational angles.
\end{abstract}

%\author{Amir Dembo\thanks{% Partially supported by NSF grant DMS-1613091.} %EndAName
%, Jian Ding\thanks{% Partially supported by NSF grant DMS-1313596.} %EndAName
%\ and Jun Yan}

\section{Introduction}

The estimation of persistence probability, that is, the probability that a
sequence of random variables stay positive,%
\begin{equation*}
p_{N}=\mathbb{P}(\Omega _{N}^{+}):=\mathbb{P}(X_{n}\geq 0, \quad \forall 
n \in [0,N) ) \,,
\end{equation*}%
is one of the central themes of research in the theory of probability. 
This topic goes back more than fifty years, to the seminal works by
Rice \cite{rice1945mathematical} and  Sleipan  \cite{slepian62}, that 
ushered and motivated some of the most widely used general tools 
in the study of Gaussian processes. See also the influential early
contribution \cite{NR62} on persistence probabilities for Gaussian 
processes and \cite{longuet1962distribution} listing a host of applications in diverse
areas of physics and engineering. Indeed, there is much 
interest in this phenomena in theoretical physics
(e.g. \cite{ehrhardt2004persistence}, or the 
review \cite{SM08} and references therein), and
in the mathematical literature (see the survey 
\cite{AS15} and the many references therein). One
focal theme has been the study of $p_N$ for a 
centered Gaussian Stationary Process (\abbr{GSP}),
either in discrete or continuous time, where of particular
note is the recent progress, beginning with \cite{FF15,KK16}, on 
% establishing 
conditions for the exponential decay of $p_N$ for such case.
% (and studying the class of processes having super-exponentially 
% decaying $p_T$). 
The law of a \abbr{GSP} is completely determined by its spectral measure
and naturally, the persistence probability is often studied by spectral 
techniques (often in combination with tools from harmonic analysis),
see  \cite{FF15,feldheim2017persistence,feldheim2018probability}.
Another line of research focuses on $p_N$ for nearly stationary, Markov
sequences and processes. Here too,  one expects an exponential
rate of decay of $p_N$. This direction, which falls 
within the classical general theory of quasi-stationary distributions, 
has been explored successfully in \cite{AM17,aurzada2017persistence,
aurzada2018persistence}, advancing our understanding of  
the relation between the Markov chain parameters and
$-\lim_N \log p_N$. The key here is the ability to 
express the latter limit in terms of the leading eigenvalue of a
suitable sub-Markov operator (c.f. the review of such a 
relation in \cite{aurzada2017persistence}). We mention in
passing the related active research on large holes in the 
distribution of zeros for certain families of complex-valued 
Gaussian random analytic functions,
%(which is outside the scope of this concise literature review), 
and the work on persistence exponents in models of random 
environment, media, or scenery (e.g. \cite{aurzada2015persistence}
where time-reversibility compensates for the loss of 
Markov and Gaussian structure).

Our focus here is on the persistence probability for
auto-regressive processes. More precisely, for a (random) sequence $\{ \xi _{n} \}$ and
\begin{equation}
Q(z)=z^{L}-\sum_{j=1}^{L}a_{j}z^{L-j},  \label{def-Q}
\end{equation}%
we define the \textsc{\lowercase{AR}}$%
(Q) $, an auto-regressive process $\{ X_{n}\}$ generated by $Q(z)$ and $\{ \xi
_{n}\}$, by%
\begin{equation}
X_{n}=\sum_{j=1}^{L}a_{j}X_{n-j}+\xi _{n}=\sum_{j=0}^{n} h_{n-j} \xi_{j},%
\qquad \text{ }\forall \text{ }n\geq 0,  \label{def-X-h}
\end{equation}%
where the right equality holds for zero initial
conditions, namely with $X_{-L}=\ldots =X_{-1}=0$. Further, let
\begin{equation}\label{def:r-star}
r^{\star }:=\max \{|\lambda |: \lambda \in \Lambda \} \,, \quad 
\Lambda:=\{ \lambda: Q(\lambda)=0 \}
\end{equation}
denote the maximum modulus among the $L$ zeros of $Q(z)$. Recall that 
the corresponding linear system is stable, namely with $|h_n| \to 0$,
if and only if $r^{\star }<1$. Taking hereafter for  $\{ \xi_n \}$ i.i.d. variables 
yields an $\R^L$-valued Markov chain, which for a stable \textsc{\lowercase{AR}} 
admits a stationary distribution. Assuming also light-tails for $\{ \xi_n \}$,
an exponential decay of the persistence probability ensues for such processes.
Indeed, the study of properties of $\lim_N \{-\log p_N\}$ for stable \textsc{\lowercase{AR}} 
processes is used by \cite{aurzada2017persistence} to showcase the main themes of this 
operator-based approach. However, among the
unstable \textsc{\lowercase{AR}} sequences one finds
 a host of stochastic processes of much interest. Perhaps the most prominent one 
is the random walk (ie. \textsc{\lowercase{AR}}$(z-1)$ sequence),
for which sharp decay rate $p_N \sim N^{-1/2}$ is well known to hold
under minimal conditions on the law of the increments $\xi_n$ 
(c.f. \cite{AS15}). Going further,  \cite{DDG13} establishes a persistence 
power exponent for integrated random walk (ie. \textsc{\lowercase{AR}}$((z-1)^{2})$), 
universally over all mean-zero, square-integrable $\xi_n$. A different extension 
is provided in \cite{AB11} which utilizes the invariance principle to establish the persistence 
power exponent for weighted random walks.
% (going beyond the scope of \textsc{\lowercase{AR}} sequences).
However, not much is known beyond these isolated special cases, 
with this paper being the first systematic study of the rich 
persistence behavior across the unstable
Gaussian \textsc{\lowercase{AR}} processes.

Specifically, assume \abbr{WLOG} that $a_{L}\neq 0$ and let
$m(\lambda) \ge 1$ denote the multiplicity of each value $\lambda$ 
in the set $\Lambda$ of zeros of $Q(z)$, now represented as
\begin{equation}
\Lambda :=\{\lambda _{1}|_{m(\lambda _{1})},\ldots ,\lambda _{\ell
}|_{m(\lambda _{\ell })}\}.  \label{def-gamma}
\end{equation}%
To better relate persistence and stability, we further use the notations
\begin{equation}\label{def:m-max}
% r^{\star }:=\max \{|\lambda |:\lambda \in \Lambda \} \,, \quad 
\Lambda^{\star}:= \{\lambda \in \Lambda :|\lambda |=r^{\star }\}\,, \quad 
m^\star := \max \{ m(\lm) : \lm \in \Lambda^\star \} \,.
\end{equation}
Then, focusing on the Gaussian case, namely with i.i.d. $\xi_n \sim N(0,1)$, we show in the sequel that:
\begin{itemize}
\item When $r^{\star }=1$, $\left\vert h_{n}\right\vert $ grows polynomially
and could be decomposed to a strictly positive part and an oscillatory part.
Correspondingly, we decompose $X_{n}$ as the sum of a zero-angle component -
an integrated random walk (\textsc{\lowercase{IRW}}), an oscillatory
component - a rotated random walk, and small order terms. The
persistence probability decay is then determined by the interaction between the first two
parts.

\item When $r^{\star }>1$, $\left\vert h_{n}\right\vert \rightarrow \infty $
exponentially, yielding a highly unstable process. In particular, if $%
m(r^{\star })\geq m(\lambda )$ for all $\lambda \in \Lambda ^{\star }$, then $%
h_{n}$ stay positive for $n$ large enough, with a significant contribution to 
$(X_n)$ from the first $\xi _{n}$'s yielding a non-decaying in $N$ 
persistence probabilty. In contrast, when $m(\lambda _{0})>m(r^{\star })\geq 1$ 
for some $\lambda _{0}\in \Lambda ^{\star }$, the oscillatory component 
competes well with the unstable zero-angle part, resulting with a roughly 
polynomial decay of $p_N$.
\end{itemize}

More precisely, our first theorem describes the qualitative behavior 
of $p_N$ for the different regimes of the set of zeros of $Q(\cdot)$.
\begin{theorem}
\label{thm-1} 
The decay of the persistence probabilities $p_N$ for a Gaussian \textsc{\lowercase{AR}}$(Q)$ 
fit exactly one of the following classes:
\newline
\noindent (a) Constant regime: 
if $r^{\star }>1$, $m(r^{\star }) = m^\star$, then $\inf_{N} \{ p_N \} >0$.
\newline
\noindent (b) Exponential regime: if either $r^{\star }<1$ or 
$m(r^{\star})=0$, then for some $C<\infty$ 
\begin{equation}
C^{-1} e^{-CN}\leq p_{N}\leq C e^{-N/C} \,, \qquad \forall N\in \N \,.
\end{equation}%
\noindent (c) Stretched exponential regime: if $r^{\star }=1$,
$m(r^\star) \in [1,m^\star)$, then
\begin{equation}
p_{N}=\exp(-N^{\alpha + o(1)}) \,, \qquad \alpha = 1 -\frac{m(r^\star)}{m^\star} \,.
\label{eq-bound-stretched}
\end{equation}%
\noindent (d) Unstable positive mode dominated by an oscillatory one:
if $r^{\star }>1$, $m(r^{\star}) \in [1,m^\star)$, then
\begin{equation}
p_{N}=N^{-\alpha +o(1)}\,,\qquad 
\alpha = \frac{1}{2} \sum_{\lambda \in \Lambda
^{\star}} (m(\lambda )-m(r^\star))_{+} (m(\lambda)-m(r^\star)+1)_{+}\,.
\end{equation}%
(e) Approximate \textsc{\lowercase{IRW}}: if $r^{\star}=1$, 
$m(r^\star) = m^\star$, then for some $C<\infty$,
\begin{equation}
C^{-1} N^{-C} \leq p_{N} \leq C N^{-1/C}\,, \qquad \forall N \in \N \,.
\end{equation}%
\noindent
\end{theorem}

\begin{remark}
Theorem \ref{thm-1} does not establish the existence of exponents in part
(b) and (e). In the stable case of part (b) with $r^\star<1$, starting the process 
at stationary initial conditions, one has the existence and formula for the 
exponent from \cite{AM17}. While a power exponent is to be expected in part (e),
its existence in such full generality remains an open problem.
%as is finding a formula for it,  
% in the sequel we prove  the existence of such power exponent 
% for one special case.
\end{remark}

Part (a)--(e) of Theorem \ref{thm-1} are established in
Section \ref{section-constant}-\ref{section-app-int-rw} respectively. 

\medskip
Recall that a flexible condition for the continuity of the persistence exponent 
of centered \abbr{GSP}'s is derived in \cite{DM15}. In contrast, our next theorem 
demonstrates the possibly highly discontinuous nature of the power exponent in part (e), 
by analyzing the special case of  \abbr{AR}$(Q)$ processes with
$\Lambda = \{1, e^{i \theta}, e^{- i \theta}\}$. Its proof
in Section \ref{section-power-law}, further elaborates the connection in this case
between $p_N$ and the probability that a Brownian motion stays for a long time
in a generalized cone. 
\begin{theorem}
\label{thm-discontinuity} 
Denote by $\mathcal{Q}$ the collection of real polynomials 
$Q(z)=(z-1)(z-e^{i \theta})(z-e^{-i \theta})$, where 
$Q_{\ell} \rightarrow Q$ in $\mathcal{Q}$ if the corresponding angles converge. 
To each $Q \in \mathcal{Q}$ corresponds a finite $\beta _{Q}>0$ such that
the persistence probabilities of the Gaussian \abbr{AR}$(Q)$ process satisfy
\begin{equation}
p_{N}=N^{-\beta _{Q}+o(1)}\,.
\end{equation}%
Furthermore, if the angle $\theta $ for $Q$ is such that $%
\theta /2\pi \in \mathbb{Q}$, then there exist some $Q_{\ell }\in 
\mathcal{Q}$ with $Q_{\ell }\rightarrow Q$ such that $\liminf_{\ell
\rightarrow \infty }\beta _{Q_{\ell }}>\beta _{Q}$. In contrast, for 
$\theta /2\pi \not\in \mathbb{Q}$ any $Q_{\ell
}\rightarrow Q$ results with $\beta _{Q_{\ell}} \to \beta _{Q}$.
\end{theorem}

\textbf{Acknowledgment.} We thank S.R.S. Varadhan for stimulating
discussions and for pointing out reference \cite{BNV94}, and we thank Nike
Sun and Allan Sly for interesting discussions. The work was initiated when
J.D. was a postdoc at MSRI program on Random Spacial Processes.

\section{Bounded persistence probability: Theorem~\protect\ref{thm-1} (a) 
\label{section-constant}}
 
Using throughout the convention $\lambda := |\lambda| e^{i \theta_\lambda}$,
with $\theta _{\lambda} \in [0,2\pi)$ denoting the angle of $\lambda \in \C$,
we start with a few standard linear algebra facts 
(e.g. see \cite[Section 3.6]{BrockwellDavis}).
\begin{lemma}
\label{lem-general-regression-coefficients}
(a). The solution of the 
%linear 
difference equation
\begin{equation}
q_{\ell }=\sum_{j=1}^{L}a_{j}q_{\ell-j}\,, \qquad \ell \geq L\,,
\label{general-regression-relation}
\end{equation}
with $\{a_j\}$ the coefficients of the 
polynomial $Q(\cdot)$ of \eqref{def-Q}, is 
for $\Lambda $ of \eqref{def-gamma}, of the form 
\begin{equation}
q_{\ell }=\sum_{\lambda \in \Lambda }\sum_{j=0}^{m(\lambda )-1}\beta
_{\lambda ,j}\lambda ^{\ell }\ell ^{j}\,,
\label{general-regression-coefficients}
\end{equation} 
where   
$\beta _{\lambda,j} \in \mathbb{C}$ are uniquely 
chosen so \eqref{general-regression-coefficients} matches 
the initial conditions $(q_{0},q_{1},\ldots,q_{L-1})$.
If $(q_{0},q_{1},\ldots,q_{L-1}) \in \R^L$ 
then $\beta _{\lambda,j}=\overline{\beta} _{\overline{\lambda},j}$.
Conversely, to any 
$\{ \beta _{\lambda ,j}:\lambda \in \Lambda ,0\leq j\leq m(\lambda
)-1\}$ with $\beta _{\lambda ,j}=\overline{\beta}_{\overline{\lambda},j}$
corresponds $(q_{0},q_{1},\ldots,q_{L-1}) \in \R^L$ such that 
\eqref{general-regression-coefficients} holds.

\noindent
(b) The solution $\{h_{n}\}$ of \eqref{def-X-h} is of the form 
\begin{equation}
h_{\ell }=\sum_{\lambda \in \Lambda }\sum_{j=0}^{m(\lambda )-1}a_{\lambda
,j}|\lambda |^{\ell }\ell ^{j}\cos (\ell \theta _{\lambda }+\theta _{\lambda
,j}),  \label{def-h}
\end{equation}
for some $a_{\lambda ,j} \in \C$ and $\theta _{\lambda,j} \in \R$ 
that are determined by $Q(\cdot)$, such that   
$a_{\lambda,m(\lambda)-1}\neq 0$, $a_{\lambda,j}=a_{\bar{\lambda},j}$,
$\theta _{\lambda ,j}=-\theta _{\overline{\lambda},j}$ and 
\abbr{WLOG} we set $\theta_{\lambda,j}=0$ if $\theta_\lambda=0$.
\end{lemma}

Equipped with Lemma \ref{lem-general-regression-coefficients}, 
we next show that $p_N$ is uniformly bounded away from zero
when $r^\star > 1$ and $m(r^\star)=m^\star$ of \eqref{def:m-max}.

\begin{proof}[Proof of Theorem~\protect\ref{thm-1}(A)]
Denote $b_{n,i}:=h_{n-i}$. From (\ref{def-h})%
\begin{equation}
|b_{n,i}|\leq C(r^{\star })^{n-i}(n-i)^{m(r^{\star })-1}\text{,}
\label{bound-coefficient-b}
\end{equation}%
for a constant $C<\infty $. \abbr{WLOG} we assume that $%
a_{r^{\star },m(r^{\star })-1}>0$ and $\theta _{r^{\star },m(r^{\star
})-1}=0 $. With Lemma \ref{lem-general-regression-coefficients} and Lemma~%
\ref{lemma-convolution}, we can verify that for $n\geq n_{1}$%
\begin{equation*}
\sum_{i=1}^{n_{1}}b_{n,i}(r^{\star })^{i}=\sum_{i=1}^{n_{1}}a_{r^{\star
},m(r^{\star })-1}(r^{\star })^{n}(n-i)^{m(r^{\star })-1}+o\left( (r^{\star
})^{n}n^{m(r^{\star })-1}\right) ,
\end{equation*}%
which further implies that there exists $n_{1}$ sufficiently large such that%
\begin{equation}
\sum_{i=1}^{n_{1}}b_{n,i}(r^{\star })^{i}\geq 2(r^{\star })^{n}n^{m(r^{\star
})-1}\mbox{
for all }n\geq n_{1}.  \label{bn-r-star-bound-1}
\end{equation}%
Applying Lemma~\ref{lem-general-regression-coefficients}, we have that for
some $C<\infty $,%
\begin{equation}
\sum_{i=1}^{n_{1}}\left\vert b_{n,i}\right\vert (r^{\star })^{i}\leq
C\sum_{i=1}^{n_{1}}(r^{\star })^{n}(n-i)^{m(r^{\star })-1}\leq
Cn_{1}(r^{\star })^{n}n^{m(r^{\star })-1}.  \label{bn-r-star-bound-2}
\end{equation}%
Writing $\xi _{i}=$ $(r^{\star })^{i}+(\xi _{i}-(r^{\star })^{i})$, in light
of (\ref{def-X-h}), (\ref{bn-r-star-bound-1}) and (\ref{bn-r-star-bound-2}),
choosing $\delta =1/(3Cn_{1})$, we have that under the event $\mathcal{A}%
_{\delta }:=\{1+2\delta \leq (r^{\star })^{-i}\xi _{i}\leq 1+3\delta
,i=1,...,n_{1}\}$,%
\begin{equation}
\sum_{i=1}^{n_{1}}b_{n,i}\xi _{i}\geq 2(r^{\star })^{n}n^{m(r^{\star
})-1}-3\delta Cn_{1}(r^{\star })^{n}n^{m(r^{\star })-1}\geq (r^{\star
})^{n}n^{m(r^{\star })-1}\text{, }\mbox{
for all }n\geq n_{1}\,.  \label{eq-b-n-i-convolution}
\end{equation}%
Define%
\begin{equation*}
\widehat{X}_{n}:=\sum_{i=1}^{n\wedge n_{1}}b_{n,i}\xi _{i}\text{ and }%
X_{n}^{\star }=X_{n}-\widehat{X}_{n}.
\end{equation*}%
By \eqref{eq-b-n-i-convolution} we can see that for a constant $c>0$,%
\begin{equation}
\mathbb{P}(\widehat{X}_{n}\geq (r^{\star })^{n}n^{m(r^{\star })-1}%
\mbox{ for
all }n\geq n_{1})\geq \mathbb{P}(\mathcal{A}_{\delta })\geq c.
\label{eq-large-n-1}
\end{equation}%
Note that $\widehat{X}_{n}\in \mathrm{span}\{\widehat{X}_{n_{1}+1},\ldots ,%
\widehat{X}_{n_{1}+L}\}$ for all $n\geq n_{1}+L$. Combined with %
\eqref{eq-b-n-i-convolution}, it follows that there exists a set $B\subseteq
(\mathbb{R}^{+})^{L}$ with positive Lebesgue measure such that%
\begin{equation}
\widehat{X}_{n}\geq (r^{\star })^{n}n^{m(r^{\star })-1}\mbox{ for all }n\geq
n_{1}+L\mbox{ if }(\widehat{X}_{n_{1}+1},\ldots ,\widehat{X}_{n_{1}+L})\in
B\,.  \label{eq-in-span}
\end{equation}%
Note that $(\widehat{X}_{1},\ldots ,\widehat{X}_{n_{1}+L})$ is a Gaussian
vector of full rank (the covariance matrix is strictly positive definite),
and thus the conditional Gaussian vector of $(\widehat{X}_{1},\ldots ,%
\widehat{X}_{n_{1}})$ given $(\widehat{X}_{n_{1}+1},\ldots ,\widehat{X}%
_{n_{1}+L})$ is also of full rank. Combined with \eqref{eq-large-n-1} and %
\eqref{eq-in-span}, it follows that for a constant $c^{\prime }>0$%
\begin{equation}
\mathbb{P}(\widehat{X}_{n}\geq (r^{\star })^{n}n^{m(r^{\star })-1}%
\mbox{ for
all }n\geq n_{1}\mbox{
and }X_{n}\geq 0\mbox{ for }1\leq n\leq n_{1})\geq c^{\prime }.
\label{eq-truncate-good}
\end{equation}%
Under the event $\mathcal{B}:=\{|\xi _{i}|\leq C_{0}^{-1}(1-\alpha /r^{\star
})\alpha ^{i}\mbox{ for all }i\geq n_{1}+1\}$, for fixed $\alpha \in
(1,r^{\star })$, we have%
\begin{equation*}
\left\vert X_{n}^{\star }\right\vert =|\sum\limits_{i=n_{1}+1}^{n}b_{n,i}\xi
_{i}|\leq \sum\limits_{i=n_{1}+1}^{n}C(r^{\star })^{n-i}(n-i)^{m(r^{\star
})-1}C_{0}^{-1}(1-\alpha /r^{\star })\alpha ^{i}\leq \frac{C_{1}}{C_{0}}%
\left( r^{\star }\right) ^{n}n^{m(r^{\star })-1},
\end{equation*}%
for some $C_{1}<\infty $ independent of $n$. We can choose $C_{0}>0$ such
that%
\begin{equation}
\{|X_{n}^{\star }|\leq (r^{\star })^{n}n^{m(r^{\star })-1}\mbox{ for all }%
n\geq n_{1}\}\supseteq \mathcal{B}.  \label{constant-regime-last-eq}
\end{equation}%
By standard Gaussian bounds the sequence $\mathbb{P}(|\xi _{i}|\geq
C_{0}^{-1}(1-\alpha /r^{\star })\alpha ^{i})$ is summable, hence by
independence $\mathbb{P}(\mathcal{B})>0$. Therefore with (\ref%
{constant-regime-last-eq}) we get for a constant $c^{\prime \prime }>0$%
\begin{equation*}
\mathbb{P}(|X_{n}^{\star }|\leq (r^{\star })^{n}n^{m(r^{\star })-1}%
\mbox{ for
all }n\geq n_{1})\geq \mathbb{P}(\mathcal{B})>0.
\end{equation*}%
Combined with \eqref{eq-truncate-good}, by independence of $X_{n}^{\star }$
and $\widehat{X}_{n}$ the proof is completed.
\end{proof}

\section{Exponential decay: Theorem~\protect\ref{thm-1} (b)\label%
{section-exponential}}

We start by an easy exponential lower bound.

\begin{lemma}
\label{lem-exponential-lower} For $(X_{n})$ \textsc{\lowercase{AR}}$(Q)$,
there exists $0<c<1$ such that $p_{N}\geq c^{N}$ for all $N\in \mathbb{N}$.
\end{lemma}

\begin{proof}
Conditioned on $\{0\leq X_{i}\leq 1\mbox{
for all }1\leq i\leq n\}$, the Gaussian variable $X_{n+1}$ has mean $\mu $
with $|\mu |\leq \sum_{i=1}^{L}|a_{i}|$ and variance 1. Therefore,%
\begin{equation*}
\mathbb{P}(0\leq X_{n+1}\leq 1\mid \{0\leq X_{i}\leq 1\mbox{ for all }1\leq
i\leq n\})\geq c
\end{equation*}%
for a constant $c>0$. Recursively applying this inequality yields the
desired lower bound.
\end{proof}

Briefly, there are two forces resulting in an exponentially persistence
probability upper bound: negative dependence with archetype $%
X_{n}=-X_{n-1}+\xi _{n}$, which we handle in Lemma~\ref{lem-sub-negative};
and almost independence with archetype $X_{n}=\xi _{n}$, which we deal with
in Corollary~\ref{cor-almost-independence-exponential}.

\subsection{Negative dependence: $r^{\star }>1$ and $m(r^{\star })=0$
\label{subsection-negative-dependence}}

For a polynomial
\begin{equation*}
P(z)=\sum_{i=0}^{L}c_{i}z^{i},
\end{equation*}
we define the operation $P$ on a sequence $\mathbf{x}=(x_{n})$ by $P(\mathbf{%
x})_{n}=\sum_{i=0}^{L}c_{i}x_{n-L+i}$. We claim that $(P_{1}P_{2})(\mathbf{x}%
)=P_{1}(P_{2}(\mathbf{x}))$, since if we denote $P_{j}(z)=%
\sum_{i=0}^{L_{j}}c_{ji}z^{i}$ for $j=1,2$, then it is straightforward to
verify that the coefficients of $x_{n-L+i}$ in $(P_{1}P_{2})(\mathbf{x})_{n}$
and in $P_{1}(z)P_{2}(z)$ are same for all $0\leq i\leq L_{1}+L_{2}$. We say
a polynomial is a non-negative polynomial if it has non-negative
coefficients for each term. Now we are ready to prove the main result in
this subsection.

\begin{lemma}
\label{lem-sub-negative} Let $(X_{n})$ be \textsc{\lowercase{AR}}($Q$) for $%
Q(\cdot )$ with $r^{\star }>1$ and $m(r^{\star })=0$. Then, $p_{N}\leq C%
\mathrm{e}^{-cN}$ for some $c>0,C<\infty $ and all $N\in \mathbb{N}$.
\end{lemma}

\begin{proof}
For $\lambda \in \Lambda ^{\star }$, let $\mathbf{Y}_{\lambda
}=(Y_{n,\lambda })$ be an auto-regressive process satisfying%
\begin{equation*}
P_{\lambda ,m(\lambda )}(\mathbf{Y}_{\lambda
})_{n}=\sum\limits_{k=0}^{d(\lambda )}\gamma _{\lambda ,k}\xi _{n-k},
\end{equation*}%
where%
\begin{equation*}
\text{ for }\lambda \in \Lambda ^{\star }\backslash \{-r^{\star }\}\text{, }%
P_{\lambda ,m(\lambda )}:=(z-\lambda )^{m(\lambda )}(z-\bar{\lambda}%
)^{m(\lambda )},d(\lambda ):=2m(\lambda )-1\text{,}
\end{equation*}%
\begin{equation*}
\text{for }\lambda =-r^{\star }\text{ (if }-r^{\star }\in \Lambda \text{), }%
P_{-r^{\star },m(-r^{\star })}:=(z+r^{\star })^{m(-r^{\star })},d(\lambda
):=m(-r^{\star })-1,
\end{equation*}%
and $\gamma _{\lambda ,k}$'s are coefficients to be determined later. By
Lemma \ref{lem-general-regression-coefficients}, we can choose proper $%
\gamma _{\lambda ,k}$'s such that%
\begin{equation*}
Y_{n,\lambda }=\sum\limits_{i=0}^{n}\widetilde{b}_{\lambda ,n,i}\xi _{i},%
\text{ }\widetilde{b}_{\lambda ,n,i}=\sum\limits_{j=0}^{m(\lambda
)-1}a_{\lambda ,j}\left\vert \lambda \right\vert ^{n-i}(n-i)^{k}\cos
((n-i)\theta _{\lambda }+\theta _{\lambda ,j}).
\end{equation*}%
Compared with Lemma~\ref{lem-general-regression-coefficients}, we could
write $X_{n}=Y_{n}+Z_{n}$ where $Y_{n}=\sum_{\lambda \in \Lambda ^{\star
}}Y_{n,\lambda }$ and $\var(Z_{n})\leq Cr^{2n}$ for constants $C<\infty $
and $1<r<r^{\star }$. We first claim that%
\begin{equation}
\mathbb{P}(Y_{n}\geq 0\mbox{ for all }N/2\leq n\leq N)\leq \mathrm{e}%
^{-cN}\,,\mbox{ for a constant }c>0\,.  \label{eq-persistence-Y_n}
\end{equation}%
To this end, we apply Corollary~\ref{cor-sub-negative} and deduce that for
each $\lambda $, there exists a polynomial $P_{\lambda ,m(\lambda )}^{\prime
}$ such that $P_{\lambda ,m(\lambda )}P_{\lambda ,m(\lambda )}^{\prime }$ is
non-negative. We can further request that for all $\lambda $, $P_{\lambda
,m(\lambda )}P_{\lambda ,m(\lambda )}^{\prime }$'s equal to the same
non-negative polynomial $P$, because if this doesn't hold, we can consider
the following polynomials which achieve this request,%
\begin{equation*}
\widetilde{P}_{\lambda ,m(\lambda )}^{\prime }=P_{\lambda ,m(\lambda
)}^{\prime }\prod\limits_{\lambda ^{\prime }\in \Lambda ^{\star }\backslash
\{\lambda \}}P_{\lambda ^{\prime },m(\lambda ^{\prime })}P_{\lambda ^{\prime
},m(\lambda ^{\prime })}^{\prime }.
\end{equation*}%
Write $\mathbf{Y}=(Y_{n})$. On the event $\{Y_{n}>0:n\in \lbrack N]\}$, we
have%
\begin{equation}
0<P(\mathbf{Y})_{n}=\sum_{\lambda \in \Lambda ^{\star }}P_{\lambda
,m(\lambda )}^{\prime }(P_{\lambda ,m(\lambda )}(\mathbf{Y}_{\lambda }))_{n}.
\label{eq-bi-not-all-zero}
\end{equation}%
Notice that $\sum_{\lambda \in \Lambda ^{\star }}P_{\lambda ,m(\lambda
)}^{\prime }(P_{\lambda ,m(\lambda )}(\mathbf{Y}_{\lambda
}))_{n}=\sum_{i=0}^{K}c_{i}\xi _{n-i}$ for some $K$ and $c_{i}$'s, it
follows that there exist $\{b_{i}:i=1,\ldots ,K\}$ such that%
\begin{equation*}
\zeta _{j}:=\sum_{i=j(K+1)}^{j(K+1)+K}b_{i}\xi _{jK+i}>0\text{ on }%
\{Y_{n}>0:n\in \lbrack N]\}
\end{equation*}%
for $j=\lceil N/2K\rceil ,\lceil N/2K\rceil +1,\ldots ,\lfloor N/K\rfloor $.
If $b_{i}$'s are all $0$, obviously the probability that $\zeta _{j}>0$ is $%
0 $ for all $j$. If $b_{i}$'s are not all $0$, since $\zeta _{j}$ are
i.i.d.\ centered Gaussian variables, it is easy to see that %
\eqref{eq-persistence-Y_n} holds. Choosing $\widetilde{r}$ with $r<%
\widetilde{r}<r^{\star }$, then the fact that $\var(Y_{n})\geq
C_{1}(r^{\star })^{2n}$ for some $C_{1}>0$ and $\var(Z_{n})\leq Cr^{2n}$
implies that for some constants $c_{0}$ and $C_{0}$ $>0$%
\begin{equation}
\sum_{N/2\leq n\leq N}\left[ (\mathbb{P}(|Y_{n}|\leq \widetilde{r}^{n})+%
\mathbb{P}(Z_{n}\geq \widetilde{r}^{n})\right] \leq C_{0}\mathrm{e}%
^{-c_{0}N}\,.  \label{bound-Yn-Zn}
\end{equation}%
Combining now (\ref{eq-persistence-Y_n}), (\ref{bound-Yn-Zn}) and {\small 
\begin{equation*}
\Omega _{N}^{+}\subset \left\{ \bigcap\limits_{n=N/2}^{N}\{Y_{n}\geq
0\}\right\} \bigcup \left\{ \bigcup_{n=N/2}^{N}\{|Y_{n}|\leq (\widetilde{r}%
^{n})\}\right\} \bigcup \left\{ \bigcup_{n=N/2}^{N}\{Z_{n}\geq (\widetilde{r}%
^{n})\}\right\}
\end{equation*}%
}completes the proof of the lemma.
\end{proof}

\subsection{Nearly independence: $r^{\star }<1$ or $r^{\star }=1$, $%
m(r^{\star })=0$\label{subsection-nearly-independence}}

\begin{lemma}\label{lem-almost-independence-exponential} 
Let $\{Z_{n}:n=1,\ldots ,N\}$ be
a mean zero Gaussian process with non-negative correlation coefficients 
$\{\rho _{i,j}\}_{1\leq i,j\leq N}$ such that $\sum_{i,j}^{N}\rho _{i,j}\leq
\Delta N$. Then there exists a constant $C=C(\Delta )<\infty $ such that
\begin{equation*}
\mathbb{P}(Z_{n}\geq 0\mbox{ for all }1\leq n\leq N)\leq C\mathrm{e}
^{-N/200\Delta }\,.
\end{equation*}
\end{lemma}

\begin{proof}
\abbr{WLOG} we assume that $\var(Z_{i})=1$. For convenience
of notation, denote by $\mathcal{A}_{N}:=\{Z_{n}\geq 0\mbox{ for all }1\leq
n\leq N\}$. Clearly $\Gamma _{i}(z):=\mathbb{P}(\mathcal{A}_{N}\mid Z_{i}=z)-%
\mathbb{P}(\mathcal{A}_{N})/\mathbb{P}(Z_{i}\geq 0)$ is an increasing
function on $z$ in $[0,\infty )$, and $\int_{0}^{\infty }\Gamma
_{i}(z)p_{Z_{i}}(z)dz=0$ (since $\mathcal{A}_{N}\mathcal{\subset }\left\{
Z_{i}\geq 0\right\} $). Hence%
\begin{equation}
0\leq \int_{0}^{\infty }z\Gamma _{i}(z)p_{Z_{i}}(z)dz=\mathbb{P}(\mathcal{A}%
_{N})\left( {\mathbb{E}}\left[ Z_{i}\mid \mathcal{A}_{N}\right] -{\mathbb{E}}%
\left[ Z_{i}\mid Z_{i}\geq 0\right] \right) .
\label{conditional_increasing_1}
\end{equation}%
Noting that ${\mathbb{E}}\left[ Z_{i}\mid Z_{i}\geq 0\right] =\sqrt{2/\pi }$%
, we deduce from (\ref{conditional_increasing_1}) that%
\begin{equation}
{\mathbb{E}}\left[ Z_{i}\mid \mathcal{A}_{N}\right] \geq \sqrt{2/\pi }\,.
\label{eq-conditional-exp-A}
\end{equation}%
By assumption $S_{N}=\sum_{i=1}^{N}Z_{i}$ is a centered Gaussian with Var$%
(S_{N})=\sum_{i,j}\rho _{i,j}\leq \Delta N$. Hence, recalling that for $G$
standard Gaussian $\mathbb{P}(G\geq x)\leq \exp (-x^{2}/2)$, we have%
\begin{equation*}
\mathbb{P}(S_{N}\geq 10^{-1}N)\leq e^{-N/(200\Delta )}.
\end{equation*}%
Noting that%
\begin{equation*}
\mathbb{P(\mathcal{A}}_{N}\mathbb{)}\leq \frac{\mathbb{P}(S_{N}\geq 10^{-1}N)%
}{\mathbb{P}(S_{N}\geq 10^{-1}N\mid \mathcal{A}_{N})},
\end{equation*}%
it thus suffices to show that%
\begin{equation}
\mathbb{P}(S_{N}\geq 10^{-1}N\mid \mathcal{A}_{N})\geq 1/(100\sqrt{\Delta }).
\label{eq-SN-condition-A}
\end{equation}%
To this end, by FKG we know that $\mathbb{P(\mathcal{A}}_{N})\mathbb{\geq }%
2^{-N}$. Hence, recalling that $\upsilon _{N}:=$Var$(S_{N})\leq \Delta N$,
we get by simple Gaussian calculation that%
\begin{equation}
{\mathbb{E}}\left[ S_{N}1_{S_{N}\geq 10\sqrt{\Delta }N}\mid \mathcal{A}_{N}%
\right] \leq 2^{N}{\mathbb{E}}\left[ S_{N}1_{S_{N}\geq 10\sqrt{\Delta }N}%
\right] \leq 2^{N}\frac{\sqrt{\upsilon _{N}}}{\sqrt{2\pi }}e^{-\frac{(10%
\sqrt{\Delta }N)^{2}}{2\upsilon _{N}}}\leq e^{-N}.
\label{eq-expectation-SN-condition-A}
\end{equation}%
Further, by (\ref{eq-conditional-exp-A}),%
\begin{equation}
\sqrt{\frac{2}{\pi }}N\leq {\mathbb{E}}\left[ S_{N}\mid \mathcal{A}_{N}%
\right] .  \label{eq-recall-conditional-exp-A}
\end{equation}%
Breaking the range of $S_{N}$ into $[0,10^{-1}N)$, $[10^{-1}N,10\sqrt{\Delta 
}N)$ and $[10\sqrt{\Delta }N,\infty )$, we deduce that%
\begin{equation*}
{\mathbb{E}}\left[ S_{N}\mid \mathcal{A}_{N}\right] \leq 10^{-1}N+10\sqrt{%
\Delta }N\mathbb{P}(S_{N}\geq 10^{-1}N\mid \mathcal{A}_{N})+{\mathbb{E}}%
\left[ S_{N}1_{S_{N}\geq 10\sqrt{\Delta }N}\mid \mathcal{A}_{N}\right] .
\end{equation*}%
In view of (\ref{eq-expectation-SN-condition-A}) and (\ref%
{eq-recall-conditional-exp-A}) yields (\ref{eq-SN-condition-A}).
\end{proof}

We record here the following comparison result of Slepian \cite{slepian62},
which will be used repeatedly.

\noindent \textbf{Slepian's Lemma} \label{thm-slepian} Let $\{U_{i}:1\leq
i\leq n\}$ and $\{V_{i}:1\leq i\leq n\}$ be centered zero Gaussian variables
with%
\begin{equation}
\var(U_{i})=\var(V_{i}),\mbox{ and }\mathrm{\Cov}(U_{i},U_{j})\leq \mathrm{%
\Cov}(V_{i},V_{j})\mbox{ for all }1\leq i,j\leq n\,.  \label{eq-assumption}
\end{equation}%
Then for all real numbers $\lambda _{1},\ldots ,\lambda _{n}$,%
\begin{equation*}
\mathbb{P}(U_{i}\leq \lambda _{i}\mbox{ for all }1\leq i\leq n)\leq \mathbb{P%
}(V_{i}\leq \lambda _{i}\mbox{ for all }1\leq i\leq n)\,.
\end{equation*}

\begin{cor}
\label{cor-almost-independence-exponential} If $r^{\star }<1$, or $r^{\star
}=1$ with $m(r^{\star })=0$, then there exist constants $c,C>0$ such that $%
p_{N}\leq C\mathrm{e}^{-cN}$ for all $N\in \mathbb{N}$.
\end{cor}

\begin{proof}
First write $Q=Q_{1}Q_{2}$ where $Q_{1}$ has no positive zeroes and $Q_{2}$
has only positive zeroes. By Corollary~\ref{cor-sub-negative}, there exists
a non-negative polynomial $P_{1}$ such that $P_{1}Q_{1}$ is non-negative.
Now, let $(\zeta _{n})$ be an auto-regressive process generated by $Q_{2}$
and $(\xi _{n})$. Then $(X_{n})$ can be viewed as an auto-regressive process
generated by $Q_{1}$ and $(\zeta _{n})$, since $Q_{2}(Q_{1}(X_{n})_{n})_{n}=%
\xi _{n}$ and $Q_{2}(\zeta _{n})_{n}=\xi _{n}$ imply that(easy to verify the
boundary) $Q_{1}(X_{n})_{n}=\zeta _{n}$. Analogues to the proof of Lemma~\ref%
{lem-sub-negative}, there exist $k>0$ and $\{b_{i}:1\leq i\leq k\}$ such that%
\begin{equation}
\{X_{n}>0:1\leq n\leq N\}\,\subseteq \{\chi _{j}>0:1\leq j\leq \lbrack
N/k]\}\,,  \label{eq-Omega-chi}
\end{equation}%
where $\chi _{j}=\sum_{i=j(k+1)}^{j(k+1)+k}b_{i}\zeta _{jk+i}$. If $b_{i}$'s
are all $0$, then this corollary is automatically true by (\ref{eq-Omega-chi}%
). If $b_{i}$'s are not all $0$, since $Q_{2}$ has only positive zeroes
strictly less than 1, it is easy to verify that there exists constants $%
c<\infty $ and $0<r<1$ such that the correlation coefficients of $\{\chi
_{i}\}$ satisfy%
\begin{equation}
\left\vert \rho (\chi _{i},\chi _{j})\right\vert \leq cr^{\left\vert
i-j\right\vert }\text{ for all }i,j\text{.}
\label{correlation-bound-almost-independence-case}
\end{equation}%
Note that we can find a positive integer $k_{0}$ such that the matrix $%
\mathsf{A}%\mathbf{A}
_{n}:=(a_{ij})_{n\times n}$ with $a_{ij}=cr^{\left\vert i-j\right\vert
k_{0}} $ is positive definite for each $n$, thus we can construct a Gaussian
vector $(Z_{1}^{\prime },...,Z_{n}^{\prime })$ whose covariance matrix is $%
\mathsf{A}_{n}$. Since $0<r<1$, we can apply Lemma \ref%
{lem-almost-independence-exponential} to $(Z_{1}^{\prime },...,Z_{n}^{\prime
})$, and it completes the proof with (\ref%
{correlation-bound-almost-independence-case}) and Slepian's Lemma, by noting
that $\Omega _{N}^{+}\subseteq \{\chi _{jk_{0}}\geq 0:1\leq j\leq \lbrack
N/(k\times k_{0})]\}$.
\end{proof}

\section{Stretched exponential decay: Theorem \protect\ref{thm-1} (c)\label%
{section-stretched-exponential}}

Assume $\lambda _{0}=1$ with multiplicity $m:=m(1)$. Let $\Lambda =\Lambda
^{\star }\cup \{\widetilde{\lambda }_{1}|_{m(\widetilde{\lambda }_{1})},...,%
\widetilde{\lambda }_{q}|_{m(\widetilde{\lambda }_{q})}\}$ where $|%
\widetilde{\lambda }_{i}|<1$, and $\Lambda ^{\star }=\{\lambda
_{0}|_{m},\lambda _{1}|_{m_{1}},\ldots ,\lambda _{\ell }|_{m_{\ell }}\}$
where $\lambda _{j}=\mathrm{e}^{\sqrt{-1}\theta _{j}}$, such that $m^{\prime
}=\max_{j\geq 1}\{m_{j}\}>m$. For a Gaussian random variable $Z$ and $%
n_{1}\leq n_{2}$, we denote the contribution of $\{\xi _{n}:n_{1}\leq n\leq
n_{2}\}$ in $Z$ as%
\begin{equation}
Z(n_{1},n_{2}):={\mathbb{E}}(Z\mid \{\xi _{n}:n_{1}\leq n\leq n_{2}\}).
\label{def-Z-0,n}
\end{equation}%
For $M\in \mathbb{Z}^{+}$, denote by $(Z_{n,M})$ the \textsc{\lowercase{AR}}$%
((z-1)^{M})$, \textsc{\lowercase{IRW}} of order $M$, that is,%
\begin{equation}
Z_{n,M}:=\sum_{i=1}^{n}b_{n,i,M-1}\xi _{i}\text{,}  \label{eq-def-b-n-i-m}
\end{equation}%
where%
\begin{equation}
b_{n+1,i,M}:=b_{n,i,M}+b_{n+1,i,M-1}\text{ and }b_{n,i,0}:=1\text{ for all }%
i\leq n.  \label{eq-b-iterative}
\end{equation}%
For $j\in \{0,...,\ell \}$ and $k\in \lbrack m_{j}]-1$, define $\mathbf{%
\mathbb{\overrightarrow{\mathbb{T}}}}_{j,n,k}=[\mathbb{T}_{j,n,k},\mathbb{T}%
_{j,n,k}^{^{\prime }}]^{T}$, where%
\begin{equation}
\mathbb{T}_{j,n,k}:=\sum_{i=1}^{n}b_{n,i,k}\cos ((n-i)\theta _{j}+\theta
_{\lambda _{j},k})\xi _{i}\,\text{, }\mathbb{T}_{j,n,k}^{\prime
}:=\sum_{i=1}^{n}b_{n,i,k}\sin ((n-i)\theta _{j}+\theta _{\lambda
_{j},k})\xi _{i}.  \label{def-new-T}
\end{equation}%
Hereafter, $\theta _{j,k}=\theta _{\lambda _{j},k}$, and for $p\in \lbrack q]
$, $k\in \lbrack m(\widetilde{\lambda }_{p})]-1$, let%
\begin{equation}
Z_{p,n,k}:=\sum_{i=1}^{n}|\widetilde{\lambda }_{p}|^{n-i}b_{n,i,k}\cos
((n-i)\theta _{\widetilde{\lambda }_{p}}+\theta _{\widetilde{\lambda }%
_{p},k})\xi _{i}.  \label{def-Z-p-n-k}
\end{equation}%
The upper and lower bounds are established in Sections \ref%
{subsection-stretched-lower-bound} and \ref{subsection-stretched-upper-bound}%
, respectively.

\subsection{Upper bound\label{subsection-stretched-lower-bound}}

Comparing (\ref{def-new-T}), (\ref{def-Z-p-n-k}) with (\ref%
{lem-general-regression-coefficients}), we have the decomposition%
\begin{equation}
X_{n}=X_{n}^{(0)}+X_{n}^{(1)}+X_{n}^{(2)}:=\sum_{k=0}^{m-1}c_{0,k}\mathbb{T}%
_{0,n,k}+\sum_{j=1}^{\ell }\sum_{k=0}^{m_{j}-1}c_{j,k}\mathbb{T}%
_{j,n,k}+\sum_{p=1}^{q}\sum_{k=0}^{m(\widetilde{\lambda }_{q})-1}\widetilde{c%
}_{p,k}Z_{p,n,k},  \label{eq-Xn-expression-stretch}
\end{equation}%
for some real $c_{j,k}$, $\widetilde{c}_{p,k}\in \mathbb{R}$ where $%
c_{j_{1},k}=c_{j_{2},k}$ if $\theta _{j_{1}}=-\theta _{j_{2}}$ (and $%
\widetilde{c}_{p_{1},k}=\widetilde{c}_{p_{2},k}$ if $\widetilde{\theta }%
_{p_{1}}=-\widetilde{\theta }_{p_{2}}$). Set%
\begin{equation}
\gamma _{1}=m-1/2+\alpha \text{, }\alpha =\frac{m^{\prime }-m}{2m^{\prime }}.
\label{def-gamma-1}
\end{equation}%
Define the event ${\tiny \Xi }_{N}{\tiny :=\Xi }_{N}^{(0)}\cap {\tiny \Xi }%
_{N}^{(1)}$ where%
\begin{eqnarray*}
{\tiny \Xi }_{N}^{(0)} &:&=\bigcap_{n=N/2}^{N}{\tiny \{|X}_{n}^{(0)}{\tiny %
|\leq N}^{\gamma _{1}}{\tiny (}\log {\tiny N)}^{-2}{\tiny \}}\bigcap
\bigcap_{n=1}^{N}{\tiny \{|\xi }_{n}{\tiny |\leq }\sqrt{N}{\tiny \},} \\
{\tiny \Xi }_{N}^{(1)} &:&=\bigcup_{n=N/2}^{3N/4}{\tiny \{X_{n}^{(1)}\leq -}%
N^{\gamma _{1}}{\tiny \}\,.}
\end{eqnarray*}%
Note that under $\Xi _{N}$, ${\tiny |\xi }_{n}{\tiny |\leq }\sqrt{N}$, by (%
\ref{def-Z-p-n-k}), (\ref{eq-Xn-expression-stretch}), (\ref{def-gamma-1})
and the fact that $|\widetilde{\lambda }_{i}|<1$, we have that $%
\max_{n=N/2}^{N}\{X_{n}^{(2)}\}=o(N^{\gamma _{1}})$, hence $\Omega
_{N}^{+}\subseteq (\Xi _{N})^{c}$ for all $N$ large enough, so it suffices
to prove%
\begin{equation}
\mathbb{P}({\tiny \Xi }_{N}^{(0)})\leq \mathrm{e}^{-N^{2\alpha +o(1)}}\,
\label{bound-stretch-prob-t-0-n}
\end{equation}%
and%
\begin{equation}
\mathbb{P}({\tiny \Xi }_{N}^{(1)})\leq \mathrm{e}^{-N^{2\alpha +o(1)}}.
\label{bound-stretch-event-1}
\end{equation}%
Note that $\var(X_{n}^{(0)})=O(N^{2m-1})$ and $\var(\xi _{n})=1$, having $%
2\alpha <1$, (\ref{bound-stretch-prob-t-0-n}) follows by a union bound. Next
we show (\ref{bound-stretch-event-1}). Define the rotation matrix%
\begin{equation}
\mathsf{R}_{\theta }:=\left[ 
\begin{array}{cc}
\cos \theta  & -\sin \theta  \\ 
\sin \theta  & \cos \theta 
\end{array}%
\right] .  \label{def-R-theta}
\end{equation}%
Notice that by (\ref{def-new-T}) and our assumption on $\theta _{j,k}$'s in
Lemma \ref{lem-general-regression-coefficients}, we have%
\begin{equation}
\mathbb{T}_{j_{1},n,k}=\mathbb{T}_{j_{2},n,k}\text{ if }\theta
_{j_{1}}=-\theta _{j_{2}},  \label{eq-T-j1-equal-T-j2}
\end{equation}%
and for $n^{\prime }\geq n$, iterating over $i=1,...,k$ we get from (\ref%
{eq-b-iterative}) that{\small 
\begin{eqnarray}
\mathbf{\mathbb{\overrightarrow{\mathbb{T}}}}_{j,n^{\prime }+1,k}(0,n) &=&%
\mathsf{R}_{\theta _{j}}\mathbf{\overrightarrow{\mathbb{T}}}_{j,n^{\prime
},k}(0,n)+\mathsf{R}_{\theta _{j,k}-\theta _{j,k-1}}\mathbf{\overrightarrow{%
\mathbb{T}}}_{j,n^{\prime }+1,k-1}\,(0,n)  \label{eq-T-j-n+1-k} \\
&=&\mathsf{R}_{\theta _{j}}(\sum_{i=0}^{k}\mathsf{R}_{\theta _{j,k}-\theta
_{j,i}}\mathbf{\overrightarrow{\mathbb{T}}}_{j,n^{\prime },i})(0,n).  \notag
\end{eqnarray}%
}Further iterating over $n^{\prime }=n,n+1,...,n+s-1$, $s\in \mathbb{Z}^{+}$
yields the identity%
\begin{equation}
\mathbf{\overrightarrow{\mathbb{T}}}_{j,n+s,k}(0,n)=\left( \mathsf{R}%
_{\theta _{j}}\right) ^{s}(\sum_{i=0}^{k}P_{s}(k-i)\mathsf{R}_{\theta
_{j,k}-\theta _{j,i}}\mathbf{\overrightarrow{\mathbb{T}}}_{j,n,i})(0,n),
\label{eq-T-j-n+s-k}
\end{equation}%
as in (\ref{eq-s-iter}), for $P_{s}(\cdot )$ as in (\ref{eq-def-P-s}). Thus,
letting%
\begin{equation*}
\mathcal{D}_{n,s}:=\{|\sum_{j=1}^{\ell }\sum_{k=0}^{m_{j}-1}c_{j,k}\mathbb{T}%
_{j,n+s,k}(n,n+s)|\leq N^{\gamma _{1}}\},
\end{equation*}%
by Lemma \ref{Lemma-roration-general-case}{\small 
\begin{equation}
{\tiny \Xi }_{N}^{(1)}\supseteq \bigcup_{n=N/2}^{3N/4-M}\{|X_{n}^{(1)}|\geq 
\frac{2}{C^{\prime }}N^{\gamma _{1}}\}\cap
\bigcap_{n=N/2}^{3N/4-M}\bigcap_{s=1}^{M}\mathcal{D}_{n,s}.
\label{relation-stretch-sets-1}
\end{equation}%
} Since $M$ is a constant independent of $N$, by a union bound we have%
\begin{equation}
\mathbb{P}(\bigcup_{n=N/2}^{3N/4-M}\bigcup\limits_{s=1}^{M}\mathcal{D}%
_{n,s}^{C})\leq \mathrm{e}^{-N^{2\gamma _{1}+o(1)}}\,.
\label{bound-stretch-n-n+s}
\end{equation}%
Set $\beta =\frac{\gamma _{1}}{m^{\prime }-1/2}$. Conditioning on the $%
\sigma $-field $\sigma (\xi _{1},\ldots ,\xi _{\lbrack N/2+(i-1)N^{\beta }]})
$, it is easy to see that $\sum_{j=1}^{\ell }\sum_{k=0}^{m_{j}-1}c_{j,k}%
\mathbb{T}_{j,[N/2+iN^{\beta }],k}\thicksim N(\mu _{N},\sigma _{N}^{2})$
where $\sigma _{N}^{2}=$ $O(N^{2\gamma _{1}})$. Note that $\exists \delta >0$
such that%
\begin{equation*}
\sup_{\mu }\{\mathbb{P}(|N(\mu ,\sigma _{N}^{2})|\leq 2\sqrt{\sigma _{N}^{2}}%
/C^{\prime })\}\leq 1-\delta \text{.}
\end{equation*}%
Thus we obtain that there exists $0<\epsilon <1$ such that{\small 
\begin{equation*}
\mathbb{P}(|\sum_{j=1}^{\ell }\sum_{k=0}^{m_{j}-1}c_{j,k}\mathbb{T}%
_{j,[N/2+iN^{\beta }],k}|\leq \frac{2}{C^{\prime }}N^{\gamma _{1}}\mid
\bigcap\limits_{s=1}^{i-1}|\sum_{j=1}^{\ell }\sum_{k=0}^{m_{j}-1}c_{j,k}%
\mathbb{T}_{j,[N/2+sN^{\beta }],k}|\leq \frac{2}{C^{\prime }}N^{\gamma
_{1}})\leq \epsilon ,
\end{equation*}%
}for all $1\leq i\leq \lbrack N^{1-\beta }/4]$. Recalling that $\gamma
_{1}=m-1/2+\alpha $, we get (\ref{bound-stretch-event-1}), completing the
proof on the upper bound.

\subsection{Lower bound\label{subsection-stretched-upper-bound}}

The intuition for the lower bound comes from the fact that the event of
persistence will present if the \textsc{\lowercase{IRW}} sits above a
certain curve while the rotated (integrated) random walk sits below it.
Furthermore, the probability for the intersection of the two events is close
to the product of each of the probabilities. The latter claim requires a
careful justification, to which end we use the following lemma.

\begin{lemma}
\label{lem-conditional-gaussian}Assume that $(U_{1},U_{2},V_{1},...,V_{d})$
is a multivariate Gaussian random vector, with mean $(\mu _{U_{1}},...,\mu
_{V_{d}})$ and correlation matrix (not the covariance matrix)%
\begin{equation*}
\left( 
\begin{array}{cc}
\mathsf{\Sigma }_{U,U} & \mathsf{\mathsf{\Sigma }}_{U,V} \\ 
\mathsf{\Sigma }_{U,V}^{T} & \mathsf{\Sigma }_{V,V}%
\end{array}%
\right) .
\end{equation*}%
Write $\mathsf{\Sigma }_{U_{i},V}$ as the $i$th row of $\mathsf{\Sigma }%
_{U,V}$ for $i=1,2$. Then for the conditional expectation we have%
\begin{equation*}
\left\vert \mathbb{E}[U_{1}\mid V_{1},...,V_{d}]-\mu _{U_{1}}\right\vert
\leq d\sqrt{\var(U_{1})}||\mathsf{\Sigma }_{V,V}^{-1}||_{op}||\mathsf{\Sigma 
}_{U_{1},V}||_{\infty }||(V_{i}/\sqrt{\var(V_{i})})_{i\in \lbrack
d]}||_{\infty }.
\end{equation*}%
What's more, for the conditional covariance we have%
\begin{equation*}
|\Cov\lbrack U_{1},U_{2}\mid V_{1},...,V_{d}]-\Cov\lbrack U_{1},U_{2}]|\leq d%
\sqrt{\var(U_{1})\var(U_{2})}||\mathsf{\Sigma }_{U_{1},V}||_{\infty }||%
\mathsf{\Sigma }_{U_{2},V}||_{\infty }||\mathsf{\Sigma }_{V,V}^{-1}||_{op}.
\end{equation*}
\end{lemma}

\begin{proof}
Recall that from the conditional multivariate Gaussian formula we have%
\begin{equation}
\mathbb{E}[U_{1}\mid V_{1},...,V_{d}]=\mu _{U_{1}}+\sqrt{\var(U_{1})}\mathsf{%
\Sigma }_{U_{1},V}\mathsf{\Sigma }_{V,V}^{-1}(V_{i}/\sqrt{\var(V_{i})}%
)_{i\in \lbrack d]}^{T}\text{,}  \label{eq-conditional-gaussian-mean}
\end{equation}%
and%
\begin{equation}
\Cov\lbrack U_{1},U_{2}\mid V_{1},...,V_{d}]=\Cov\lbrack U_{1},U_{2}]-\sqrt{%
\var(U_{1})\var(U_{2})}\mathsf{\Sigma }_{U_{1},V}\mathsf{\Sigma }_{V,V}^{-1}%
\mathsf{\Sigma }_{U_{2},V}\text{.}  \label{eq-conditional-gaussian-var}
\end{equation}%
Lemma \ref{lem-conditional-gaussian} is a direct result from (\ref%
{eq-conditional-gaussian-mean}) and (\ref{eq-conditional-gaussian-var}).
\end{proof}

It is easy to check that for $s_{1},s_{2}\in \lbrack m]-1$,%
\begin{equation*}
\Corr(\mathbb{T}_{0,n_{k+1},s_{1}}(n_{k},n_{k+1}),\mathbb{T}%
_{0,n_{k+1},s_{2}}(n_{k},n_{k+1}))=\frac{2\sqrt{\left( s_{1}+1/2\right)
\left( s_{2}+1/2\right) }}{s_{1}+s_{2}+1}+o(1)\text{.}
\end{equation*}%
Write $\mathsf{\Sigma }_{m}$ as a $m\times m$ matrix with $\mathsf{\Sigma }%
_{m}(i,j)=2\sqrt{\left( i-1/2\right) \left( j-1/2\right) }/(i+j-1)$. Then we
can find $\varepsilon >0$ and $a_{0},...,a_{m-1}>\varepsilon $, such that if 
$b_{i}\in \lbrack a_{i}-\varepsilon ,a_{i}+\varepsilon ]$ $\forall i\in
\lbrack m]-1$, then%
\begin{equation}
\mathsf{\Sigma }_{m}^{-1}(b_{0},...,b_{m-1})^{T}>0\text{ entrywisely.}
\label{eq-sig-inv-b-positive}
\end{equation}%
To see this, we can choose $a_{i}=\mathsf{\Sigma }_{m}(i+1,1)+...+\mathsf{%
\Sigma }_{m}(i+1,m)$, in which case $\mathsf{\Sigma }%
_{m}^{-1}(a_{0},...,a_{m-1})^{T}=(1,1,...,1)^{T}$ since $\mathsf{\Sigma }%
_{m}^{-1}\mathsf{\Sigma }_{m}=\mathsf{I}$. Letting $\varepsilon $ be
sufficiently small gives (\ref{eq-sig-inv-b-positive}). Consider $%
n_{k}=2^{k} $ for $k=1,\ldots ,[\log _{2}N]$. Let%
\begin{equation*}
\alpha =\frac{m^{\prime }-m}{2m^{\prime }}\text{, }\beta =\frac{\gamma _{1}}{%
m^{\prime }-1/2}\text{, }n_{k,\alpha }=n_{k}^{\gamma _{1}}\text{, and }%
\gamma _{2}(\ell ):=-2\ell -2m^{\prime }-10.
\end{equation*}%
For convenience of notation, we write $\widetilde{\mathbb{T}}_{j,r,s}=%
\mathbb{T}_{j,n_{k}+rn_{k}^{\beta },s}(n_{k}+(r-1)n_{k}^{\beta
},n_{k}+rn_{k}^{\beta })$, and $\widetilde{Z}_{p,r,s}=Z_{p,n_{k}+rn_{k}^{%
\beta },s}(n_{k}+(r-1)n_{k}^{\beta },n_{k}+rn_{k}^{\beta })$. Define the
events{\small 
\begin{equation*}
\Omega _{k,0}:=\{\mathbb{T}_{0,n_{k+1},s}(n_{k},n_{k+1})\in n_{k,\alpha
}(\log N)^{10}n_{k}^{s-m}(a_{s}-\varepsilon ,a_{s}+\varepsilon )%
\mbox{ for
all }s\in \lbrack m]-1\},
\end{equation*}%
} 
\begin{equation*}
\Omega _{k,\mathbb{T}}:=\bigcap_{r=1}^{n_{k}^{1-\beta }}\{\left\vert 
\widetilde{\mathbb{T}}_{j,r,s}\right\vert \leq N^{\gamma _{2}(m_{j})}:j\in
\lbrack \ell ],s\in \lbrack m_{j}]-1\}\,,
\end{equation*}%
\begin{equation*}
\Omega _{k,Z}:=\bigcap_{r=1}^{n_{k}^{1-\beta }}\{\left\vert \widetilde{Z}%
_{p,r,s}\right\vert \leq N^{\gamma _{2}(m(\widetilde{\lambda }_{j}))}:p\in
\lbrack q],s\in \lbrack m(\widetilde{\lambda }_{j})]-1\}\,,
\end{equation*}%
and%
\begin{equation*}
\Omega _{k}:=\Omega _{k,0}\bigcap \Omega _{k,\mathbb{T}}\bigcap \Omega
_{k,Z}.
\end{equation*}

\begin{lemma}
\label{lem-Omega-k-stretched} Using the above definitions, there exists a
constant $C>0$ such that for all $1\leq k\leq \lbrack \log _{2}N]$,%
\begin{equation*}
\mathbb{P}(\Omega _{k})\geq \mathrm{e}^{-C(\log N\cdot n_{k}^{1-\beta
}+(\log N)^{20}n_{k}^{2\alpha })}\,.
\end{equation*}
\end{lemma}

\begin{proof}
It suffices to consider $K_{0}\leq k\leq \log _{2}N$ for a large number $%
K_{0}$ independent of $N$. Fix such a $k$. From the definition of $%
\widetilde{\mathbb{T}}_{j,r,s}$ we can verify that $\var\widetilde{\mathbb{T}%
}_{j,r,s}\leq O(N^{2m_{j}-2})$ $\forall s\in \lbrack m_{j}]-1$, thus we can
verify that there exists $C>0$ such that for any $j\in \lbrack \ell ],s\in
\lbrack m_{j}]-1$ and $r\in \lbrack n_{k}^{1-\beta }]$, 
\begin{equation}
\mathbb{P}(\left\vert \widetilde{\mathbb{T}}_{j,r,s}\right\vert \leq \frac{1%
}{\sqrt{c_{1}}}N^{\gamma _{2}(m_{j})})\geq N^{-C}.
\label{lower-bound-P-T-j-r-s}
\end{equation}%
Similarly for some $C_{0}>0$ and any $p\in \lbrack q],s\in \lbrack m(%
\widetilde{\lambda }_{j})]-1,r\in \lbrack n_{k}^{1-\beta }]$, we have%
\begin{equation}
\mathbb{P}(\left\vert \widetilde{Z}_{p,r,s}\right\vert \leq \frac{1}{\sqrt{%
c_{1}}}N^{\gamma _{2}(m(\widetilde{\lambda }_{j}))})\geq N^{-C_{0}}.
\label{lower-bound-Z-1}
\end{equation}%
Note the fact that for multivariate Gaussian $(W_{1},...,W_{d})$ with $\var%
(W_{i})=1$ $\forall i\in \lbrack d]$, we can write each $W_{i}$ as a linear
combination with i.i.d. standard Gaussian random variables $U_{1},...,U_{d}$
with each coefficient having absolute value smaller than 1. Thus for any $%
w_{1},...,w_{d}>0$,%
\begin{equation*}
\mathbb{P}(\left\vert W_{i}\right\vert \leq w_{i}\text{ }\forall i\in
\lbrack d])\geq \mathbb{P}(\left\vert U_{i}\right\vert \leq \min (w_{i})/d%
\text{ }\forall i\in \lbrack d])=\mathbb{P}(\left\vert U_{1}\right\vert \leq
\min (w_{i})/d)^{d}\text{.}
\end{equation*}%
By normalizing $\widetilde{\mathbb{T}}_{j,r,s}$ and $\widetilde{Z}_{p,r,s}$
to have variance 1 and using the above argument, we have that{\small 
\begin{equation*}
\mathbb{P}(\left\vert \widetilde{\mathbb{T}}_{j,r,s_{1}}\right\vert
,\left\vert \widetilde{Z}_{p,r,s_{2}}\right\vert \leq N^{\gamma _{2}}\text{, 
}\forall \text{ }j\in \lbrack \ell ],s_{1}\in \lbrack m(\lambda
_{j})]-1,p\in \lbrack q],s_{2}\in \lbrack m(\widetilde{\lambda }%
_{p})]-1)\geq N^{-C_{1}},\,
\end{equation*}%
}for some $C_{1}>0$. By independence we further see that for some $C>0$,%
\begin{equation}
\mathbb{P}(\Omega _{k,\mathbb{T}},\Omega _{k,\mathbb{Z}})\geq \mathrm{e}%
^{-C\log N\cdot n_{k}^{1-\beta }}\,.  \label{bound-P-omega-k-t}
\end{equation}%
With Lemma~\ref{lemma-convolution}, for all $j\in \lbrack \ell ],r\in
\lbrack n_{k}^{1-\beta }]$, $s\in \lbrack m]-1$ and $s_{1}\in \lbrack
m(\lambda _{j})]-1$, $s_{2}\in \lbrack m(\widetilde{\lambda }_{j})]-1$,%
\begin{eqnarray}
\rho (\mathbb{T}_{0,n_{k+1},s}(n_{k}+(r-1)n_{k}^{\beta },n_{k}+rn_{k}^{\beta
}),\widetilde{\mathbb{T}}_{j,r,s_{1}}) &=&O(n_{k}^{-(\beta +1)/2})\,  \notag
\\
\rho (\mathbb{T}_{0,n_{k+1},s}(n_{k}+(r-1)n_{k}^{\beta },n_{k}+rn_{k}^{\beta
}),\widetilde{Z}_{p,r,s_{2}}) &=&O(n_{k}^{-1/2})\,.  \label{eq-stretched-cor}
\end{eqnarray}%
Define 
\begin{equation*}
\mathcal{F}_{\mathbf{T}}:=\sigma (\widetilde{\mathbb{T}}_{j,r,s_{1}},%
\widetilde{Z}_{p,r,s_{2}}:j\in \lbrack \ell ],s_{1}\in \lbrack m(\lambda
_{j})]-1,p\in \lbrack q],s_{2}\in \lbrack m(\widetilde{\lambda }%
_{p})]-1,r\in \lbrack n_{k}^{1-\beta }]).
\end{equation*}%
By \eqref{eq-stretched-cor}, considering the conditional distribution
blockwisely for $(n_{k}+(r-1)n_{k}^{\beta },n_{k}+rn_{k}^{\beta })$ where $%
r\in \lbrack n_{k}^{1-\beta }]$, with Lemma \ref{lem-conditional-gaussian},
we can verify that there exist constants $c,C>0$ (independent of $N$) such
that for all $s\in \lbrack m]-1$,%
\begin{equation*}
cn_{k}^{2s+1}\leq \var(\mathbb{T}_{0,n_{k+1},s}(n_{k},n_{k+1})\mid \mathcal{F%
}_{\mathbf{T}})\leq Cn_{k}^{2s+1}\,,
\end{equation*}%
and%
\begin{equation*}
{\mathbb{E}}(\mathbb{T}_{0,n_{k+1},s}(n_{k},n_{k+1})\mid \Omega _{k,\mathbb{T%
}},\Omega _{k,\mathbb{Z}})=O(1)\text{.}
\end{equation*}%
It follows that for a constant $C^{\prime }>0$,%
\begin{equation*}
\mathbb{P}(\Omega _{k,0}\mid \Omega _{k,\mathbb{T}},\Omega _{k,\mathbb{Z}%
})\geq \mathrm{e}^{-C^{\prime }(\log N)^{20}n_{k}^{2\alpha }}\,.
\end{equation*}%
Combined with (\ref{bound-P-omega-k-t}), this completes the proof of the
lemma.
\end{proof}

Now, we further define $\Omega _{k}^{\star }$ to be the intersection of $%
\Omega _{k}$ and $\Omega _{k,R}$, where{\small 
\begin{eqnarray*}
\Omega _{k,R}:= &&\bigcap_{t=n_{k}}^{n_{k+1}}\{\mathbb{T}_{0,t,s}(n_{k},t)%
\geq -\kappa (\log N)^{10}n_{k,\alpha }\text{, for any }s\in \lbrack m]-1\}
\\
\bigcap &&\bigcap_{t=n_{k}}^{n_{k+1}}\{\left\vert \mathbb{T}%
_{j,t,s}(n_{k},t)\right\vert \leq \kappa (\log N)^{10}n_{k,\alpha },%
\mbox{ for all
}j\in \lbrack \ell ],s\in \lbrack m(\lambda _{j})]-1\} \\
\bigcap &&\bigcap_{t=n_{k}}^{n_{k+1}}\{\left\vert
Z_{p,t,s}(n_{k},t)\right\vert \leq \kappa (\log N)^{10}n_{k,\alpha },%
\mbox{ for all
}p\in \lbrack q],s\in \lbrack m(\widetilde{\lambda }_{j})]-1\}\,,
\end{eqnarray*}%
}and $\kappa $ is a small positive constant independent of $N$. By
calculating blockwisely the correlation and using Lemma \ref%
{lem-conditional-gaussian}, it is easy to verify that: (1) Conditioned on $%
\mathcal{F}_{\mathbf{T}}$, the conditional correlation matrix of $\{\mathbb{T%
}_{0,n_{k+1},0},...,\mathbb{T}_{0,n_{k+1},m-1}\}$ converges to $\mathsf{%
\Sigma }_{m}$ as $k\rightarrow \infty $. (2) Conditioned on any $\omega \in 
\mathcal{F}_{\mathbf{T}}$ such that $\omega \in \Omega _{k,\mathbb{T}}\cap
\Omega _{k,\mathbb{Z}}$, there exists $\,C<\infty $ independent of $N$ such
that $\mathbb{E}[\mathbb{T}_{0,r,s}\mid \omega ]\geq -CN^{-10-m^{\prime }}$
and $\Corr(\mathbb{T}_{0,r,s_{1}},\mathbb{T}_{0,n_{k+1},s_{2}})>0$ for any $%
r\in \lbrack n_{k}^{1-\beta }]$ and $s_{1},s_{2}\in \lbrack m]-1$.

Now instead of the condition of $\omega $, we further condition on $\Omega
_{k,\mathbf{0}}$ and thus get $\Omega _{k}$. First we consider $\mathbb{E}[%
\mathbb{T}_{0,r,s}\mid \Omega _{k}]$. Recall our choice of $a_{i}$ and $%
\varepsilon $. With (1) and (2) above and (\ref{eq-conditional-gaussian-mean}%
), regarding $U_{1}=\mathbb{T}_{0,r,s}\mid \omega $ and $Y_{i}=\mathbb{T}%
_{0,n_{k+1},i-1}\mid \omega $ for $i\in \lbrack m]$, we can notice that the
second term in the right hand side of (\ref{eq-conditional-gaussian-mean})
is positive, thus for some $C<\infty $ independent of $N$%
\begin{equation}
\mathbb{E}[\mathbb{T}_{0,r,s}\mid \Omega _{k}]\geq -CN^{-10-m^{\prime }}.
\label{eq-new-bound-E-T-condition}
\end{equation}%
Writing $r_{t}:=\max \{\ell :n_{k}+rn_{k}^{\beta }\leq t\}$. With similar
method to above analysis, for any $n_{k}+r_{t}n_{k}^{\beta }\leq t_{0}$ $%
<n_{k}+(r_{t}+1)n_{k}^{\beta }$, by first calculating the joint distribution
of $\{\xi _{t_{0}},\mathbb{T}_{0,n_{k+1},0},...,\mathbb{T}_{0,n_{k+1},m-1}\}$
conditioned on $\omega \in \Omega _{k,\mathbb{T}}\cap \Omega _{k,\mathbb{Z}}$%
, and then condition on $\Omega _{k,0}$, we can verify that there exists $%
C<\infty $ independent of $N$ such that for any $n_{k}\leq t_{0}<n_{k+1}$,%
\begin{equation}
\left\vert \mathbb{E}[\xi _{t_{0}}\mid \Omega _{k,0}]\right\vert \leq
Cn_{k}^{\alpha -1/2}(\log N)^{10}.  \label{eq-new-bound-ks-t0}
\end{equation}%
From (\ref{eq-new-bound-E-T-condition}) and (\ref{eq-new-bound-ks-t0}), with
(\ref{eq-T-j-n+s-k}) it is easy to check that for some $C_{1}<\infty $
independent of $N$,%
\begin{equation*}
{\mathbb{E}}[\mathbb{T}_{0,t,s}(n_{k},t)\mid \Omega _{k}]\geq
-C_{1}n_{k}^{\beta (s+1)+\alpha -1/2}(\log N)^{10}.
\end{equation*}%
Combined with the fact that $\beta <1$ and $\var(\mathbb{T}%
_{0,t,s}(n_{k},t)\mid \Omega _{k})=O(n_{k}^{2s+1})$ (since conditional
variance is always smaller than variance), we have%
\begin{equation}
\sum_{t=n_{k}}^{n_{k+1}}\mathbb{P}(\mathbb{T}_{0,t,s}(n_{k},t)\leq -\kappa
(\log N)^{10}n_{k,\alpha }\mid \Omega _{k})=O(1/N)\,.  \label{eq-S-Omega-k}
\end{equation}%
For $s\in \lbrack m_{j}]-1$, By the definition of $\Omega _{k}$ we see that
on $\Omega _{k}$%
\begin{equation*}
\left\vert \mathbb{T}_{j,t,s}(n_{k},t)\right\vert \leq 1+\left\vert \mathbb{T%
}_{j,t,s}(n_{k}+r_{t}n_{k}^{\beta },t)\right\vert \,.
\end{equation*}%
It is easy to check that $\Corr(\mathbb{T}_{j,t,s}(n_{k}+r_{t}n_{k}^{\beta
},t),\mathbb{T}_{0,n_{k+1},s_{1}}(n_{k},n_{k+1}))=O(n_{k}^{-1/2})$ $\forall $
$s_{1}\in \lbrack m]-1$. With Lemma \ref{lem-conditional-gaussian}, we can
further notice that $\forall \omega \in \mathcal{F}_{\mathbf{T}}$ such that $%
\omega \in $ $\Omega _{k,\mathbb{T}}\cap \Omega _{k,\mathbb{Z}}$, $\Corr(%
\mathbb{T}_{j,t,s}(n_{k}+r_{t}n_{k}^{\beta },t),\mathbb{T}%
_{0,n_{k+1},s_{1}}(n_{k},n_{k+1})\mid \omega )=O(n_{k}^{-1/2})$. Therefore
applying Lemma \ref{lem-conditional-gaussian} again with $U_{1}=$ $\mathbb{T}%
_{j,t,s}(n_{k}+r_{t}n_{k}^{\beta },t)\mid \omega $, $V_{i}=\mathbb{T}%
_{0,n_{k+1},i-1}(n_{k},n_{k+1})\mid \omega $ for $i\in \lbrack m]$, we see
that for some $C_{3}<\infty $ independent of $N$%
\begin{equation}
|{\mathbb{E}}[\mathbb{T}_{j,t,s}(n_{k}+r_{t}n_{k}^{\beta },t)\mid \Omega
_{k}]|\leq C_{3}(\log N)^{10}n_{k}^{\beta (s+1/2)+\alpha -1/2}=o((\log
N)^{10}n_{k,\alpha }),  \label{eq-e-T-j-t-s-rt}
\end{equation}%
where the last step is due to $\beta (m^{\prime }-1/2)<m$. Note that $\var(%
\mathbb{T}_{j,t,s}(n_{k}+r_{t}n_{k}^{\beta },t)\mid \Omega
_{k})=O(n_{k}^{\beta (2m^{\prime }-1)})=O(n_{k,\alpha })$ for all $s\in
\lbrack m^{\prime }]-1$, since $\beta (2s+1)\leq m+\alpha -1/2$. Thus for
any $\kappa >0$ there exists a constant $K_{0}=K_{0}(\kappa )$ such that for
all $k\geq K_{0}$,%
\begin{equation*}
\sum_{t=n_{k}}^{n_{k+1}}\mathbb{P}(|\mathbb{T}_{j,t,s}(n_{k},t)|\geq \kappa
(\log N)^{10}n_{k,\alpha }\mid \Omega _{k})=O(1/N).
\end{equation*}%
Similarly we can show that%
\begin{equation*}
\sum_{t=n_{k}}^{n_{k+1}}\mathbb{P}(|Z_{p,t,s}(n_{k},t)|\geq \kappa (\log
N)^{10}n_{k,\alpha }\mid \Omega _{k})=O(1/N).
\end{equation*}%
Combined with Lemma~\ref{lem-Omega-k-stretched} and \eqref{eq-S-Omega-k}, it
follows that for a constant $C<\infty $ independent of $N$ and all $k\geq
K_{0}$,%
\begin{equation}
\mathbb{P}(\Omega _{k}^{\star })\geq \mathbb{P}(\Omega _{k}^{\star }\mid
\Omega _{k})\mathbb{P}(\Omega _{k})\geq \mathrm{e}^{-C(\log N\cdot
n_{k}^{1-\beta }+(\log N)^{20}n_{k}^{2\alpha })}\,.  \label{eq-Omega-k-star}
\end{equation}%
Now let $\Omega ^{\star }=\cap _{k=1}^{\log _{2}N}\Omega _{k}^{\star }$.
Under $\Omega ^{\star }$, $\forall t\leq N,$ we can find $k,r$ such that $%
n_{k}+rn_{k}^{\beta }\leq t<n_{k}+(r+1)n_{k}^{\beta }$ where $r\leq
n_{k}^{1-\beta }-1$. Breaking $X_{t}$ into $X_{t}(0,n_{k}+rn_{k}^{\beta })$
and $X_{t}(n_{k}+rn_{k}^{\beta },t)$, then by the definition of $\Omega
^{\star }$, with (\ref{lower-bound-P-T-j-r-s}) it is direct to check $\Omega
^{\star }\subseteq \Omega _{N}^{+}$ when $\kappa $ is a sufficiently small
constant. Using \eqref{eq-Omega-k-star} and independence, we deduce that $%
\mathbb{P}(\Omega ^{\star })\geq \mathrm{e}^{-N^{1-\beta +o(1)}+N^{2\alpha
+o(1)}}$, and it completes the proof of the lower bound.

\section{Dominating unstable oscillatory mode: Theorem \protect\ref{thm-1}
(d)}

We first establish an important lemma below, and then give the proofs for
the lower bound and the upper bound in Subsection \ref%
{subsection-d-upper-bound} and \ref{subsection-d-lower-bound}\ respectively.

Let $\widehat{\Lambda }=\{\lambda /r^{\star },\lambda \in \Lambda \}$, $%
\Lambda _{1}=\widehat{\Lambda }\cap \partial \mathbb{D}$ and $\Lambda _{2}=%
\widehat{\Lambda }\backslash \partial \mathbb{D}$. Denote by $\widehat{Q}%
(z),Q_{1}(z),Q_{2}(z)$ monic polynomials with zero set $\widehat{\Lambda }$, 
$\Lambda _{1}$, $\Lambda _{2}$ respectively. In addition, write $\widehat{X}%
_{n}=X_{n}(r^{\star })^{-n}$, $\widehat{\xi }_{n}=\xi _{n}(r^{\star })^{-n}$
and let $(\zeta _{n})$ be an auto-regressive process generated by $Q_{2}$
and $(\widehat{\xi }_{n})$. Since all the zeroes of $Q_{2}$ has norm less
than 1, there exist $C^{\prime }<\infty $ and $r>1$ such that%
\begin{equation}
\var\zeta _{n}\leq C^{\prime }(r)^{-n}\mbox{ for all }n\in \mathbb{N}\,.
\label{eq-variance-zeta}
\end{equation}%
Due to $\widehat{Q}(z)=Q_{1}(z)Q_{2}(z)$, we see that $(\widehat{X}_{n})$ is
an auto-regressive process generated by $Q_{1}(z)$ and $(\zeta _{n})$.
Obviously the persistence of the process $(X_{n})$ is equivalent to the
persistence of $(\widehat{X}_{n})$, thus in what follows we will only
consider the process $(\widehat{X}_{n})$, and for convenience we drop the
\textquotedblleft hat\textquotedblright\ in the notation. That is to say, in
the rest of this subsection we assume $(X_{n})$ is an auto-regressive
process generated by $Q_{1}(z)$ and Gaussian sequence $(\zeta _{n})$, where $%
\Lambda \subseteq \partial \mathbb{D}$ and $\var\zeta _{n}\leq C^{\prime
}r^{-n}$ for a certain $r>1$.

Assume that the degree for $Q_{1}(z),Q_{2}(z)$ and $Q(z)$ are $L_{1},L_{2},L$
respectively. Denote by $\mathbf{Y}_{n}=[X_{n-L_{1}+1},\ldots ,X_{n}]^{T}$
for $n\geq L_{1}$, and denote by $\mathsf{\Sigma }_{n}$ the covariance
matrix of $\mathbf{Y}_{n}$. For a degree-$\ell $ polynomial $%
P(z)=\sum_{i=1}^{\ell }a_{i}z^{\ell -i}$, define an $\ell \times \ell $
matrix $\mathsf{A}=\mathsf{A}(P)$ by%
\begin{equation}
\mathsf{A}_{1,j}=a_{j}\mbox{ for }1\leq j\leq \ell ;\quad \mathsf{A}%
_{i,i-1}=1\mbox{
for }2\leq i\leq \ell ;\quad \mathsf{A}_{i,j}=0\mbox{ otherwise}.
\label{eq-def-A}
\end{equation}%
The following lemma provides estimates on $\lambda _{\mathrm{min}}(\mathsf{%
\Sigma }_{n}) $ and $\lambda _{\mathrm{max}}(\mathsf{\Sigma }_{n})$, the
minimum and maximum eigenvalues for $\mathsf{\Sigma }_{n}$.

\begin{lemma}
\label{lem-eigenvalues} Write $m^{\star }=$ $\max \{m(\lambda ):\lambda \in
\Lambda _{1}\}$. There exist constants $C,c>0$ such that%
\begin{equation*}
cn^{2m^{\star }-2}\leq \lambda _{\mathrm{min}}(\mathsf{\Sigma }_{n})\leq
\lambda _{\mathrm{max}}(\mathsf{\Sigma }_{n})\leq Cn^{2L}\mbox{ for all }%
n\in \mathbb{N}\,,
\end{equation*}
\end{lemma}

\begin{proof}
Let $\mathsf{A}_{(1)}:=\mathsf{A}(Q_{1})$ as in \eqref{eq-def-A} and note
that%
\begin{equation}
\mathbf{Y}_{n}=\mathsf{A}_{(1)}^{n-L_{1}}\mathbf{Y}_{L_{1}}+%
\sum_{j=L_{1}+1}^{n}\mathsf{A}_{(1)}^{n-j}\zeta _{j}\mathbf{e}_{L_{1}},
\label{eq-Yn-iteration}
\end{equation}%
where $\mathbf{e}_{L_{1}}\in \mathbb{R}^{L_{1}}$ is a vector with a unique
nonzero entry in the first coordinate (whose value is 1). Note that the set
of eigenvalues of $\mathsf{A}_{(1)}$ is exactly $\Lambda _{1}$. Thus by %
\eqref{eq-variance-zeta} and (\ref{eq-Yn-iteration}), we get that for any $%
0\leq i,j\leq L_{1},$ there is a constant $C^{\prime }<\infty $ such that%
\begin{equation*}
\left\vert \left( \mathsf{\Sigma }_{n}\right) _{(i,j)}\right\vert
=\left\vert \text{Cov}(X_{n-i},X_{n-j})\right\vert \leq C^{\prime }n^{2L},
\end{equation*}%
where in the last inequality we use the fact that $\Vert \mathsf{A}%
_{(1)}^{k}\Vert _{\infty }=O(1)k^{L_{1}}$ due to Lemma~\ref%
{lem-Jordan-matrix}. This gives the upper bound $\lambda _{\mathrm{max}%
}\left( \mathsf{\Sigma }_{n}\right) \leq Cn^{2L}$.

For the lower bound, we just need to show that there exists $c>0$ such that
for any $\nu =(\nu _{1},...,\nu _{L_{1}})$ with $\left\Vert \nu \right\Vert
_{2}=1$,%
\begin{equation*}
\var(\sum_{j=1}^{L_{1}}\nu _{j}X_{n+1-j})\geq cn^{2m^{\star }-2}.
\end{equation*}%
Notice that it is enough to prove it for $n$ large enough. Since $%
X_{n}=\sum_{j=1}^{n}b_{n,j}\xi _{j}$, we have%
\begin{equation*}
\sum_{j=1}^{L_{1}}\nu _{j}X_{n+1-j}=\sum_{j=1}^{L_{1}}\nu
_{j}\sum_{i=1}^{n+1-j}b_{n+1-j,i}\xi _{i}\text{,}
\end{equation*}%
which implies that for any fixed $K$, when $n$ is large enough we have%
\begin{equation}
\var(\sum_{j=1}^{L_{1}}\nu _{j}X_{n+1-j})\geq \var(\sum_{i=1}^{K}(%
\sum_{j=1}^{L_{1}}\nu _{j}b_{n+1-j,i})\xi _{i})\geq
r^{-K}\sum_{i=1}^{K}(\sum_{j=1}^{L_{1}}\nu _{j}b_{n+1-j,i})^{2}.  \notag
\end{equation}%
If we denote $\Lambda _{3}:=\{\lambda \in \Lambda _{1},m(\lambda )=m^{\star
}\}$, then%
\begin{eqnarray*}
b_{n+1-j,i} &=&\sum_{\lambda \in \Lambda _{3}}a_{\lambda ,m^{\star
}-1}\left( n+1-i-j\right) ^{m^{\star }-1}\cos (\left( n+1-i-j\right) \theta
_{\lambda }+\theta _{\lambda ,i})+o(n^{m^{\star }-1}) \\
&=&n^{m^{\star }-1}\sum_{\lambda \in \Lambda _{3}}a_{\lambda ,m^{\star
}-1}\cos (\left( n+1-i-j\right) \theta _{\lambda }+\theta _{\lambda
,i})+o(n^{m^{\star }-1}),
\end{eqnarray*}%
which implies that%
\begin{equation*}
\sum_{j=1}^{L_{1}}\nu _{j}b_{n+1-j,i}=n^{m^{\star
}-1}\sum_{j=1}^{L_{1}}\sum_{\lambda \in \Lambda _{3}}\nu _{j}a_{\lambda
,m^{\star }-1}\cos (\left( n+1-i-j\right) \theta _{\lambda }+\theta
_{\lambda ,i})+o(n^{m^{\star }-1}).
\end{equation*}%
Now, applying Lemma \ref{lem-rotation-stretched-exponential}, we can see
that there exists $K_{1}>0$ such that for any $n$, there exists $1\leq i\leq
K_{1}$ such that%
\begin{equation}
|\sum_{j=1}^{L_{1}}\sum_{\lambda \in \Lambda _{3}}\nu _{j}a_{\lambda
,m^{\star }-1}\cos (\left( n+1-i-j\right) \theta _{\lambda }+\theta
_{\lambda ,i})|\geq \frac{1}{4}\max_{1\leq j\leq L_{1},\lambda \in \Lambda
_{3}}|\nu _{j}a_{\lambda ,m^{\star }-1}|.  \label{eigen_lower_1}
\end{equation}%
Since $\left\vert \left\vert \nu \right\vert \right\vert _{2}=1$ implies $%
\max_{j=1}^{L}\left\{ \left\vert \nu _{j}\right\vert \right\} \geq 1/(4\sqrt{%
L_{1}})$, setting $K=K_{1}$ we conclude that%
\begin{eqnarray*}
\var(\sum_{j=1}^{L_{1}}\nu _{j}X_{n+1-j}) &\geq
&r^{-K_{1}}(\sum_{j=1}^{L_{1}}\nu _{j}b_{n+1-j,i})^{2} \\
&=&r^{-K_{1}}n^{2m^{\star }-2}\frac{1}{16L_{1}}\min_{\lambda \in \Lambda
_{3}}|a_{\lambda ,m^{\star }-1}|^{2}+o(n^{2m^{\star }-2}),
\end{eqnarray*}%
which gives the desired result.%
\end{proof}

\subsection{Upper bound\label{subsection-d-upper-bound}}

We now provide the proof of the upper bound on the persistence probability.

\begin{proof}[Proof of the upper bound: case (d)]
Continue to write $\mathsf{A}_{(1)}=\mathsf{A}(Q_{1})$. We note that since $%
Q_{1}(1)=0$, it follows that $\lambda =1$ is an eigenvalue of $\mathsf{A}%
_{(1)}$ with eigenvector $\mathbf{1}$. In light of \eqref{eq-variance-zeta},
we can choose $n^{\prime }:=\kappa \log N$ with $\kappa >0$ a large enough
constant such that $\var\zeta _{n}\leq N^{-13L}$ for all $n\geq n^{\prime }$%
, combined with (\ref{eq-Yn-iteration}) $\Corr\Vert \mathsf{A}%
_{(1)}^{k}\Vert _{\infty }=O(1)k^{L_{1}}\leq O(1)N^{L_{1}}$, we get%
\begin{equation*}
\var\left( \Vert \mathbf{Y}_{n}-\mathsf{A}_{(1)}^{n-n^{\prime }}\mathbf{Y}%
_{n^{\prime }}\Vert _{\infty }\right) \leq C^{\prime }N^{-10L},
\end{equation*}%
for some $C^{\prime }<\infty $. Therefore, we deduce that for a suitably
large constant $C<\infty $%
\begin{equation}
\mathbb{P}(\exists n^{\prime }\leq n\leq N:\Vert \mathbf{Y}_{n}-\mathsf{A}%
_{(1)}^{n-n^{\prime }}\mathbf{Y}_{n^{\prime }}\Vert _{\infty }\geq
N^{-4L})\leq C\mathrm{e}^{-N}\,,  \label{eq-remove-noise}
\end{equation}%
which suggests that the persistence of the process until time $N$ is mainly
determined by $\mathbf{Y}_{n^{\prime }}$. Let $\{\mathbf{v}_{\lambda
,j}:\lambda \in \Lambda _{1},j\in \lbrack m(\lambda )]\}$ be a basis (of
unit norm) for $\mathsf{A}_{(1)}$ as in the statement of Lemma~\ref%
{lem-Jordan-matrix}, then for $\mathbf{y}\in \mathbb{R}^{L_{1}}$, there
exists a unique $\{c_{\lambda ,r}(y):\lambda \in \Lambda _{1},r\in \lbrack
m(\lambda )]\}$ such that%
\begin{equation}
\mathbf{y}=\sum_{\lambda \in \Lambda _{1}}\sum_{r=1}^{m(\lambda )}c_{\lambda
,r}(\mathbf{y})\mathbf{v}_{\lambda ,r}\,.  \label{eq-expansion-y}
\end{equation}%
Define $\Pi \subseteq \mathbb{R}^{L_{1}}$ as%
\begin{equation}
\Pi :=\{\mathbf{y}:|c_{\lambda ,r}(\mathbf{y})|\leq (N^{m-r}(\log N)^{20L})%
\text{, }\mbox{ for
all }\lambda \in \Lambda _{1}\setminus \{1\},r\in \lbrack m(\lambda )]\,\}.
\label{eq-volume-small}
\end{equation}%
For $\mathbf{y}\in B_{(\log N)^{4L}}\setminus \Pi $ where $B_{(\log
N)^{4L}}\subseteq \mathbb{R}^{L_{1}}$ is a $L_{\infty }$ ball of radius $%
(\log N)^{4L}$, by the definition (\ref{eq-volume-small}) there exists $%
(\lambda ^{\star },r^{\star })$ such that $|c_{\lambda ^{\star },r^{\star }}(%
\mathbf{y})|\geq (N^{m-r^{\star }}(\log N)^{20L})$. Then by Lemma \ref%
{lem-polynomial-approximation}, there exists a constant $c>0$ and $N^{\star
}\in \{N/L,2N/L,\ldots ,N\}$ such that%
\begin{equation}
|\sum_{j=0}^{m(\lambda ^{\star })-1}\frac{(N^{\star })^{j}}{j!}(\lambda
^{\star })^{-j}c_{\lambda ^{\star },j+1}(\mathbf{y})|\geq cN^{m-1}(\log
N)^{20L}\,.  \label{eq-coefficients-large}
\end{equation}%
Denoting by $\mathbf{1}\in \mathbb{R}^{L_{1}}$ where each coordinate takes
value 1. Noting that $\left\Vert \mathbf{v}_{\lambda ,r}\right\Vert _{\infty
}\leq \left\Vert \mathbf{v}_{\lambda ,r}\right\Vert _{2}=1$ and $c_{\lambda
,r}\leq (\log N)^{4L}$, by Lemma \ref{lem-Jordan-matrix} we have that for
all $N^{\star }-\log N\leq N^{\prime }\leq N^{\star }$,%
\begin{equation}
\mathsf{A}_{(1)}^{N^{\prime }}\mathbf{y}\leq (\log N)^{5L}N^{m-1}\mathbf{1}%
+(1+O(\log N/N))\sum_{\lambda \in \Lambda _{1}\setminus
\{1\}}\sum_{r=1}^{m(\lambda )}\sum_{j=0}^{m(\lambda )-r}\frac{N^{\star j}}{j!%
}\lambda ^{N^{\prime }-j}c_{\lambda ,r+j}(\mathbf{y})\mathbf{v}_{\lambda
,r}\,,  \label{bound-A-1-N-prime-y}
\end{equation}%
where the "$\leq $" means entry-wise less than or equal to. For each $%
\lambda \in \Lambda _{1}\setminus \{1\}$ we can write $\lambda ^{N^{\prime
}-j}=\cos ((N^{\prime }-j)\theta _{\lambda })+i\sin ((N^{\prime }-j)\theta
_{\lambda })$, recalling that $\bar{c}_{\lambda ,j}=c_{\bar{\lambda},j}$,
applying Lemma~\ref{lem-rotation-stretched-exponential} and using %
\eqref{eq-coefficients-large}, with (\ref{bound-A-1-N-prime-y}) we deduce
that for large enough $N$ there exists a $N^{\star }\in (N^{\star }-\log
N,N^{\star })$ such that%
\begin{equation}
\min_{j\in \lbrack L_{1}]}(\mathsf{A}_{(1)}^{N^{\star }}\mathbf{y})_{j}\leq -%
\frac{c}{8}N^{m-1}(\log N)^{20L}\,.  \label{bound-min-A-1-N-y}
\end{equation}%
From the fact that $\var X_{n}=O(n^{2L})$ and Gaussian estimate, we have $%
\mathbb{P}(\mathbf{Y}_{n^{\prime }}\not\in B_{(\log N)^{4L}})\leq C^{\prime }%
\mathrm{e}^{-(\log N)^{2}/2}$ for some $C^{\prime }<\infty $, thus combined
with (\ref{bound-min-A-1-N-y}) and \eqref{eq-remove-noise}, it follows that
there exists $C<\infty $ such that%
\begin{align}
p_{N}& \leq C\mathrm{e}^{-N}+\mathbb{P}(\mathbf{Y}_{n^{\prime }}\not\in
B_{(\log N)^{4L}})+\mathbb{P}(\mathbf{Y}_{n^{\prime }}\in B_{(\log
N)^{4L}}\bigcap \Pi )  \notag  \label{eq-upper-bound-prelim-poly-2} \\
& \leq C\mathrm{e}^{-(\log N)^{2}}+\mathbb{P}(\mathbf{Y}_{n^{\prime }}\in
B_{(\log N)^{4L}}\bigcap \Pi )\,.
\end{align}%
It remains to bound the last term on the right hand side. Using %
\eqref{eq-volume-small}, we deduce that%
\begin{equation*}
\mathrm{vol}(B_{(\log N)^{4L}}\bigcap \Pi )=O((\log N)^{4L^{2}})N^{-\alpha
}\,,
\end{equation*}%
where $\alpha =\sum_{\lambda \in \Lambda ^{\star }}(m(\lambda
)-m)_{+}(m(\lambda )-m+1)_{+}/2$. Combined with Lemma~\ref{lem-eigenvalues},
by writing the probability as an integral against Gaussian density, it
follows that%
\begin{equation*}
\mathbb{P}(\mathbf{Y}_{n^{\prime }}\in B_{(\log N)^{4L}}\bigcap \Pi )\leq
(\log N)^{8L^{2}}N^{-\alpha }\,.
\end{equation*}%
Plugging the preceding inequality into \eqref{eq-upper-bound-prelim-poly-2}
completes the proof on the upper bound.
\end{proof}

\subsection{Lower bound\label{subsection-d-lower-bound}}

We next turn to the proof of the lower bound. For $C_{0}\geq 1$ and $%
0<\delta <1/r^{\star }$ to be selected, define $N_{1}:=C_{0}\log \log N$ and 
$N_{2}:=C_{0}\log N$, and the event%
\begin{equation*}
\Omega _{N_{1},\delta }:=\bigcap_{n=1}^{N_{1}}\{|X_{n}-1|\leq \delta
^{n}\}\,.
\end{equation*}%
Recall that $1$ is an eigenvalue of $\mathsf{A}(Q)$ of eigenvector $\mathbf{1%
}$. With the fact that $\var\xi _{n}=(r^{\star })^{-n}$ it is easy to verify
that for some $C_{1},C_{2}>0$ we have%
\begin{equation*}
\mathbb{P}(|X_{n}-1|\leq \delta ^{n}\mid
\bigcap\limits_{k=n-L}^{n-1}|X_{k}-1|\leq \delta ^{k})\geq C_{1}(r^{\star
}\delta )^{\frac{n}{2}}e^{-C_{2}(r^{\star }\delta )^{n}},
\end{equation*}%
which implies that for $N$ large enough we have%
\begin{equation}
\mathbb{P}(\Omega _{N_{1},\delta })\geq (\delta /4)^{N_{1}^{2}}\,.
\label{eq-N-1-delta}
\end{equation}%
In addition, we denote by $\mathsf{\Sigma }_{N_{2}}^{\star }$ the
conditional covariance matrix of $\mathbf{Y}_{N_{2}}$ given the $\sigma $%
-field $\mathcal{F}_{N_{1}}$ generated by $\{\xi _{1},\ldots ,\xi _{N_{1}}\}$%
. Denote by $\widetilde{X}_{n}:=X_{n+N_{1}}-E[X_{n+N_{1}}\mid \mathcal{F}%
_{N_{1}}]$ and $\widetilde{\xi }_{n}=\xi _{n+N_{1}}$. Then $\widetilde{X}_{n}
$ is a regressive process generated by $Q$ and $\{\widetilde{\xi }_{n}\}$.
In addition by (\ref{eq-variance-zeta}) we have Var$(\widetilde{\xi }%
_{n})=(r^{\star })^{-(n+N_{1})}$ with $r^{\star }>1$. Revoking the proof of
Lemma~\ref{lem-eigenvalues} and using the aforementioned information, it is
easy to see that we can keep the same upper bound of $\lambda _{\mathrm{max}%
}(\mathsf{\Sigma }_{N_{2}}^{\star })$; and the only change for the lower
bound of $\lambda _{\mathrm{min}}(\mathsf{\Sigma }_{N_{2}}^{\star })$ is
caused by the different order of Var$(\widetilde{\xi }_{n})$. Thus by the
last equation in the proof of Lemma~\ref{lem-eigenvalues} we get $\lambda _{%
\mathrm{min}}(\mathsf{\Sigma }_{N_{2}}^{\star })\geq c(r^{\star })^{-\frac{%
N_{1}}{2}}(N_{2}-N_{1})^{2m^{\star }-2}\geq N_{2}^{2m^{\star }-2-c^{\prime }}
$ for some $c^{\prime }<\infty $, and consequently for some $C<\infty $ we
have%
\begin{equation}
N_{2}^{2m^{\star }-2-c^{\prime }}\leq \lambda _{\mathrm{min}}(\mathsf{\Sigma 
}_{N_{2}}^{\star })\leq \lambda _{\mathrm{max}}(\mathsf{\Sigma }%
_{N_{2}}^{\star })\leq CN_{2}^{2L}\,.  \label{eq-conditional-eigenvalues}
\end{equation}%
\abbr{WLOG} we assume $\gamma _{0}:=m^{\star }-1-c^{\prime
}/2\leq -10L$. We define $\Pi _{N_{2}}\subseteq \mathbb{R}^{L_{1}}$ as%
{\small 
\begin{align}
\Pi _{N_{2}}& :=\{\mathbf{y}:|c_{\lambda ,r}(\mathbf{y})|\leq N_{2}^{\gamma
_{0}-1}(N^{m-r}\wedge 1),\mbox{ for
all }\lambda \in \Lambda _{1}\setminus \{1\},r\in \lbrack m(\lambda )]\,\} 
\notag \\
\cap & \{\mathbf{y}:c_{1,1}(\mathbf{y})\in (1-N_{2}^{\gamma
_{0}-1}/3,1+N_{2}^{\gamma _{0}-1}/3)\text{, }c_{1,j}(\mathbf{y})\in
N_{2}^{\gamma _{0}-1}(2L,3L)\mbox{ for }2\leq j\leq m\,\},
\label{eq-volum-small-lbd}
\end{align}%
}where we use the expansion of \textbf{$y$} as in \eqref{eq-expansion-y} and 
$c_{1,1}$ is the coefficient of $\mathbf{1}$ for \textbf{$y$} in the
expansion. Let $\alpha =\sum_{\lambda \in \Lambda ^{\star }}(m(\lambda
)-m)_{+}(m(\lambda )-m+1)_{+}/2$. It is clear that%
\begin{equation}
\mathrm{vol}(\Pi _{N_{2}})\geq N^{-\alpha +o(1)}\,.  \label{eq-volume-Pi-2}
\end{equation}%
We are now ready to provide

\begin{proof}[Proof of the lower bound: Case (d)]
Using the same method of \eqref{eq-remove-noise}, with the definition of $%
N_{1},N_{2}$ we see that there exists $C<\infty $ such that%
\begin{equation}
\begin{aligned} \mathbb{P}(\exists n \in [N_1, N_2]: \|\mathbf{Y}_n -
A_{(1)}^{n - N_1} \mathbf{Y}_{N_1}\|_\infty \geq N_2^{-4L}) &\leq C
\mathrm{e}^{-(\log N)^2}\,,\\ \mathbb{P}(\exists n \in [N_2 , N]:
\|\mathbf{Y}_n - A_{(1)}^{n - N_2} \mathbf{Y}_{N_2}\|_\infty \geq
N^{-4L})&\leq C\mathrm{e}^{-N} \,.\end{aligned}  \label{two-Y-bounds}
\end{equation}%
Expanding $Y_{N_{1}}-\mathbf{1}$ in the basis of eigenvectors $\left\{
v_{\lambda ,r}\right\} $ of $\mathsf{A}_{(1)}$ as%
\begin{equation*}
\mathbf{Y}_{N_{1}}-\mathbf{1}=\sum_{\lambda \in \Lambda
_{1}}\sum_{r=1}^{m(\lambda )}\widetilde{c}_{\lambda ,r}\mathbf{v}_{\lambda
,r},
\end{equation*}%
we have that under $\Omega _{N_{1},\delta }$,%
\begin{equation*}
\sum_{\lambda \in \Lambda _{1}}\sum_{r=1}^{m(\lambda )}\widetilde{c}%
_{\lambda ,r}^{2}=\left\Vert \mathbf{Y}_{N_{1}}-\mathbf{1}\right\Vert
_{2}\leq L_{1}\delta ^{2N_{1}-2L_{1}},
\end{equation*}%
consequently for any $\lambda \in \Lambda _{1}$ and $r$%
\begin{equation}
\left\vert \widetilde{c}_{\lambda ,r}\right\vert \leq \sqrt{L_{1}}\delta
^{N_{1}-L_{1}}.  \label{bound-c-lambda-r}
\end{equation}%
For any $N_{1}<n\leq N_{2}$, in light of Lemma \ref{lem-Jordan-matrix} we
have that%
\begin{equation}
\mathsf{A}_{(1)}^{n-N_{1}}\mathbf{Y}_{N_{1}}-\mathbf{1}=\sum_{\lambda \in
\Lambda _{1}}\sum_{r=1}^{m(\lambda )}\sum_{j=0}^{m(\lambda )-r}\binom{n-N_{1}%
}{j}\lambda ^{n-N_{1}-j}\widetilde{c}_{\lambda ,r+j}\mathbf{v}_{\lambda ,r}.
\label{eq-expansion-A-1-Y-N-1-1}
\end{equation}%
Hence, with the fact that $\left\vert \lambda \right\vert \leq 1$ and (\ref%
{bound-c-lambda-r}) we see that%
\begin{equation}
\left\Vert \mathsf{A}_{(1)}^{n-N_{1}}\mathbf{Y}_{N_{1}}-\mathbf{1}%
\right\Vert _{\infty }\leq L_{1}\max_{j\leq L_{1}}\binom{n-N_{1}}{j}%
\left\vert \widetilde{c}_{\lambda ,r+j}\right\vert \left\Vert \mathbf{v}%
_{\lambda ,r}\right\Vert _{\infty }\leq L_{1}^{\frac{3}{2}%
}N_{2}^{L_{1}}\delta ^{N_{1}-L_{1}}.  \label{bound-A-1-Y-N-1-1}
\end{equation}%
At this point, we choose $\delta =e^{m^{\star }-L_{1}-2-c^{\prime }/2}\wedge
e^{-2}$ and consequently%
\begin{equation}
L_{1}^{\frac{3}{2}}N_{2}^{L_{1}}\delta ^{N_{1}-L_{1}}<(N_{2}^{\gamma
_{0}}\wedge N_{2}^{-1})/3\text{ as }N\rightarrow \infty .\,
\label{delta_bound}
\end{equation}%
Hence when $N$ is large enough, under $\Omega _{N_{1},\delta }$, for any $%
N_{1}<n\leq N_{2}$ we have $||\mathsf{A}_{(1)}^{n-N_{1}}\mathbf{Y}_{N_{1}}-%
\mathbf{1||}_{\infty }<1/2$, and consequently $\mathsf{A}_{(1)}^{n-N_{1}}%
\mathbf{Y}_{N_{1}}\geq 1/2\times \mathbf{1}$. Similarly, under $\Pi _{N_{2}}$
of (\ref{eq-volum-small-lbd}),\ for any $N_{2}<n\leq N$ we have\ $\mathsf{A}%
_{(1)}^{n-N_{2}}\mathbf{Y}_{N_{2}}\geq 1/2\times \mathbf{1}$. Therefore, we
deduce from (\ref{two-Y-bounds}) that for $N$ large enough%
\begin{equation}
p_{N}\geq \mathbb{P}(\Omega _{N_{1},\delta }\bigcap \{\mathbf{Y}_{N_{2}}\in
\Pi _{N_{2}}\})-O(1)\mathrm{e}^{-(\log N)^{2}}\,.  \label{eq-lbd-case4b}
\end{equation}%
From (\ref{eq-conditional-eigenvalues}), (\ref{eq-volum-small-lbd}), (\ref%
{bound-A-1-Y-N-1-1}) and (\ref{delta_bound}), we observe that for any $\Vert 
$\textbf{$y$}$_{N_{1}}-\mathbf{1}\Vert _{\infty }\leq \delta ^{N_{1}}$ and 
\textbf{$y$}$_{N_{2}}\in \Pi _{N_{2}}$, as $N$ is large enough%
\begin{equation*}
\Vert \mathbf{y}_{N_{2}}-\mathsf{A}_{(1)}^{N_{2}-N_{1}}\mathbf{y}%
_{N_{1}}\Vert _{\infty }\leq \Vert \mathbf{y}_{N_{2}}-\mathbf{1}\Vert
_{\infty }+\Vert \mathsf{A}_{(1)}^{N_{2}-N_{1}}\mathbf{y}_{N_{1}}-\mathbf{1}%
\Vert _{\infty }<N_{2}^{\gamma _{0}}\leq \sqrt{\lambda _{\mathrm{min}}(%
\mathsf{\Sigma }_{N_{2}}^{\star })}.
\end{equation*}%
Hence, there exists $C>0$ such that the conditional density of $\mathbf{Y}%
_{N_{2}}$ restricted to $\Pi _{N_{2}}$ given \textbf{$y$}$_{N_{1}}$with in $%
\Omega _{N_{1},\delta }$ satisfies%
\begin{eqnarray*}
&&p_{\mathbf{Y}_{N_{2}}\mid \mathbf{Y}_{N_{1}}}(\mathbf{y}_{N_{2}}\mid 
\mathbf{y}_{N_{1}}) \\
&=&\left( 2\pi \right) ^{-\frac{L_{1}}{2}}\left( \det (\mathsf{\Sigma }%
_{N_{2}}^{\star })\right) ^{-\frac{1}{2}}\text{exp}\left( -\frac{1}{2}\left( 
\mathbf{y}_{N_{2}}-\mathsf{A}_{(1)}^{N_{2}-N_{1}}\mathbf{y}_{N_{1}}\right)
^{T}(\mathsf{\Sigma }_{N_{2}}^{\star })^{-1}\left( \mathbf{y}_{N_{2}}-%
\mathsf{A}_{(1)}^{N_{2}-N_{1}}\mathbf{y}_{N_{1}}\right) \right)  \\
&\geq &C(\lambda _{\mathrm{max}}(\mathsf{\Sigma }_{N_{2}}^{\star }))^{-\frac{%
L_{1}}{2}},
\end{eqnarray*}%
where the last inequality is due to%
\begin{equation*}
v^{\prime }\mathsf{\Sigma }^{-1}v\leq \frac{\left\vert \left\vert
v\right\vert \right\vert _{2}^{2}}{\lambda _{\min }(\mathsf{\Sigma })}\leq
L_{1}\frac{\left\vert \left\vert v\right\vert \right\vert _{\infty }^{2}}{%
\lambda _{\min }(\mathsf{\Sigma })}\leq L_{1}\text{.}
\end{equation*}%
Therefore, from \eqref{eq-N-1-delta}, \eqref{eq-conditional-eigenvalues} and %
\eqref{eq-volume-Pi-2} we see that for a constant $C>0$%
\begin{equation*}
\mathbb{P}(\Omega _{N_{1},\delta }\bigcap \{Y_{\mathbf{N}_{2}}\in \Pi
_{N_{2}}\})\geq \mathbb{P}(\Omega _{N_{1},\delta })\cdot C(\lambda _{\mathrm{%
max}}(\mathsf{\Sigma }_{N_{2}}^{\star }))^{-\frac{L_{1}}{2}}\mathrm{vol}(\Pi
_{N_{2}})\geq N^{-\alpha +o(1)}\,.
\end{equation*}%
Thus \eqref{eq-lbd-case4b} completes the proof.
\end{proof}

\section{Approximately \textsc{\lowercase{IRW}}: Theorem \protect\ref{thm-1}
(e)\label{section-app-int-rw}}

We first show the upper bound in Subsection \ref{subsection-e-upper-bound},
by comparing the process with an order $m$ \textsc{\lowercase{IRW}}. For the
lower bound, we prove it for two different cases: the case with dominant
zero-angle component, and the case with competing oscillatory components, in
Subsection \ref{subsection-e-lower-bound-1} and \ref%
{subsection-e-lower-bound-2} respectively.

\subsection{Upper bound\label{subsection-e-upper-bound}}

Next lemma states that the persistence probability for an \textsc{%
\lowercase{IRW}} of any order has polynomial decay, which we need later.

\begin{lemma}
\label{lem-integrated-positive} Let $\{Y_{n}:n=1,\ldots ,N\}$ be \textsc{%
\lowercase{AR}}$((z-1)^{m+1})$, and assume $K$ is a fixed positive integer.
Then there exist constants $c$,$C>0$ such that for any $N\geq K$,%
\begin{equation}
\mathbb{P}(Y_{n}\geq 0\mbox{ for all }K\leq n\leq N)\leq cN^{-C}.
\label{Jun_add_1}
\end{equation}
\end{lemma}

\begin{proof}
It is obvious that we just need to show (\ref{Jun_add_1}) for $N$ large
enough. We denote $n_{k}=2^{k}$, and notice that $\{Y_{n}\geq 0%
\mbox{ for
all }K\leq n\leq N\}$ implies $\{Y_{n_{k}}\geq 0\mbox{ for all }\lbrack \log
K]+1\leq k\leq \lbrack \log N]\}$. We calculate the correlation between $%
Y_{n_{i}}$ and $Y_{n_{j}}$ $(i<j)$ as follows,%
\begin{equation*}
\rho (Y_{n_{i}},Y_{n_{j}})=\rho (\sum\limits_{\ell =1}^{n_{i}}b_{n_{i},\ell
}\xi _{\ell },\sum\limits_{\ell =1}^{n_{j}}b_{n_{j},\ell }\xi _{\ell })=%
\frac{\sum_{\ell =1}^{n_{i}}b_{n_{i},\ell }b_{n_{j},\ell }}{\sqrt{\sum_{\ell
=1}^{n_{i}}b_{n_{i},\ell }^{2}}\sqrt{\sum_{\ell =1}^{n_{j}}b_{n_{j},\ell
}^{2}}}.
\end{equation*}%
Since $b_{n,\ell }=a_{1,m}(n-\ell )^{m}(1+o(n-\ell ))$, we get for large $%
i,j $ that{\small 
\begin{equation}
\rho (Y_{n_{i}},Y_{n_{j}})=\frac{\sum_{\ell =1}^{n_{i}}(n_{i}-\ell
)^{m}(n_{j}-\ell )^{m}}{\sqrt{\sum_{\ell =1}^{n_{i}}(n_{i}-\ell )^{2m}}\sqrt{%
\sum_{\ell =1}^{n_{j}}(n_{j}-\ell )^{2m}}}(1+o(1))\leq 2\left( \frac{n_{i}}{%
n_{j}}\right) ^{\frac{1}{2}}(1+o(1))=2^{1-\left\vert i-j\right\vert
/2}(1+o(1)).  \label{lem-intpos_eq1}
\end{equation}%
} From (\ref{lem-intpos_eq1}) we see that there exists $K_{0}>0$ such that $%
\rho (Y_{n_{i}},Y_{n_{j}})>0$ for all $i,j>K_{0}$, and $\sum_{i,j>K_{0}}^{k}%
\rho (Y_{n_{i}},Y_{n_{j}})\leq \Delta (k-K_{0})$ for some $\Delta <\infty $
and all $k\geq K_{0}$. Combined with Lemma~\ref%
{lem-almost-independence-exponential}, we see that there exist $c,C>0$ such
that%
\begin{equation*}
\mathbb{P}(Y_{n}\geq 0\mbox{ for all }K_{0}\leq n\leq N)\leq \mathbb{P}%
(Y_{n_{k}}\geq 0\mbox{ for all }K_{0}<k\leq \lbrack \log N])\leq ce^{-C\log
N}=cN^{-C}\text{.}
\end{equation*}%
Since $K_{0}$ is independent of $N$, the proof is completed.
\end{proof}

Now we give an upper bound on the persistence probability.

\begin{lemma}
\label{lem-poly-upper-bound}Suppose that $\Lambda \subset \mathbb{D}$ and $%
m(1)\geq m(\lambda )$ for all $\lambda \in \Lambda ^{\star }$. Then there
exist constants $c,C>0$ such that $p_{N}\leq cN^{-C}$ for all $N\in \mathbb{N%
}$.
\end{lemma}

\begin{proof}
By our assumption, the zero set $\Lambda $ for the generating polynomial $%
Q(z)$ satisfies that $\Lambda \subseteq \mathbb{D}$. Let $m=\max_{\lambda
\in \Lambda }m(\lambda )$, and let $(Y_{n})$ be \textsc{\lowercase{AR}}($%
(z-1)^{m+1}$). It is then straightforward to apply Lemma~\ref%
{lem-general-regression-coefficients} and verify that for a number $K>0$
(independent of $N$),%
\begin{equation*}
\rho (X_{n},X_{m})\leq \rho (Y_{n},Y_{m})\mbox{ for all }n,m\geq K\,.
\end{equation*}%
Combined with Slepian's Lemma, it follows that%
\begin{equation*}
p_{N}\leq \mathbb{P}(Y_{n}\geq 0\mbox{ for all }K\leq n\leq N)\,,
\end{equation*}%
which completes the proof together with Lemma~\ref{lem-integrated-positive}.%
\label{lem-reduce-to-BM copy(1)}
\end{proof}

\subsection{Lower bound with competing oscillatory components\label%
{subsection-e-lower-bound-2}}

Here we show the lower bound of first case in part (e), where $r^{\star }=1$
and $\exists \lambda \in \Lambda ^{\star }\setminus \{1\}\mbox{ such that }%
m(\lambda )=m(1)$. In the proof the key is Lemma \ref{lem-rotation}. Before
stating the lemma let's introduce some notations. Suppose that $\Lambda
^{\star }=\{\lambda _{0}|_{m(1)},\lambda _{1}|_{m_{1}},\ldots ,\lambda
_{\ell }|_{m_{\ell }}\}$ where $\lambda _{0}=1$. Write $\lambda _{j}=\mathrm{%
e}^{\sqrt{-1}\theta _{j}}$ for $j\in \lbrack \ell ]$. By Lemma~\ref%
{lem-general-regression-coefficients}, we could write%
\begin{equation}
X_{n}=\sum_{j=0}^{\ell }c_{j}T_{j,n}+R_{n}\,,  \label{eq-X-n-expansion}
\end{equation}%
where $c_{j}\neq 0$, $T_{j,n}:=\sum_{i=1}^{n}b_{n,i,m_{j}-1}\cos
((n-i)\theta _{j}+\theta _{\lambda _{j},m_{j}-1})\xi _{i}$ and $%
R_{n}=\sum_{i=1}^{n}O((n-i)^{m(1)-2})\xi _{i}$. Now, for $j\in \lbrack \ell
] $ we define%
\begin{equation}
T_{j,n}^{\prime }:=\sum_{i=1}^{n}b_{n,i,m_{j}-1}\sin ((n-i)\theta
_{j}+\theta _{\lambda _{j},m_{j}-1})\xi _{i}\mbox{ and }\mathbf{T}%
_{j,n}:=[T_{j,n},T_{j,n}^{\prime }]^{T}\,.  \label{eq-def-S}
\end{equation}%
In addition, for $k\geq 1$ we write $\mathbf{T}%
_{j,n,k}=[T_{j,n,k},T_{j,n,k}^{\prime }]^{T}$ where%
\begin{equation}
T_{j,n,k}=\sum_{i=1}^{n}b_{n,i,k}\cos ((n-i)\theta _{j}+\theta _{\lambda
_{j},m_{j}-1})\xi _{i}\,,\,T_{j,n,k}^{\prime }=\sum_{i=1}^{n}b_{n,i,k}\sin
((n-i)\theta _{j}+\theta _{\lambda _{j},m_{j}-1})\xi _{i}\,.
\label{eq-def-T-small-order}
\end{equation}%
The difference between (\ref{eq-def-T-small-order}) and (\ref{def-new-T}) is
that here we replace $\theta _{\lambda _{j},k}$ by $\theta _{\lambda
_{j},m_{j}-1}$. Due to (\ref{eq-b-iterative}) it is easy to see that for all 
$k\geq 2$%
\begin{equation}
\mathbf{T}_{j,n+1,k}=\mathsf{R}_{\theta _{j}}\mathbf{T}_{j,n,k}+\mathbf{T}%
_{j,n+1,k-1}\,,  \label{eq-R_theta-general}
\end{equation}%
by noting that $T_{j,n}=T_{j,n,m_{j}-1}$,$\,$especially%
\begin{equation}
\mathbf{T}_{j,n+1}=\mathsf{R}_{\theta _{j}}\mathbf{T}_{j,n}+\mathbf{T}%
_{j,n+1,m_{j}-2}\,\,.  \label{eq-R_theta}
\end{equation}%
Write $\mathbf{T}_{n}$ as a $2\ell $-dimensional vector such that $\mathbf{T}%
_{n}=[\mathbf{T}_{1,n},\ldots ,\mathbf{T}_{\ell ,n}]^{T}$, and in the
similar manner write $\mathbf{t}=[\mathbf{t}_{1},\ldots ,\mathbf{t}_{\ell
}]^{T}$ where $\mathbf{t}_{j}\in \mathbb{R}^{2}$ for each $j\in \lbrack \ell
]$. For any fixed $K\in \mathbb{N}$, define $\phi _{K}:\mathbb{R}^{2\ell
}\mapsto \mathbb{R}$ by%
\begin{equation}
\phi _{K}(\mathbf{t}):=\min_{0\leq i\leq K}[1,0]\sum_{j=1}^{\ell }c_{j}(%
\mathsf{R}_{\theta _{j}})^{i}\mathbf{t}_{j}\mbox{ for }\mathbf{t}\in \mathbb{%
R}^{2\ell }\,.  \label{eq-def-phi}
\end{equation}

\iffalse
\abbr{WLOG} we assume $\theta _{1,0}=0$ and $c_{0}>0$, then
we have that%
\begin{equation}
T_{0,n+k}(1,n)-T_{0,n}(1,n)=\sum_{i=1}^{n}O((n-i)^{m(1)-2})\xi _{i}%
\mbox{
for all }0\leq k\leq K.  \label{bound-T-0-ni}
\end{equation}%
For $j\in \lbrack \ell ]$, from (\ref{eq-def-T-small-order}) and the
condition $m_{j}\leq m(1)$, obviously%
\begin{equation}
\left\vert T_{j,n+k,m_{j}-2}(1,n)\right\vert +|T_{j,n+k,m_{j}-2}^{\prime
}(1,n)|=\sum_{i=1}^{n}O((n-i)^{m(1)-2})\xi _{i}\,\mbox{ for all }0\leq k\leq
K.  \label{bound-T-mj-1}
\end{equation}%
With (\ref{eq-R_theta}) it is implied that%
\begin{equation}
T_{j,n+k}(1,n)=[1,0](\mathsf{R}_{\theta _{j}})^{k}\mathbf{T}%
_{j,n}+\sum_{i=1}^{n}O((n-i)^{m(1)-2})\xi _{i}\mbox{ for all }0\leq k\leq K.
\label{bound-T-ni-condition}
\end{equation}%
Combining with (\ref{eq-X-n-expansion}), (\ref{eq-def-phi}), (\ref%
{bound-T-0-ni}) and (\ref{bound-T-ni-condition}) we can see that%
\begin{equation}
\min_{0\leq k\leq K}X_{n+k}(1,n)=c_{0}T_{0,n}+\phi _{K}(\mathbf{T}%
_{n})+\sum_{i=1}^{n}O((n-i)^{m(1)-2})\xi _{i}.  \label{eq-using-phi}
\end{equation}%
\fi

The following lemma is stronger than what we need in the current section,
and will be used later in establishing the power law.

\begin{lemma}
\label{lem-rotation} Assume $\exists \lambda \in \Lambda ^{\star }\setminus
\{1\}\mbox{ such that }m(\lambda )=m(1)$. $\,$Let $\{\tilde{T}_{0,\cdot }\}$
and $\{\tilde{\mathbf{T}}_{j,\cdot }\}$ (for $j=1,\ldots ,\ell $) be
independent while $\{\tilde{T}_{0,\cdot }\}$ has the same distribution as $%
\{T_{0,\cdot }\}$ and $\{\tilde{\mathbf{T}}_{j,\cdot }\}$ has the same
distribution as $\{\mathbf{T}_{j,\cdot }\}$. Then for any $\epsilon >0$ and $%
K\in \mathbb{N}$, there exists $C>0$ such that for any $N\in \mathbb{N}$,%
\begin{align*}
\mathbb{P}& (\Omega _{N}^{+})\leq \mathbb{P}(c_{0}(1+\epsilon )\tilde{T}%
_{0,n}+\phi _{K}(\tilde{\mathbf{T}}_{n})\geq -\epsilon n^{m(1)-1/2}\,,%
\mbox{
for }\log N\leq n\leq N-K)+\mathrm{e}^{-C(\log N)^{2}}\,, \\
\mathbb{P}& (\Omega _{N}^{+})\geq C\mathrm{e}^{-C(\log N)^{3/4}}(\mathbb{P}%
(c_{0}(1-\epsilon )\tilde{T}_{0,n}+\phi _{K}(\tilde{\mathbf{T}}_{n})\geq
\epsilon n^{m(1)-1/2},\mbox{ for }1\leq n\leq N)-\mathrm{e}^{-C(\log
N)^{4/3}})\,.
\end{align*}
\end{lemma}

\begin{proof}
\ We first prove the upper bound. Define $\Lambda _{0}:=\{\lambda \in
\Lambda ^{\star },m(\lambda )=m(1)\}$. With Lemma \ref%
{lem-general-regression-coefficients} and Lemma \ref{lemma-convolution} we
get%
\begin{equation}
\mathrm{\var}\left[ X_{n}\right] =\sum_{\ell =0}^{n}(\sum_{\lambda \in
\Lambda }\sum_{j=0}^{m(\lambda )-1}a_{\lambda ,j}|\lambda |^{\ell }\ell
^{j}\cos (\ell \theta _{\lambda }+\theta _{\lambda ,j}))^{2}=\sum_{\lambda
\in \Lambda _{0}}a_{\lambda ,m(1)-1}^{2}n^{2m(1)-1}+O(n^{2m(1)-2}),
\label{eq-var-Xn}
\end{equation}%
and similarly%
\begin{equation}
\mathrm{\var}[(1+\epsilon )c_{0}\tilde{T}_{0,n}+\sum_{j=1}^{\ell }c_{j}%
\tilde{T}_{j,n}]=((1+\epsilon )^{2}a_{1,m(1)-1}^{2}+\sum_{\lambda \in
\Lambda _{0}\backslash \{1\}}a_{\lambda
,m(1)-1}^{2})n^{2m(1)-1}+O(n^{2m(1)-2}).  \label{eq-Var-1+epi-Xn}
\end{equation}%
For any $\epsilon >0$, because%
\begin{equation*}
\frac{(1+\epsilon )^{2}a_{1,m(1)-1}^{2}+\sum_{\lambda \in \Lambda
_{0}\backslash \{1\}}a_{\lambda ,m(1)-1}^{2}}{\sum_{\lambda \in \Lambda
_{0}}a_{\lambda ,m(1)-1}^{2}}<(1+\epsilon )^{2},
\end{equation*}%
combining with (\ref{eq-var-Xn}) and (\ref{eq-Var-1+epi-Xn}) it follows that
there exists $n_{1}\in \mathbb{N}$ and $\kappa _{\epsilon }<1$ such that for
all $n\geq n_{1}$, we can find $\epsilon _{n}<\kappa _{\epsilon }\epsilon $
such that $W_{n}:=((1+\epsilon )c_{0}\tilde{T}_{0,n}+\sum_{j=1}^{\ell }c_{j}%
\tilde{T}_{j,n})/(1+\epsilon _{n})$ has the same variance as $X_{n}$.
Furthermore by Lemma \ref{lem-general-regression-coefficients} and Lemma \ref%
{lemma-convolution} it is not hard to verify that there exists $n_{\epsilon
}\geq n_{1}$ such that%
\begin{equation*}
\mathrm{\Cov}(X_{n},X_{n^{\prime }})\leq \mathrm{\Cov}(W_{n},W_{n^{\prime }})%
\mbox{ for all }n,n^{\prime }\geq n_{\epsilon }\,.
\end{equation*}%
Therefore, by Slepian's Lemma we obtain that%
\begin{equation}
\mathbb{P}(X_{n}\geq 0\mbox{ for }n_{\epsilon }\leq n\leq N)\leq \mathbb{P}%
(W_{n}\geq 0\mbox{ for }n_{\epsilon }\leq n\leq N)\,.  \label{eq-X-W-U-1}
\end{equation}%
Consider $N$ large enough such that $\log N\geq n_{\epsilon }$. Denote by%
\begin{align*}
E& :=\{c_{0}(1+\epsilon )\tilde{T}_{0,n}+\phi _{K}(\tilde{\mathbf{T}}%
_{n})\geq -\epsilon n^{m(1)-1/2}\,,\mbox{ for all }\log N\leq n\leq N-K\}\,,
\\
F& :=\bigcup\limits_{n=\log N}^{N}\{\mbox{$\sum_{j=1}^l$}\Vert \tilde{%
\mathbf{T}}_{j,n,m_{j}-2}\Vert _{2}\geq c\epsilon K^{-1}n^{m(1)-1/2}\}\,.
\end{align*}%
Note that under $F^{c}$, for any $\log N\leq n\leq N-K$ and $i\leq K$, by (%
\ref{eq-R_theta}) we have{\small 
\begin{equation}
|[1,0]\sum_{j=1}^{\ell }c_{j}(\mathsf{R}_{\theta _{j}})^{i}\tilde{\mathbf{T}}%
_{j,n}-\sum_{j=1}^{\ell }c_{j}\tilde{T}_{j,n+i}|\leq \max_{1\leq j\leq \ell
}\left\{ \left\vert c_{j}\right\vert \right\} \sum_{j=1}^{\ell
}\sum_{p=1}^{i}\Vert \tilde{\mathbf{T}}_{j,n+p,m_{j}-2}\Vert _{2}\leq
\epsilon \ell c\max_{1\leq j\leq \ell }\left\{ \left\vert c_{j}\right\vert
\right\} n^{m(1)-1/2}.  \label{bound-phi-T-til}
\end{equation}%
}Choose $c$ such that $c\ell \max_{1\leq j\leq \ell }\left\{ \left\vert
c_{j}\right\vert \right\} <1/2$. Combining with (\ref{bound-phi-T-til}), the
definition of $W_{n}$ and the definition of $\phi _{K}$, it follows that $%
E^{c}\cap F^{c}\subseteq \{W_{n}\geq 0\mbox{ for  }\log N\leq n\leq N\}^{c}$%
, which implies that $\{W_{n}\geq 0\mbox{ for  }\log N\leq n\leq
N\}\subseteq E\cup F$. A simple union bound gives that $\mathbb{P}(F)\leq 
\mathrm{e}^{-C(\log N)^{2}}$ for some $C>0$, yielding the desired upper
bound together with (\ref{eq-X-W-U-1}).

We next turn to the proof of the lower bound. With the similar method to the
proof of the upper bound, we can show that for any $\epsilon >0$, there
exist $n_{\epsilon }^{\prime }\in \mathbb{N}$ and $\kappa _{\epsilon
}^{\prime }<1$ such that for all $n\geq n_{\epsilon }^{\prime }$, we can
find $\epsilon _{n}<\kappa _{\epsilon }^{\prime }\epsilon $ such that $%
U_{n}=((1-\epsilon )c_{0}\tilde{T}_{0,n}+\sum_{j=1}^{\ell }c_{j}\tilde{T}%
_{j,n})/(1-\epsilon _{n}^{\prime })$ has the same variance as $X_{n}$, and%
\begin{equation*}
\mathrm{\Cov}(X_{n},X_{n^{\prime }})\geq \mathrm{\Cov}(U_{n},U_{n^{\prime }})%
\mbox{ for all }n,n^{\prime }\geq n_{\epsilon }^{\prime }.
\end{equation*}%
Therefore, by Slepian's Lemma we obtain that%
\begin{equation}
\mathbb{P}(X_{n}\geq 0\mbox{ for }n_{\epsilon }^{\prime }\leq n\leq N)\geq 
\mathbb{P}(U_{n}\geq 0\mbox{ for }n_{\epsilon }^{\prime }\leq n\leq N)\,.
\label{eq-X-W-U}
\end{equation}%
Denote by%
\begin{equation*}
\widehat{E}:=A\bigcap B\text{,}
\end{equation*}%
where $A:=\{X_{n}\geq 0\mbox{ for }n_{\epsilon }^{\prime }\leq n\leq N\}$, $%
B:=\{|X_{n}|\leq (\log N)^{2/3}\mbox{ for }n_{\epsilon }^{\prime }-L+1\leq
n\leq n_{\epsilon }^{\prime }\}$. By \eqref{eq-X-W-U} and a union bound on $%
\{|X_{n}|\geq (\log N)^{2/3}\}$, we get that%
\begin{equation}
\mathbb{P}(\widehat{E})\geq \mathbb{P}(U_{n}\geq 0\mbox{ for }n_{\epsilon
}^{\prime }\leq n\leq N)-\mathrm{e}^{-C_{2}(\log N)^{4/3}},
\label{eq-hat-E-bound}
\end{equation}%
for some $C_{2}>0$. For $i\in \left\{ 1,...,L\right\} $ denote $h_{n,i}$ as
the solution to (\ref{general-regression-relation}) with initial condition $%
h_{i,i}=1,h_{j,i}=0$ for $j\in \left\{ 1,...,L\right\} \backslash \{i\}$.
Then by Lemma \ref{lem-general-regression-coefficients} we have $%
h_{n,i}=O(n^{m(1)-1})$, and by the regressive relation we can see that
starting from $\{X_{n}:n_{\epsilon }^{\prime }-L+1\leq n\leq n_{\epsilon
}^{\prime }\}$, the coefficient of $X_{n_{\epsilon }^{\prime }-L+i}$ in $%
X_{n^{\prime }}$ is $h_{n-n_{\epsilon }^{\prime }+L,i}$. Thus under $B$, for
all $n^{\prime }>n_{\epsilon }^{\prime }$ we have%
\begin{equation}
{\mathbb{E}}(X_{n^{\prime }}\mid \{X_{n}:n_{\epsilon }^{\prime }-L\leq n\leq
n_{\epsilon }^{\prime }\})=\sum_{i=0}^{L}h_{n-n_{\epsilon }^{\prime
}+L,i}X_{n_{\epsilon }^{\prime }-L+i}\leq O(1)(\log N)^{2/3}(n^{\prime
}-n_{\epsilon }^{\prime })^{m(1)-1}\,.  \label{eq-hat-E}
\end{equation}%
Denote by $n^{\star }=(\log N)^{3/4}$. For some $M,\delta >0$ to be
determined later, define%
\begin{equation}
\widehat{E}_{p}=\bigcap\limits_{n=1}^{L}\{M\leq X_{n}\leq \left( 1+\delta
\right) M\}\bigcap \bigcap\limits_{n=L+1}^{p}\{1\leq \xi _{n}\leq 1+\delta
\}\,.  \label{def-E-hat-p}
\end{equation}%
From Lemma \ref{lem-general-regression-coefficients} and the fact that $1$
is a root of $Q$, for $n>L$ there exists $C_{1}>0$ such that $X_{n}(1,L)\geq
M-C_{1}\delta Mn^{m(1)-1}$. By Lemma \ref%
{lem-general-regression-coefficients} and Lemma \ref{lemma-convolution}, it
is easy to see that there exists $C_{2},C_{3}>0$ and $N_{1}\in \mathbb{Z}^{+}
$, such that $X_{n}(L+1,n)\geq C_{2}n^{m(1)}$ if $n\geq N_{1}$, and $%
X_{n}(L+1,n)\geq -C_{3}$ if $n\leq N_{1}$. Choosing $M=2C_{3}+1$, $\delta
=(2C_{1}N_{1}^{m(1)-1})^{-1}\wedge C_{2}(2C_{1}M)^{-1}$, one can then verify
that for any $p\in \mathbb{Z}^{+}$, $\widehat{E}_{p}\subset \Omega _{p}^{+}$%
. In the following we consider $\widehat{E}_{n^{\star }}$. Under $\widehat{E}%
_{n^{\star }}$, one can verify that there exists $L^{\prime }>L$ such that $%
X_{n}(1,L^{\prime })\geq 0$ for $n$ large enough, and there exists $C>0$
such that for all $n^{\prime }>n^{\star }$, $X_{n^{\prime }}(L^{\prime
}+1,n^{\star })\geq C(\log N)^{17/24}(n^{\prime }-n^{\star })^{m(1)-1}$ as $N
$ is large enough, where the exponent $17/24$ is an arbitrarily chosen
number between $2/3$ and $3/4$. Thus as $N$ is large enough%
\begin{equation}
{\mathbb{E}}(X_{n^{\prime }}\mid \{X_{n}:n^{\star }-L+1\leq n\leq n^{\star
}\})\geq C(\log N)^{17/24}(n^{\prime }-n^{\star })^{m(1)-1}\,.
\label{eq-expectation-X-n-prime}
\end{equation}%
Note that for all $n^{\prime }>n_{\epsilon }^{\prime }$, the distribution of 
$X_{n^{\prime }}-{\mathbb{E}}(X_{n^{\prime }}\mid \{X_{n}:n_{\epsilon
}^{\prime }-L+1\leq n\leq n_{\epsilon }^{\prime }\})$ is same as the
distribution of $X_{n^{\prime }-n_{\epsilon }^{\prime }+n^{\star }}-{\mathbb{%
E}}(X_{n^{\prime }-n_{\epsilon }^{\prime }+n^{\star }}\mid \{X_{n}:n^{\star
}-L+1\leq n\leq n^{\star }\}),$ combined with \eqref{eq-hat-E} and (\ref%
{eq-expectation-X-n-prime}) it yields that for $N$ large enough%
\begin{equation}
\mathbb{P}(X_{n}\geq 0\mbox{ for all }n^{\star }\leq n\leq N-n_{\epsilon
}^{\prime }+n^{\star }\mid \widehat{E}_{n^{\star }})\geq \mathbb{P}%
(X_{n}\geq 0\mbox{ for all }n_{\epsilon }^{\prime }\leq n\leq N\mid B).
\label{jun-add-0403-1}
\end{equation}%
Further we\ can choose $N$ large enough such that $n^{\star }\geq $ $%
n_{\epsilon }^{\prime }$, thus with (\ref{jun-add-0403-1}) the following
holds,%
\begin{equation}
\mathbb{P}(X_{n}\geq 0\mbox{ for all }n^{\star }\leq n\leq N\mid \widehat{E}%
_{n^{\star }})\geq \mathbb{P}(A\mid B)\geq \mathbb{P}(A,B)=\mathbb{P}(%
\widehat{E}).  \notag
\end{equation}%
Note that obviously $\mathbb{P}(\widehat{E}_{n^{\star }})\geq (C_{\delta
}\delta )^{n^{\star }}$ for some $C_{\delta }>0$. Recalling that $\Omega
_{n^{\star }}^{+}\subset \widehat{E}_{n^{\star }}$, we have%
\begin{equation*}
p_{N}\geq \mathbb{P}(\widehat{E}_{n^{\star }})\mathbb{P}(X_{n}\geq 0%
\mbox{
for all }n^{\star }\leq n\leq N\mid \widehat{E}_{n^{\star }})\geq \mathbb{P}(%
\widehat{E}_{n^{\star }})\mathbb{P}(\widehat{E})\geq (C_{\delta }\delta
)^{n^{\star }}\mathbb{P}(\widehat{E})\,.
\end{equation*}%
Combined with \eqref{eq-hat-E-bound} and the fact that $\sum_{j=1}^{\ell
}c_{j}\tilde{T}_{j,n}\geq \phi _{K}(\tilde{\mathbf{T}}_{n})$ by the
definition of $\phi _{K}$ at (\ref{eq-def-phi}), it completes the proof of
the lemma.
\end{proof}

Here we show the lower bound in case $r^{\star }=1,m(1)=m(\lambda )$\ for
some $\lambda \in \Lambda ^{\star }$.

\begin{lemma}
\label{lem-rotated-random-walk} For any constant $C>0$, there exist $%
c,C^{\prime }>0$ such that for all $j\leq \ell $%
\begin{equation}
\mathbb{P}(\Vert \mathbf{T}_{j,n}\Vert _{2}\leq Cn^{m(\lambda _{j})-\tfrac{1%
}{2}}\mbox{ for all }1\leq n\leq N)\geq cN^{-C^{\prime }}\mbox{ for all }%
N\in \mathbb{N}\,.
\end{equation}
\end{lemma}

\begin{proof}
Applying \eqref{eq-R_theta} and the triangle inequality, we see that%
\begin{equation}
\{\Vert \mathbf{T}_{j,n}\Vert _{2}\leq Cn^{m(\lambda _{j})-\tfrac{1}{2}}%
\mbox{ for
all }1\leq n\leq N\}\supseteq \{\Vert \mathbf{T}_{j,n,0}\Vert _{2}\leq C%
\sqrt{n}\mbox{ for all }1\leq n\leq N\}\,.  \label{eq-T-j-n-T-j-n-1}
\end{equation}%
By (\ref{eq-R_theta}) we have that%
\begin{equation*}
\mathbf{T}_{j,n,0}=\sum_{i=1}^{n}(\mathsf{R}_{\theta _{j}})^{n-i}[\cos
\theta _{\lambda _{j},m_{j}},\theta _{\lambda _{j},m_{j}}]^{T}\xi _{i}\,,
\end{equation*}%
and therefore%
\begin{equation}
\Vert \mathbf{T}_{j,n,0}\Vert _{2}\leq \Vert \sum_{i=0}^{n}(\mathsf{R}%
_{\theta _{j}})^{-i}[1,0]^{T}\xi _{i}\Vert _{2}+\Vert \sum_{i=0}^{n}(\mathsf{%
R}_{\theta _{j}})^{-i}[0,1]^{T}\xi _{i}\Vert _{2}.  \label{ineq-T-j-n-1}
\end{equation}%
Write $S_{1,n}=[1,0]\sum_{i=0}^{n}(\mathsf{R}_{\theta
_{j}})^{-i}[1,0]^{T}\xi _{i}$ and $S_{2,n}=[0,1]\sum_{i=0}^{n}(\mathsf{R}%
_{\theta _{j}})^{-i}[1,0]^{T}\xi _{i}$. Define%
\begin{equation*}
\Gamma _{k}:=\{\left\vert S_{1,\gamma ^{k+1}}\right\vert <C\sqrt{\gamma
^{k+1}}/4\}\text{.}
\end{equation*}%
Noting that $\var(S_{1,n})=n+o(n)$, by the independence of $S_{1,\gamma
^{k+1}}-S_{1,\gamma ^{k}}$ and $S_{1,\gamma ^{k}}$, it is easy to see that
there exists a $\gamma >1$ just depending on $C$ and a constant $C_{0}>0$,
such that $\mathbb{P(}\Gamma _{k}\mid S_{1,\gamma ^{k}}=z)>C_{0}$ for each $k
$ and any $z$ with $|z|\leq C\sqrt{\gamma ^{k}}/4$. We can check that $%
\{|S_{1,n}|\leq C\sqrt{n}/4\mbox{ for all }1\leq n\leq N\}\subset \cap
_{k=1}^{[\log _{\gamma }N-1]}\Gamma _{k}$. Then by writing $\mathbb{P(}\cap
_{k=1}^{[\log _{\gamma }N-1]}\Gamma _{k})$ as the multiplication of a
sequence of conditional probability, with previous upper bound of the
conditional probability we see that there exist $c,C^{\prime }>0$ such that%
\begin{equation*}
\mathbb{P}(|S_{1,n}|\leq C\sqrt{n}/4\mbox{ for all }1\leq n\leq N)\geq
c^{1/4}N^{-C^{\prime }/4}.
\end{equation*}%
Same applies for $\mathbb{P}(|S_{2,n}|\leq C\sqrt{n}/4\mbox{ for all }1\leq
n\leq N)$. By Gaussian correlation inequality (see \cite{R14}, and weaker
versions in \cite{Li99, SSZ98} which would also work for our proof), we have%
\begin{equation}
\mu (A\bigcap B)\geq \mu (A)\mu (B)\mbox{ for all convex
symmetric sets }A\mbox{ and }B\,,  \label{ineq-gaussian-peri}
\end{equation}%
where $\mu $ is a Gaussian measure on Euclidean space. Applying (\ref%
{ineq-gaussian-peri}) with%
\begin{equation*}
A=\{|S_{1,n}|\leq C\sqrt{n}/4\mbox{ for all }1\leq n\leq N\}\text{, }%
B=\{|S_{1,n}|\leq C\sqrt{n}/4\mbox{ for all }1\leq n\leq N\},
\end{equation*}%
we have%
\begin{align}
\mathbb{P}& (\Vert \sum_{i=0}^{n}(\mathsf{R}_{\theta _{j}})^{-i}[1,0]^{T}\xi
_{i}\Vert _{2}\leq C\sqrt{n}/2\mbox{ for all }1\leq n\leq N)  \notag \\
\geq & \mathbb{P}(|S_{1,n}|\leq C\sqrt{n}/4\mbox{ for all }1\leq n\leq N)%
\mathbb{P}(|S_{2,n}|\leq C\sqrt{n}/4\mbox{ for all }1\leq n\leq N)\geq
c^{1/2}N^{-C^{\prime }/2}\,.  \notag
\end{align}%
Similar argument applies to show that $\mathbb{P(}\Vert \sum_{i=0}^{n}(%
\mathsf{R}_{\theta _{j}})^{-i}[0,1]^{T}\xi _{i}\Vert _{2}\leq C\sqrt{n}%
/2)\geq c^{1/2}N^{-C^{\prime }/2}$. Using (\ref{ineq-T-j-n-1}) and applying (%
\ref{ineq-gaussian-peri}) again gives a desired lower bound for the
probability of the event in the right hand side of (\ref{eq-T-j-n-T-j-n-1}).
Hence with (\ref{eq-T-j-n-T-j-n-1}), the desired estimates follows.
\end{proof}

Write $T_{0,n,0}=\sum_{i=1}^{n}\xi _{i}$. For a Brownian motion $B$ and a
constant $C>0$, we consider the following intervals $%
[2^{k},2^{k+1}],k=0,...,[\log N]$, which obviously cover $[1,N]$. Define the
following events for each $k$,%
\begin{equation*}
\Lambda _{k}:=\{B(t)-B(2^{k})>C\sqrt{t}-2C\sqrt{2^{k}}\text{ for any }t\in
\lbrack 2^{k},2^{k+1}]\text{ and }B(2^{k+1})-B(2^{k})>2C\sqrt{2^{k}}\}.
\end{equation*}%
By scaling, it is easy to see that there exists a $C_{0}>0$, such that for
any $k$ we have $\mathbb{P}(\Lambda _{k})>C_{0}$. Noting that by the
definition of $\Lambda _{k}$'s we can check that $\{B(t)>C\sqrt{t}$ for all $%
t\in \lbrack 0,N]\}\in \{B(1)>C\}\cap \cap _{k=1}^{[\log N]}\Lambda _{k}$.
With the independence of $\Lambda _{k}$'s and the fact that we can regard $%
T_{0,n,0}$ as a Brownian motion at integer times, we see that there exists $%
c>0$ such that%
\begin{equation}
\mathbb{P}(T_{0,n,0}\geq 2Cm(1)!\sqrt{n}\mbox{ for all }1\leq n\leq N)\geq
cN^{-1/c}.  \label{ineq-P-T-0-n-1}
\end{equation}%
Observing that%
\begin{equation*}
\{T_{0,n,0}\geq 2Cm(1)!\sqrt{n}\mbox{ for all }1\leq n\leq N\}\subset
\{T_{0,n}\geq Cn^{m(1)-1/2}\mbox{ for all }1\leq n\leq N\},
\end{equation*}%
combined with (\ref{ineq-P-T-0-n-1}) we get%
\begin{equation}
\mathbb{P}(T_{0,n}\geq Cn^{m(1)-1/2}\mbox{ for all }1\leq n\leq N)\geq
cN^{-1/c}\,.  \label{lower-bound-integrated-prob}
\end{equation}%
Noting that $\phi _{K}(\mathbf{t})\geq -\sum_{j=2}^{\ell -1}c_{j}\Vert 
\mathbf{t}_{j}\Vert $, with Lemmas~\ref{lem-rotation} and \ref%
{lem-rotated-random-walk}, we complete the proof on the lower bound.

\subsection{\textbf{Lower bound with dominant zero-angle component\label%
{subsection-e-lower-bound-1}}}

Here we complete the proof of lower bound of part (e) in Theorem \ref{thm-1}%
, for the case $r^{\star }=1$ and $m(1)>m(\lambda )$\ for all $\lambda \in
\Lambda ^{\star }\backslash \{1\}$.

\begin{lemma}
\label{lem-rotation-2} If $m(\lambda )<m(1)$ for all $\lambda \in \Lambda
^{\star }\setminus \{1\}$, then there exist constants $c,C>0$ such that for
any $N$,%
\begin{equation*}
p_{N}\geq cN^{-C}\,.
\end{equation*}
\end{lemma}

\begin{proof}
First we assume $m(1)>1$. Denote by $m=m(1)$, $\gamma =m-3/2$. Let $%
S_{n}^{(m)}$ be an \textsc{\lowercase{IRW}} of order $m$, then by (\ref%
{lower-bound-integrated-prob}) there exist constants $c_{1},C_{1}>0$ such
that%
\begin{equation}
\mathbb{P}(S_{n}^{(m)}\geq 0\text{ for }n=1,...,N)\geq c_{1}N^{-C_{1}}.
\label{lower-bound-integrated-prob-2}
\end{equation}%
By Lemma \ref{lem-general-regression-coefficients} we can write%
\begin{equation*}
X_{n}=c_{0}S_{n}^{(m)}+Y_{n},\text{ where Var}(Y_{n})=O(n^{2\gamma })\text{.}
\end{equation*}%
Denote by 
\begin{equation*}
n^{\star }:=[\log N]^{2},n^{\prime }:=[\log \log N]^{2},
\end{equation*}%
\begin{equation*}
E_{1}:=\bigcap\limits_{j=1}^{m}\{S_{n^{\star }}^{(j)}\geq (\log N)^{2j-2}\}.
\end{equation*}%
Note that under $E_{1}$, there exists $C_{2}>0$ such that $%
S_{n}^{(m)}(0,n^{\star })\geq C_{2}n^{m-1}\geq C_{2}n^{\gamma }\log N$ for
all $n^{\star }\leq n\leq N$. Let $\widetilde{S}_{n}^{(m)}$ be an
auto-regressive process generated by $(z-1)^{m}$ and $\xi _{n}1_{n\geq
n^{\star }}$, then obviously $\widetilde{S}_{n^{\star }+k}^{(m)}$ has the
same distribution as $S_{k}^{(m)}$ for $k\in \mathbb{Z}$, and we have that
under $E_{1}$ 
\begin{equation}
S_{n^{\star }+k}^{(m)}=\widetilde{S}_{n^{\star }+k}^{(m)}+S_{n^{\star
}+k}^{(m)}(0,n^{\star })\geq \widetilde{S}_{n^{\star
}+k}^{(m)}+C_{2}(n^{\star }+k)^{\gamma }\log N.  \label{eq-S-n-star-k}
\end{equation}%
Let $E_{0}:=\{S_{n}^{(m)}\geq C_{2}n^{\gamma }\log N$ for all $n^{\star
}\leq n\leq N\},$ thus with (\ref{lower-bound-integrated-prob-2}) and (\ref%
{eq-S-n-star-k}) we get%
\begin{equation}
\mathbb{P}(E_{0}\mid E_{1})\geq \mathbb{P}(\widetilde{S}_{n}^{(m)}\geq 0%
\text{ for }n=1,...,N)\geq c_{1}N^{-C_{1}}.  \label{bound-P-E-0-E-1}
\end{equation}%
Similarly, let $E_{4}:=\cap _{j=1}^{m}\{S_{n^{\prime }}^{(j)}\geq C_{5}(\log
\log N)^{2j}\}$ with $C_{5}$ to be determined later. Under $E_{4}$ there
exists $C_{3}>0$ such that$\ S_{n}^{(j)}(0,n^{\prime })\geq
C_{3}C_{5}n^{(j-1)}(\log \log N)^{2}$ for all $n^{\prime }\leq n\leq N$ and $%
j\in \lbrack m]$. Let%
\begin{equation*}
E_{2}:=\{S_{n}^{(m)}\geq n^{m-1}(\log \log N)^{2}\text{ for all }n^{\prime
}\leq n\leq n^{\star }\text{, }S_{n^{\star }}^{(j)}\geq (\log N)^{2j-2}\text{
for all }j\in \lbrack m]\}.
\end{equation*}%
Choosing $C_{5}=1/C_{3}$, with the similar argument to (\ref{bound-P-E-0-E-1}%
) we see that%
\begin{equation}
\mathbb{P}(E_{2}\mid E_{4})\geq c_{3}(\log N)^{-C_{3}}\geq c_{3}N^{-C_{3}},
\label{bound-P-E-2-E-4}
\end{equation}%
for some $c_{3},C_{3}>0$. Let%
\begin{equation*}
E_{3}:=\{X_{n}\geq 0\text{ for all }1\leq n\leq n^{\prime }\text{, }%
S_{n^{\prime }}^{(j)}\geq C_{5}(\log \log N)^{2j}\text{ for all }j\in
\lbrack m]\}.
\end{equation*}%
Recalling the definition of $\widehat{E}_{p}$ (\ref{def-E-hat-p}), it is
easy to see that $\widehat{E}_{n^{\prime }}\subset E_{3}$ for $n^{\prime }$
large enough, thus%
\begin{equation}
\mathbb{P}(E_{3})\geq c_{4}e^{-C_{4}n^{\prime }}.  \label{bound-P-E-3}
\end{equation}%
Further define{\small 
\begin{eqnarray*}
F_{0} &:&=\{Y_{n}\geq -(C_{2}/c_{0})n^{\gamma }\log N\text{ for all }%
n^{\star }\leq n\leq N\},\text{ } \\
F_{1} &:&=\{Y_{n}\geq -(1/c_{0})n^{\gamma }(\log \log N)^{2}\text{ for all }%
n^{\prime }\leq n\leq n^{\star }\}.
\end{eqnarray*}%
}Then by simple union bound there exist $C_{6},C_{7}>0$ such that%
\begin{equation}
\mathbb{P}(F_{0}^{c})\leq e^{-C_{6}(\log N)^{2}}\text{, }\mathbb{P}%
(F_{1}^{c})\leq e^{-C_{7}(\log \log N)^{4}}.  \label{bound-P-F-1-2}
\end{equation}%
By definition it is straightforward to see that%
\begin{equation}
p_{N}\geq \mathbb{P}(E_{3},E_{2},E_{0},F_{1},F_{0})\geq (\mathbb{P}(E_{3})%
\mathbb{P}(E_{2}\mid E_{3})-\mathbb{P}(F_{1}^{c}))\mathbb{P}(E_{0}\mid
F_{1},E_{2},E_{3})-\mathbb{P}(F_{0}^{c}).  \label{eq-P-omega-E-F}
\end{equation}%
Note that $E_{1}$ is independent of $\sigma (\xi _{1},...,\xi _{n^{\prime }})
$ conditioned on $\{S_{n^{\star }}^{(j)},$ $j\in \lbrack m]\}$, and the
similar holds for $E_{2}$. Plugging (\ref{bound-P-E-0-E-1}), (\ref%
{bound-P-E-2-E-4}), (\ref{bound-P-E-3}) and (\ref{bound-P-F-1-2}) into (\ref%
{eq-P-omega-E-F}), we see that the lemma is true in this case.

Now if $m(1)=1$, then according to the condition we get $\Lambda ^{\star
}=\{1\}$, and thus by Lemma \ref{lem-general-regression-coefficients} we can
write%
\begin{equation*}
X_{n}=c_{0}S_{n}^{(1)}+Y_{n},\text{ where Var}(Y_{n})\leq C_{0}\text{ for
some }C_{0}<\infty \text{.}
\end{equation*}%
Consider 
\begin{eqnarray*}
E_{0} &:&=\{S_{n}^{(1)}\geq (\log N)/2\text{ for all }\log N\leq n\leq N\}%
\text{, }E_{1}:=\{S_{\log N}^{(1)}\geq (\log N)/2\}, \\
E_{2} &:&=\{X_{n}\geq 0\text{ for all }1\leq n\leq \log N\text{, }S_{\log
N}^{(1)}\geq (\log N)/2\}, \\
F_{0} &:&=\{Y_{n}\geq -(\log N)/(2c_{0})\text{ for all }\log N\leq n\leq N\}.
\end{eqnarray*}%
Then by the similar argument to above, there exist $%
c_{2},C_{2},c_{3},C_{3},C_{4}>0$ such that%
\begin{equation*}
\mathbb{P}(E_{0}\mid E_{1})\geq c_{2}N^{-C_{2}}\text{, }\mathbb{P}%
(E_{2})\geq c_{3}e^{-C_{3}\log N}\text{, }\mathbb{P}(F_{0}^{c})\leq
e^{-C_{4}(\log N)^{2}}.
\end{equation*}%
Thus the proof is completed by%
\begin{equation*}
p_{N}\geq \mathbb{P}(E_{0},E_{2},F_{0})\geq \mathbb{P}(E_{0}\mid E_{2})%
\mathbb{P}(E_{2})-\mathbb{P}(F_{0}^{c})\geq
c_{2}c_{3}N^{-C_{2}-C_{3}}-e^{-C_{4}(\log N)^{2}}.
\end{equation*}
\end{proof}

\section{Persistence power exponent for \textsc{\lowercase{AR}}$_{3}$:
Theorem~\protect\ref{thm-discontinuity}\label{section-power-law}}

First, in Subsection \ref{subsection-ar-bm} we reduce the persistence
probability for regressive processes considered in Theorem~\ref%
{thm-discontinuity} to the probability for a 3-dimensional Brownian motion
to stay in a generalized cone. Then, in Subsection \ref%
{subsection-con-discon}, we show the existence of the persistence power
exponents in the case $\Lambda =\{1,\mathrm{e}^{\sqrt{-1}\theta },\mathrm{e}%
^{-\sqrt{-1}\theta }\}$ for $\theta \in (0,\pi )$, and analyze the
continuity and discontinuity of the power exponents depending on $\theta $
is rational or not.

\subsection{AR processes and the Brownian motion in a cone\label%
{subsection-ar-bm}}

We start with the following lemma.

\begin{lemma}
\label{lem-random-nickdim} For an \textsc{\lowercase{IRW}} $(S_{n}^{(m)})$
of order $m$, and an arbitrary deterministic sequence $(f_{n})$, there exist 
$c_{\epsilon }\rightarrow _{\epsilon \rightarrow 0}0$ such that the
following holds for all $N\in \mathbb{N}$,%
\begin{equation*}
\mathbb{P}(S_{n}^{(m)}\geq f_{n}+\epsilon n^{m-1/2}\,,\mbox{ for all }1\leq
n\leq N)\geq (\mathbb{P}(S_{n}^{(m)}\geq f_{n}\,,\mbox{ for all }1\leq n\leq
N))^{1+c_{\epsilon }}N^{-c_{\epsilon }}\,.
\end{equation*}
\end{lemma}

\begin{proof}
Let $\Xi :=\{S_{n}^{(m)}\geq f_{n}\mbox{ for all }1\leq n\leq N\}$. Let $%
P_{0}$ be the original law of an i.i.d.\ sequence of $(\xi _{n})$, and let $%
P_{\epsilon }$ to be the law of independent sequence $(\xi _{n}^{\prime })$
such that $\xi _{n}^{\prime }\sim N(C\epsilon n^{-1/2},1)$ for each $n$.
Then it is easy to check that there exists $C>0$ depending only on $m$ such
that%
\begin{equation}
\mathbb{P}(S_{n}^{m}\geq f_{n}+\epsilon n^{m-1/2}\,,\mbox{ for all }1\leq
n\leq N)\geq \mathbb{P}_{\epsilon }(\Xi )\,.  \label{lower-bound-P-S-m-n}
\end{equation}%
Using the fact that $\frac{d\mathbb{P}_{\epsilon }}{d\mathbb{P}_{0}}%
((x_{n}))=\mathrm{e}^{\sum_{k}(C\epsilon x_{k}k^{-1/2}-C^{2}\epsilon
^{2}/2k)}$, for any $\delta >1$ we get 
\begin{equation}
{\mathbb{E}}_{\mathbb{P}_{0}}\big(\big(\frac{d\mathbb{P}_{\epsilon }}{d%
\mathbb{P}_{0}}\big)^{1-1/\delta }\big)\leq N^{C^{2}\epsilon ^{2}/\delta
^{2}}\,.  \label{bound-expectation-P-0-RN}
\end{equation}%
By H\"{o}lder inequality and Radon-Nikodym theorem, we obtain that for all $%
0<\delta <1$, 
\begin{equation}
\mathbb{P}_{0}(\Xi )={\mathbb{E}}_{\mathbb{P}_{\epsilon }}(\big(\frac{d%
\mathbb{P}_{0}}{d\mathbb{P}_{\epsilon }}\big)\mathbf{1}_{\Xi })\leq (\mathbb{%
P}_{\epsilon }(\Xi ))^{1-\delta }({\mathbb{E}}_{\mathbb{P}_{\epsilon }}\big(%
\big(\frac{d\mathbb{P}_{0}}{d\mathbb{P}_{\epsilon }}\big)^{1/\delta }\big)%
)^{\delta }=(\mathbb{P}_{\epsilon }(\Xi ))^{1-\delta }({\mathbb{E}}_{\mathbb{%
P}_{0}}\big(\big(\frac{d\mathbb{P}_{\epsilon }}{d\mathbb{P}_{0}}\big)%
^{1-1/\delta }\big))^{\delta }\,.  
\label{eq-Radon-Nikodim}
\end{equation}%
Setting $\delta =\sqrt{\epsilon }$ and plugging (\ref%
{bound-expectation-P-0-RN}) into \eqref{eq-Radon-Nikodim}, by (\ref%
{lower-bound-P-S-m-n}) we complete the proof of the lemma.
\end{proof}

In the rest of the paper, we assume that $\Lambda =\{1,\mathrm{e}^{\sqrt{-1}%
\theta },\mathrm{e}^{-\sqrt{-1}\theta }\}$ for $\theta \in (0,\pi )$, and
denote $\lambda =\mathrm{e}^{\sqrt{-1}\theta }$. Recall the definition of $%
\theta _{\lambda ,1}$ in Lemma \ref{lem-general-regression-coefficients}.

\begin{lemma}
\label{lem-couple-T-BM} Let $\mathbf{W}$ be a planar Brownian motion, and
let $\mathbf{T}_{n}:=\sum_{i=1}^{n}\mathsf{R}_{\theta _{\lambda ,1}}(\mathsf{%
R}_{\theta })^{-i}[1,0]^{T}\xi _{i}$. Then, for all $\epsilon >0$, there
exists a coupling of $(\mathbf{W},\mathbf{T})$ and $c>0$ depending only on $%
\epsilon $ such that for all $N$,%
\begin{equation*}
\mathbb{P}(\exists n\in \lbrack N]:\Vert \frac{1}{\sqrt{2}}\mathbf{W}_{n}-%
\mathbf{T}_{n}\Vert \geq (\log N)^{4}+\epsilon \sqrt{n})\leq \mathrm{e}%
^{-c(\log N)^{2}}\,.
\end{equation*}
\end{lemma}

\begin{proof}
Let $n_{k}:=[k\left( \log N\right) ^{2}]$ for $1\leq k\leq \lbrack N/\left(
\log N\right) ^{2}]$. Writing $\mathsf{I}$ as an identity matrix of size 2,
by Lemma~\ref{lemma-convolution} and the definition of $\mathbf{T}_{n}$ we
see that the covariance matrix $\mathsf{B}_{k}$ of $(\mathbf{T}_{n_{k+1}}-%
\mathbf{T}_{n_{k}})$ satisfies%
\begin{equation*}
\Vert \mathsf{B}_{k}-\tfrac{\left( \log N\right) ^{2}}{2}\mathsf{I}\Vert
_{\infty }\leq C,
\end{equation*}%
where $C$ is a constant depending only on $\theta $. Therefore, there exists 
$C^{\prime }>0$ such that%
\begin{equation*}
\mathsf{B}_{k}=\frac{\left( \log N\right) ^{2}-C^{\prime }}{2}\mathsf{I}+%
\mathsf{A}_{k}\,,
\end{equation*}%
where $\mathsf{A}_{k}$ is a positive definite matrix with $\Vert \mathsf{A}%
_{k}\Vert _{\infty }\leq C^{\prime }$. Let $\zeta _{k}$ be an independent
Gaussian vector in ${\mathbb{R}}^{2}$ with covariance matrix $\mathsf{A}_{k}$%
, and let $\mathbf{Y}_{n}=\sum_{k=1}^{n}\zeta _{k}$. Therefore, we see that%
\begin{equation}
\{\mathbf{T}_{n_{k}}:1\leq k\leq N/\left( \log N\right) ^{2}\}\overset{%
\mathrm{law}}{=}\{\frac{1}{\sqrt{2}}(1-\frac{C^{\prime }}{\left( \log
N\right) ^{2}})^{1/2}\mathbf{W}_{n_{k}}+\mathbf{Y}_{k}:1\leq k\leq N/\left(
\log N\right) ^{2}\}\,,  \label{eq-identity-in-law}
\end{equation}%
where $\mathbf{W}$ is independent of $\mathbf{Y}$. By a simple union bound,
for a constant $c^{\prime }>0$ we have that%
\begin{equation}
\mathbb{P}(\exists k\in \lbrack N/\left( \log N\right) ^{2}]:\Vert \mathbf{Y}%
_{k}\Vert \geq \epsilon \log N\sqrt{k}/4)\leq \mathrm{e}^{-c^{\prime }(\log
N)^{2}}\,,  \label{eq-Y-n}
\end{equation}%
and%
\begin{equation}
\mathbb{P}(\exists k\in \lbrack N/\left( \log N\right) ^{2}]:\Vert \mathbf{W}%
_{n_{k}}\Vert \geq \epsilon \left( \log N\right) ^{2}\sqrt{n_{k}}/\left(
4C^{\prime }\right) )\leq \mathrm{e}^{-c^{\prime }(\log N)^{2}}\,.
\label{bound-P-W-n}
\end{equation}%
From (\ref{eq-identity-in-law}) and the fact that $|(1-x)^{1/2}-1|<x$ for $%
x\in (0,1)$, we see that{\small 
\begin{eqnarray}
&&\mathbb{\{}\exists k\in \lbrack N/\left( \log N\right) ^{2}]:\Vert \frac{1%
}{\sqrt{2}}\mathbf{W}_{n_{k}}-\mathbf{T}_{n_{k}}\Vert \geq (\log
N)^{4}/2+\epsilon \sqrt{n_{k}}/2\}  \notag \\
&\subset &\mathbb{\{}\exists k\in \lbrack N/\left( \log N\right) ^{2}]:\Vert 
\mathbf{W}_{n_{k}}\Vert \geq \frac{1}{4C^{\prime }}\epsilon \left( \log
N\right) ^{2}\sqrt{n_{k}}\}\cup \{\exists k\in \lbrack N/\left( \log
N\right) ^{2}]:\Vert \mathbf{Y}_{k}\Vert \geq \frac{1}{4}\epsilon \log N%
\sqrt{k}\}.  \notag
\end{eqnarray}%
}Thus combined with (\ref{eq-Y-n}) and (\ref{bound-P-W-n}) we get%
\begin{equation}
\mathbb{P}(\exists k\in \lbrack N/\left( \log N\right) ^{2}]:\Vert \frac{1}{%
\sqrt{2}}\mathbf{W}_{n_{k}}-\mathbf{T}_{n_{k}}\Vert \geq (\log
N)^{4}/2+\epsilon \sqrt{n_{k}}/2)\leq 2\mathrm{e}^{-c^{\prime }(\log N)^{2}}.
\label{eq_3.1bound_add_2}
\end{equation}%
Similarly, by a union bound again it is easy to verify that for some
constant $c^{\prime \prime }>0$,%
\begin{equation}
\mathbb{P}(\exists k,p:\Vert \frac{1}{\sqrt{2}}\left( \mathbf{W}_{n_{k}+p}-%
\mathbf{W}_{n_{k}}\right) -\left( \mathbf{T}_{n_{k}+p}-\mathbf{T}%
_{n_{k}}\right) \Vert \geq (\log N)^{4}/2+\epsilon \sqrt{n_{k}}/2)\leq 
\mathrm{e}^{-c^{\prime \prime }(\log N)^{2}}\,.  \label{eq_3.1bound_add_3}
\end{equation}%
Combined with (\ref{eq_3.1bound_add_2}) and (\ref{eq_3.1bound_add_3}), it
completes the proof of the lemma.
\end{proof}

Recall the definition of $\phi _{K}(\cdot )$ in \eqref{eq-def-phi}, where in
what follows we specify $\ell =1$. We denote by $\phi (\mathbf{t}%
)=\lim_{K\rightarrow \infty }\phi _{K}(\mathbf{t})$ for all $\mathbf{t}\in {%
\mathbb{R}}^{2}$.

\begin{lemma}
\label{lem-reduce-to-BM} Let $\mathbf{W}$ be a planar Brownian motion
independent of a Brownian motion $B$. Then for $K\in \mathbb{N}$ large
enough independent of $N$, we have{\small 
\begin{equation*}
\limsup_{\epsilon \rightarrow 0}\lim_{N\rightarrow \infty }\frac{\log p_{N}}{%
\log N}\leq \liminf_{\epsilon \rightarrow 0}\lim_{N\rightarrow \infty }\frac{%
\mathbb{P}(c_{0}(1+\epsilon )B_{t}+\phi (\tfrac{1}{\sqrt{2}}\mathbf{W}%
_{t})\geq 0\,,\mbox{ for }(\log N)^{9}\leq t\leq N-K)}{\log N},
\end{equation*}%
\begin{equation*}
\liminf_{\epsilon \rightarrow 0}\lim_{N\rightarrow \infty }\frac{\log p_{N}}{%
\log N}\geq \limsup_{\epsilon \rightarrow 0}\lim_{N\rightarrow \infty }\frac{%
\mathbb{P}(c_{0}(1-\epsilon )B_{t}+\phi (\tfrac{1}{\sqrt{2}}\mathbf{W}%
_{t})\geq 0\,,\mbox{ for }1\leq t\leq N)}{\log N}.
\end{equation*}%
}
\end{lemma}

\begin{proof}
Obviously it is enough to show that there exist $c,c_{\epsilon }>0$ with $%
c_{\epsilon }\rightarrow _{\epsilon \rightarrow 0}0$, such that for any $%
N\in \mathbb{N}$,{\small 
\begin{align}
p_{N}\leq & N^{c_{\epsilon }}(\mathbb{P}(c_{0}(1+\epsilon )B_{t}+\phi (%
\tfrac{1}{\sqrt{2}}\mathbf{W}_{t})\geq 0\,,\mbox{ for }(\log N)^{9}\leq
t\leq N-K)^{1-c_{\epsilon }}+\mathrm{e}^{-c(\log N)^{2}})\,,
\label{eq-BM-upper} \\
p_{N}\geq & C\mathrm{e}^{-C(\log N)^{3/4}}N^{-c_{\epsilon }}\mathbb{P}%
(c_{0}(1-\epsilon )B_{t}+\phi _{K}(\mathbf{W}_{t})\geq 0,\mbox{ for }1\leq
t\leq N-(\log N)^{9})^{1+c_{\epsilon }}+\mathrm{e}^{-C(\log N)^{4/3}}.
\label{eq-BM-lower}
\end{align}%
}Let $\mathbf{T}_{n}$ be defined as in Lemma~\ref{lem-couple-T-BM} and let $%
\widetilde{\mathbf{T}}_{1,n}$ be defined as in Lemma~\ref{lem-rotation}. By
the definition we can construct a coupling such that $\widetilde{\mathbf{T}}%
_{1,n}=(\mathsf{R}_{\theta })^{n}\mathbf{T}_{n}$ and thus $\phi (\widetilde{%
\mathbf{T}}_{1,n})=\phi (\mathbf{T}_{n})$. By the definition of $\phi (%
\mathbf{t})$, we see that for any $\epsilon _{1}>0$, we can choose $K$ large
enough such that $\phi _{K}(\mathbf{t})=\phi (\mathbf{t})$ if $\theta \in 
\mathbb{Q}$, and $\phi _{K}(\mathbf{t})\leq (1-\epsilon _{1})\phi (\mathbf{t}%
)$ if $\theta \notin \mathbb{Q}$. Furthermore, by the definition of $\phi
_{K}$, for any $K$ we have that $\left\vert \phi _{K}(\mathbf{t}_{1})-\phi
_{K}(\mathbf{t}_{2})\right\vert \leq C_{1}\left\vert |\mathbf{t}_{1}-\mathbf{%
t}_{2}|\right\vert $ for some $C_{1}<\infty $ only depending on $Q$, which
further implies that $\left\vert \phi (\mathbf{t}_{1})-\phi (\mathbf{t}%
_{2})\right\vert \leq C_{1}\left\vert |\mathbf{t}_{1}-\mathbf{t}%
_{2}|\right\vert $. Therefore, if $\Vert \mathbf{W}_{n}/\sqrt{2}-\mathbf{T}%
_{n}\Vert \leq (\log N)^{4}+\epsilon _{2}\sqrt{n}$, then%
\begin{equation}
\phi (\frac{1}{\sqrt{2}}\mathbf{W}_{n})\leq C_{1}((\log N)^{4}+\epsilon _{2}%
\sqrt{n})+\phi (\mathbf{T}_{n})\leq C_{1}((\log N)^{4}+\epsilon _{2}\sqrt{n}%
)+\phi _{K}(\mathbf{T}_{n}),  \label{upper-bound-phi-Wn}
\end{equation}%
\begin{equation}
(1-\epsilon _{1})\phi (\frac{1}{\sqrt{2}}\mathbf{W}_{n})\geq
-C_{1}(1-\epsilon _{1})((\log N)^{4}+\epsilon _{2}\sqrt{n})+\phi _{K}(%
\mathbf{T}_{n}).  \label{lower-bound-phi-Wn}
\end{equation}%
Note that there is a natural coupling of $B_{n}$ and $T_{0,n}$ defined in
Lemma~\ref{lem-rotation}. Therefore for any $\epsilon >0$, letting $\epsilon
_{1}=\epsilon $ and $\epsilon _{2}=\epsilon /(3C_{1})$, by (\ref%
{lower-bound-phi-Wn}) we have that when $N$ is large enough, for any $n\geq
(\log N)^{9}$%
\begin{equation*}
\phi _{K}(\mathbf{T}_{n})\leq (1-\epsilon )\phi (\frac{1}{\sqrt{2}}\mathbf{W}%
_{n})+\frac{\epsilon }{2}\sqrt{n}.
\end{equation*}%
By Lemma~\ref{lem-rotation} and Lemma \ref{lem-couple-T-BM} we see that
there exists $c>0$ such that{\small 
\begin{equation*}
p_{N}\leq \,\mathbb{P}(c_{0}\frac{1+\epsilon }{1-\epsilon }B_{n}+\phi (\frac{%
1}{\sqrt{2}}\mathbf{W}_{n})\geq -\frac{2\epsilon }{1-\epsilon }\sqrt{n},%
\mbox{ for }(\log N)^{9}\leq n\leq N-K)+\mathrm{e}^{-c(\log N)^{2}}.
\end{equation*}%
}Let $\zeta _{n}=\max_{s\in \lbrack n,n+1]}(|B_{s}-B_{n}|+\Vert \mathbf{W}%
_{s}-\mathbf{W}_{n}\Vert )$. Clearly, for $N$ large enough a union bound
gives%
\begin{equation*}
\mathbb{P}(\exists n\in \lbrack N]:\zeta _{n}\geq (\log N)^{2})\leq \mathrm{e%
}^{-(\log N)^{2}}\,.
\end{equation*}%
Combining the last two inequalities and with Lemma, \ref{lem-random-nickdim} %
\eqref{eq-BM-upper} is implied directly.

\bigskip Denote by $n^{\star }:=(\log N)^{9}$. Define 
\begin{equation*}
A:=\{c_{0}(1-\epsilon )B_{n}+\phi (\tilde{\mathbf{T}}_{n})\geq (\epsilon
(1+C_{1})+C_{1})\sqrt{n},\mbox{ for }1\leq n\leq n^{\star }\},
\end{equation*}%
\begin{equation*}
B:=\{c_{0}(1-\epsilon )B_{n}+\phi (\tfrac{1}{\sqrt{2}}\mathbf{W}_{n})\geq
\epsilon (1+C_{1})\sqrt{n}+C_{1}(\log N)^{4},\mbox{ for }n^{\star }\leq
n\leq N-K\},
\end{equation*}%
\begin{equation*}
D:=\{\exists n\in \lbrack N]:\Vert \frac{1}{\sqrt{2}}\mathbf{W}_{n}-\mathbf{T%
}_{n}\Vert \geq (\log N)^{4}+\epsilon \sqrt{n}\}.
\end{equation*}%
Recall that for any $\epsilon >0$, when $K$ is large enough $\phi _{K}(%
\mathbf{t})\leq (1-\epsilon )\phi (\mathbf{t})$ for any $\mathbf{t}$.
Similar to the above analysis, with Lemmas~\ref{lem-rotation}, \ref%
{lem-couple-T-BM} and (\ref{lower-bound-phi-Wn}) we obtain that there exists 
$C>0$ such that{\small 
\begin{equation}
\mathbb{P}(\Omega ^{+})\geq C\mathrm{e}^{-C(\log N)^{3/4}}\mathbb{P}%
(A,B,D^{c})-\mathrm{e}^{-C(\log N)^{4/3}}\,\geq C\mathrm{e}^{-C(\log
N)^{3/4}}\mathbb{P}(B\mid A,D^{c})\mathbb{P}(A,D^{c})-\mathrm{e}^{-C(\log
N)^{4/3}}.  \label{ineq-lower-bound-P-omega-plus-ABD}
\end{equation}%
}Denote by $\widetilde{B}_{k}:=B_{n^{\star }+k}-B_{n^{\star }},\widetilde{%
\mathbf{W}}_{k}:=\mathbf{W}_{n^{\star }+k}-\mathbf{W}_{n^{\star }},$ then
obviously $\widetilde{B}_{n},\widetilde{\mathbf{W}}_{n}$ are still Brownian
motions. Note that by the definition of $\phi $ we have%
\begin{equation}
\phi (\tfrac{1}{\sqrt{2}}\mathbf{W}_{n^{\star }+k})\geq \phi (\tfrac{1}{%
\sqrt{2}}\widetilde{\mathbf{W}}_{k})+\phi (\tfrac{1}{\sqrt{2}}\mathbf{W}%
_{n^{\star }}).  \label{inequality-tirangle-phi}
\end{equation}%
Under $\{A,D^{c}\}$ it is easy to see that%
\begin{equation*}
c_{0}(1-\epsilon )B_{n^{\star }}+\phi (\tfrac{1}{\sqrt{2}}\mathbf{W}%
_{n^{\star }})\geq c_{0}(1-\epsilon )B_{n^{\star }}+\phi (\tilde{\mathbf{T}}%
_{n^{\star }})-\Vert \frac{1}{\sqrt{2}}\mathbf{W}_{n^{\star }}-\mathbf{T}%
_{n^{\star }}\Vert \geq 2\epsilon \sqrt{n^{\star }},
\end{equation*}%
combined with (\ref{inequality-tirangle-phi}) and the fact that $2\sqrt{n}%
\geq \sqrt{n+n^{\star }}-2\sqrt{n^{\star }}$, we get%
\begin{eqnarray}
\mathbb{P}(B\mid A,D^{c}) &\geq &\mathbb{P}(c_{0}(1-\epsilon )\widetilde{B}%
_{n}+\phi _{K}(\widetilde{\mathbf{W}}_{n})\geq \frac{\epsilon }{2}(\sqrt{%
n+n^{\star }}-2\sqrt{n^{\star }}),\mbox{ for }1\leq n\leq N-n^{\star }-K) 
\notag \\
&\geq &\mathbb{P}(c_{0}(1-\epsilon )\widetilde{B}_{n}+\phi _{K}(\widetilde{%
\mathbf{W}}_{n})\geq \epsilon \sqrt{n},\mbox{ for }1\leq n\leq N-n^{\star
}-K).  \label{ineq-lower-bound-P-BADC}
\end{eqnarray}%
Note that $\phi (\mathbf{t})\geq -c\left\Vert \mathbf{t}\right\Vert $ for
some $c>0$. Thus by Lemma \ref{lem-rotated-random-walk} and (\ref%
{lower-bound-integrated-prob}), we see that $\mathbb{P(}A)\geq
c_{1}(n^{\star })^{-C_{1}}$ for some $c_{1},C_{1}>0$. Noting that $\mathbb{P}%
(A,D^{c})\geq \mathbb{P(}A)-\mathbb{P(}D)$, combined with (\ref%
{ineq-lower-bound-P-omega-plus-ABD}), (\ref{ineq-lower-bound-P-BADC}), Lemma %
\ref{lem-random-nickdim} and Lemma \ref{lem-couple-T-BM}, the proof is
completed.
\end{proof}

\subsection{Continuities and discontinuities of the power with no multiple
zeroes\label{subsection-con-discon}}

In this subsection, we will use estimates on the probability for a Brownian
motion to stay in a generalized cone to deduce persistence probability of
our regressive process. Let $\mathcal{M}$ be a domain in the unit sphere $%
\mathbb{S}^{2}$ of ${\mathbb{R}}^{3}$. Denote by $\lambda (\mathcal{M})$ the
principle eigenvalue for Laplace-Beltrami operator on $\mathcal{M}$ with
Dirichlet boundary condition. In addition, define the generalized cone $%
\mathcal{C}(\mathcal{M})$ be the set of all rays emanating from the origin 0
and passing through $\mathcal{M}$. It was proved in \cite{BS97} (see also 
\cite{Burkholder77, DeBlassie87, DeBlassie88, BB96}) that for a
3-dimensional Brownian motion $\mathbf{B}$ and any compact set $\widetilde{%
\mathcal{M}}$ in the interior of $\mathcal{M}$%
\begin{equation}
\mathbb{P}_{\mathbf{x}}(\mathbf{B}_{s}\in \mathcal{C}(\mathcal{M})%
\mbox{ for
all }0\leq s\leq t)\asymp \Vert \mathbf{x}\Vert ^{\sqrt{\lambda (\mathcal{M}%
)+1/4}/2}t^{-\sqrt{\lambda (\mathcal{M})+1/4}/2}\mbox{ for all }\mathbf{x}%
\in \mathcal{C}(\widetilde{\mathcal{M}}),  \label{eq-BS}
\end{equation}%
where $\asymp $ means the \abbr{LHS} and the \abbr{RHS} are up to a constant. Furthermore,
we have%
\begin{equation}
\mathbb{P}_{\mathbf{x}}(\mathbf{B}_{s}\in \mathcal{C}(\mathcal{M})%
\mbox{ for
all }0\leq s\leq t)\leq C_{\mathcal{M}}\Vert \mathbf{x}\Vert ^{\sqrt{\lambda
(\mathcal{M})+1/4}/2}t^{-\sqrt{\lambda (\mathcal{M})+1/4}/2}\mbox{ for
all }\mathbf{x}\in \mathcal{C}(\mathcal{M}),  \label{eq-BS-2}
\end{equation}%
where $C_{\mathcal{M}}$ is a constant depending only on $\mathcal{M}$. In
view of the preceding estimates and Lemma~\ref{lem-reduce-to-BM}, we define%
\begin{equation*}
\mathcal{M}_{\Lambda }=\{\mathbf{x}\in \mathbb{S}^{2}:c_{0}x_{1}+\phi (%
\tfrac{1}{\sqrt{2}}[x_{2},x_{3}]^{T})\geq 0\}\text{, }\mathcal{M}_{\Lambda
,\epsilon }=\{\mathbf{x}\in \mathbb{S}^{2}:c_{0}(1+\epsilon )x_{1}+\phi (%
\tfrac{1}{\sqrt{2}}[x_{2},x_{3}]^{T})\geq 0\}\,.
\end{equation*}%
It is clear that $\mathcal{M}_{\Lambda ,\epsilon }$ converges to $\mathcal{M}%
_{\Lambda }$, so we have (see, e.g., \cite{BNV94})%
\begin{equation}
\lambda (\mathcal{M}_{\Lambda ,\epsilon })\rightarrow _{\epsilon \rightarrow
0}\lambda (\mathcal{M}_{\Lambda })\,.  \label{eq-eigenvalue-converge}
\end{equation}%
Applying Lemma~\ref{lem-reduce-to-BM}, \eqref{eq-BS} and \eqref{eq-BS-2},
sending $\epsilon \rightarrow 0$ and with \eqref{eq-eigenvalue-converge} we
obtain that%
\begin{equation}
p_{N}=N^{-\sqrt{\lambda (\mathcal{M}_{\Lambda })+1/4}/2+o(1)}\,,%
\mbox{ where
}o(1)\rightarrow _{N\rightarrow \infty }0\,.  \label{eq-persistence-poly}
\end{equation}%
This establishes the existence of the power decay for the persistence
probability. In what follows, we address the continuity and discontinuity
issues for the power. First consider $\theta /2\pi \not\in \mathbb{Q}$, then
for any sequence $Q_{\ell }\rightarrow Q$ denote by $\Lambda _{\ell }$ the
zero set of $Q_{\ell }$ and by $\theta _{\ell }$ the angle of the complex
zero in $\Lambda _{\ell }$. We have%
\begin{equation*}
\phi ^{(\ell )}(\mathbf{t})\rightarrow \phi (\mathbf{t})=-\Vert t\Vert
_{2}\,,
\end{equation*}%
where $\phi ^{(\ell )}(\mathbf{t})=\lim_{K\rightarrow \infty }\phi
_{K}^{(\ell )}(\mathbf{t})$ and $\phi _{K}^{(\ell )}$ is defined as in %
\eqref{eq-def-phi} but with respect to $\theta _{\ell }$ (instead of $\theta 
$). This implies that $\mathcal{M}_{\Lambda _{\ell }}\rightarrow \mathcal{M}%
_{\Lambda }$, thereby yielding $\lambda (\mathcal{M}_{\Lambda _{\ell
}})\rightarrow \lambda (\mathcal{M}_{\Lambda })$. Combined with %
\eqref{eq-persistence-poly}, the continuity of the power follows.

Next, we consider $\theta /2\pi \in \mathbb{Q}$. Take $Q_{\ell }\rightarrow
Q $ such that $\theta _{\ell }/2\pi \not\in \mathbb{Q}$. We see that $\phi
^{(\ell )}(\mathbf{t})=\Vert \mathbf{t}\Vert $ for all $\ell $ while we have 
$\phi (\mathbf{t})>-\Vert \mathbf{t}\Vert $ for almost surely all $\mathbf{t}%
\in \mathbb{R}^{2}$. This implies that $\mathcal{M}_{\Lambda }\subset 
\mathcal{M}_{\Lambda _{\ell }}$ and the Lebesgue measure of the set $%
\mathcal{M}_{\Lambda _{\ell }}\setminus \mathcal{M}_{\Lambda }$ is lower
bounded by a positive number for all $\ell $. By \cite[Theorem 2.4]{BNV94},
we deduce that $\lim_{\ell \rightarrow \infty }\lambda (\mathcal{M}_{\Lambda
_{\ell }})<\lambda _{\mathcal{M}_{\Lambda }}$. Combined with %
\eqref{eq-persistence-poly}, the discontinuity of the power follows.

\appendix

\section{Elementary facts}

We collect here elementary facts from linear algebra and analysis that we
use in this paper.

\begin{lemma}
\label{lemma-convolution} Fix $0<\theta <2\pi $ and $k\geq 0$. There exists
a constant $C<\infty $ such that for all $\theta _{0}\in \mathbb{R}$,%
\begin{equation}
|\mbox{$\sum_{i=1}^n$}\cos (i\theta +\theta _{0})i^{k}|\leq Cn^{k}\,,%
\mbox{ for
all }n\in \mathbb{N}\,.  \label{eq-lemma-convolution}
\end{equation}
\end{lemma}

\begin{proof}
We first consider the case $\theta _{0}=0$, in which obviously the lemma
holds for $\theta =\pi $, and thus we assume in what follows $\theta \neq
\pi $. Because%
\begin{align}
& \sin \theta \mbox{$\sum_{i=1}^n$}\cos (i\theta )i^{k}=\tfrac{1}{2}%
\mbox{$\sum_{i=1}^n$}(\sin ((i+1)\theta )-\sin ((i-1)\theta ))i^{k}  \notag
\\
& =\tfrac{1}{2}(n^{k}\sin (\theta (n+1))+\sum_{i=1}^{n}\sin (i\theta
)((i-1)^{k}-(i+1)^{k}))=O(n^{k})+\sum_{i=1}^{n}O(i^{k-1})=O(n^{k})\,,
\label{eq_convolution_1}
\end{align}%
dividing both sides by $\sin \theta $ completes the proof of this case. With
the similar argument we can see that%
\begin{equation}
\mbox{$\sum_{i=1}^n$}\sin (i\theta )i^{k}=O(n^{k}).  \label{eq_convolution_2}
\end{equation}%
Expanding $\cos (i\theta +\theta _{0})$ as $\cos (\theta _{0})\cos (i\theta
)-\sin (\theta _{0})\sin (i\theta )$, (\ref{eq-lemma-convolution}) is
implied by (\ref{eq_convolution_1}) and (\ref{eq_convolution_2}).
\end{proof}

\begin{lemma}
\label{lemma-polynomial} Consider a polynomial $P(z)=z^{2}+bz+c$ with $%
b^{2}-4c<0$. Then, there exists a finite $n_{0}$ and $b_{0},\ldots
b_{n_{0}}\in \mathbb{R}_{+}$ such that $\sum_{k=0}^{n_{0}}b_{k}z^{k}P(z)$ is
a non-negative polynomial.
\end{lemma}

\begin{proof}
Let $b_{0}=1$, $b_{1}=-b/c$, and $b_{k}=-\tfrac{bb_{k-1}+b_{k-2}}{c}$ for $%
k\geq 2$. It is straightforward to verify that%
\begin{equation}
Q_{n}(z)\overset{\scriptscriptstyle\triangle }{=}%
\sum_{k=0}^{n}b_{k}z^{k}P(z)=c+(b_{n-1}+bb_{n})z^{n+1}+b_{n}z^{n+2}=Q_{n-1}(z)+b_{n}z^{n}P(z)\,.
\label{eq-Q_n}
\end{equation}%
Let $\tau =\min \{k:b_{k}\leq 0\}$ and $n_{0}=\tau -1$. Recalling that $c>0$%
, it follows from \eqref{eq-Q_n} that $%
Q_{n_{0}}(z)=c+b_{n_{0}}z^{n_{0}+2}-b_{\tau }cz^{\tau }$, having all terms
non-negative. Thus, it remains to show that $\tau <\infty $. By our
assumption that $b^{2}<4c$, we see that the polynomial $cx^{2}+bx+1$ has two
conjugate zeroes, for which we denote by $\lambda $ and $\bar{\lambda}$. A
standard analysis on our recursive procedure on $b_{k}$ yields that%
\begin{equation*}
b_{k}=a_{1}\lambda ^{k}+\overline{a}_{1}\bar{\lambda}^{k}=2\mathrm{Re}%
(a_{1}\lambda ^{k})\,,
\end{equation*}%
where $a_{1}$ and $a_{2}$ are determined by the boundary conditions $b_{0}=1$
and $b_{1}=-b/c$. Since $\lambda \in \mathbb{C}\setminus \mathbb{R}$, it is
clear that $\tau <\infty $, as desired.
\end{proof}

The next corollary is an immediate consequence of Lemma \ref%
{lemma-polynomial}.

\begin{cor}
\label{cor-sub-negative} Suppose that a polynomial $Q(z)$ has no positive
zero. Then, there exists a non-negative polynomial $P(z)$ such that $Q(z)
P(z)$ is also a non-negative polynomial.
\end{cor}

\begin{proof}
By assumption, we can write $Q(z)=\prod_{i=1}^{k}(z+a_{i})%
\prod_{i=1}^{n}(z^{2}+b_{i}z+c_{i})$ such that $a_{i}\geq 0$ and $%
b_{i}^{2}<4c_{i}$ for all $i$. By Lemma~\ref{lemma-polynomial}, for each $i$
there exists non-negative polynomial $P_{i}(z)$ such that $%
P_{i}(z)(z^{2}+b_{i}z+c_{i})$ has non-negative coefficients for all terms.
Letting $P(z)=\prod_{i=1}^{n}P_{i}(z)$, we have $Q(z)P(z)=%
\prod_{i=1}^{k}(z+a_{i})\prod_{i=1}^{n}(P_{i}(z)(z^{2}+b_{i}z+c_{i}))$ is
non-negative, as required.
\end{proof}

\begin{lemma}
\label{lem-rotation-stretched-exponential} Given distinct $\theta
_{1},\ldots ,\theta _{\ell }\in (0,\pi ]$, there exists $K>0$ such that for
any $(r_{j},\gamma _{j})_{j=1,\ldots ,\ell }$ with the constraint that $%
\gamma _{j}\in \{0,\pi \}$ if $\theta _{j}=\pi $, there exists $0\leq i\leq
K $ such that%
\begin{equation*}
\sum_{j=1}^{\ell }r_{j}\cos (i\theta _{j}+\gamma _{j})\leq -\frac{1}{4}%
\max_{j}|r_{j}|\,.
\end{equation*}
\end{lemma}

\begin{proof}
Write $\alpha _{i}:=\sum_{j=1}^{\ell }r_{j}\cos (i\theta _{j}+\gamma _{j})$
and $r:=\max_{j}|r_{j}|$. Due to the fact that $\theta _{j}$'s are distinct
and belong to $(0,\pi ]$, we can apply Lemma~\ref{lemma-convolution} and see
that there exists a constant $C<\infty $ such that for all $K\in \mathbb{N}$,%
\begin{equation*}
\sum_{i=0}^{K}\alpha _{i}\leq Cr\ell ,
\end{equation*}%
\begin{equation*}
\sum_{i=0}^{K}\sum_{j=1}^{\ell }\cos (2i\theta _{j}+2\gamma _{j})\leq C\ell ,%
\text{ for those }\theta _{j}\neq \pi \text{,}
\end{equation*}%
\begin{equation}
\sum_{i=0}^{K}(\cos (i(\theta _{j_{1}}\pm \theta _{j_{2}})+\gamma
_{j_{1}}+\gamma _{j_{2}}))\leq C\text{ for all }j_{1},j_{2}\in \{1,...,\ell
\}\text{ with }j_{1}\neq j_{2}.  \label{eq-lem-rotation-stretched-1}
\end{equation}%
Note that if $\theta _{j}=\pi $ then $\sum_{i=0}^{K}\cos (2i\theta
_{j}+2\gamma _{j})=K+1$. Choosing $K=12\ell ^{2}(C\vee 1)^{2}$, by (\ref%
{eq-lem-rotation-stretched-1}) and direct expansion we have{\small 
\begin{align*}
\sum_{i=0}^{K}\alpha _{i}^{2}& =\sum_{j=1}^{\ell }\sum_{i=0}^{K}(r_{j}\cos
(i\theta _{j}+\gamma _{j}))^{2}+\sum_{j_{1}\neq
j_{2}}\sum_{i=0}^{K}(r_{j_{1}}\cos (i\theta _{j_{1}}+\gamma
_{j_{1}})(r_{j_{2}}\cos (i\theta _{j_{2}}+\gamma _{j_{2}}) \\
& \geq \sum_{|r_{j}|=r}\sum_{i=0}^{K}\frac{r_{j}^{2}}{2}(1+\cos (2i\theta
_{j}+2\gamma _{j}))-\frac{r^{2}}{2}\sum_{j_{1}\neq j_{2}}\sum_{i=0}^{K}(\cos
(i(\theta _{j_{1}}+\theta _{j_{2}})+\gamma _{j_{1}}+\gamma _{j_{2}}) \\
& +\cos (i(\theta _{j_{1}}-\theta _{j_{2}})+\gamma _{j_{1}}-\gamma
_{j_{2}}))\, \\
& \geq (\frac{K}{2}-C)r^{2}-C\ell ^{2}r^{2}\geq \frac{K}{3}r^{2}.
\end{align*}%
}Combined with the fact that $\sum_{i=0}^{K}\alpha _{i}\leq Cr\ell $, with
the method of contradiction, the conclusion of the lemma easily follows.
\end{proof}

\begin{lemma}
\label{Lemma-roration-general-case}Let $\theta _{j}$'s, $\theta _{j,k}$'s be
within $[0,2\pi )$ such that $\theta _{\lambda _{j_{1}},k}=-\theta _{\lambda
_{j_{2}},k}$ if $\theta _{j_{1}}=-\theta _{j_{2}}$, $\theta _{\lambda ,j}=0$
if $\lambda =0$ or $\pi $, and let real $c_{j_{1},k}$'s satisfy $%
c_{j_{1},k}=c_{j_{2},k}$ if $\theta _{j_{1}}=-\theta _{j_{2}}$. There exist $%
M$ and $C>0$ depending only on $\theta _{j}$'s, $c_{j,k}$'s, $\theta _{j,k}$%
's and $m^{\prime }$, such that for any $\{$\textbf{$\overrightarrow{\mathbf{%
S}}$}$_{j,n+s,k},s\in \mathbb{N}\}$ with \textbf{$\overrightarrow{\mathbf{S}}
$}$_{j,n+s,k}:=[\mathbf{\mathbf{S}}_{j,n+s,k},\mathbf{\mathbf{S}}%
_{j,n+s,k}^{\prime }]$ and%
\begin{equation}
\mathbf{\overrightarrow{\mathbf{S}}}_{j,n+s,k}=\left( \mathsf{R}_{\theta
_{j}}\right) ^{s}(\sum_{i=0}^{k}P_{s}(k-i)\mathsf{R}_{\theta _{j,k}-\theta
_{j,i}}\mathbf{\overrightarrow{\mathbf{S}}}_{j,n,i})  \label{eq-s-iter}
\end{equation}%
for%
\begin{equation}
P_{s}(x):=\binom{s-1+x}{x}\text{,}  \label{eq-def-P-s}
\end{equation}%
there exists $s\in \{1,...,M\}$ such that%
\begin{equation}
\sum_{j=1}^{\ell }\sum_{k=0}^{m_{j}-1}c_{j,k}\mathbf{S}_{j,n+s,k}\leq
-C||\sum_{j=1}^{\ell }\sum_{k=0}^{m_{j}-1}c_{j,k}\mathbf{\overrightarrow{%
\mathbf{S}}}_{j,n,k}||.  \label{lemma-stretch-T-0-n}
\end{equation}
\end{lemma}

\begin{proof}
It suffices to upper bound the left hand side of (\ref{lemma-stretch-T-0-n})
by $-C^{\prime }\sum_{j=1}^{\ell }\sum_{k=0}^{m_{j}-1}||\mathbf{%
\overrightarrow{S}}_{j,n,k}||$ for some $C^{\prime }$ depending only on $%
\theta _{j}$'s, $c_{j,k}$'s, $\theta _{j,k}$'s and $m^{\prime }$. We use
induction on $m^{\prime }=\max_{j\geq 1}m_{j}$ to show this. By our
condition on $c_{j,k}$'s, if $\theta _{j_{1}}=-\theta _{j_{2}}$ then $%
c_{j_{1},k}\mathbf{S}_{j_{1},n+s,k}=c_{j_{2},k}\mathbf{S}_{j_{2},n+s,k}$.
Thus we can just consider those $\theta _{j}$'s $\in (0,\pi ]$. Due to $%
P_{s}(1)=1$ for all $s$, it is easy to see that when $m^{\prime }=1$ this
lemma is implied by Lemma \ref{lem-rotation-stretched-exponential}. Assume
that the lemma holds for $m^{\prime }-1$. Define%
\begin{equation*}
\mathbf{\overrightarrow{\mathbf{S}}}_{j,n,k}^{\star },=\left\{ 
\begin{array}{c}
\mathbf{\overrightarrow{\mathbf{S}}}_{j,n,k+1}\text{for }k=0,...,m_{j}-2%
\text{, if }m_{j}=m^{\prime } \\ 
\mathbf{\overrightarrow{\mathbf{S}}}_{j,n,k}\text{, if }m_{j}<m^{\prime }%
\end{array}%
\right. ,
\end{equation*}%
and similarly define \textbf{$\overrightarrow{\mathbf{S}}$}$%
_{j,n+s,k}^{\star }$ by relation (\ref{eq-s-iter}) for $s\in \mathbb{Z}^{+}$%
. Then by induction for $m^{\prime }-1$ and the definition of \textbf{$%
\overrightarrow{\mathbf{S}}$}$_{j,n,k}^{\star }$, we can find $M^{\prime }$
and $C^{\prime \prime }>0$ such that for any $\{$\textbf{$\overrightarrow{%
\mathbf{S}}$}$_{j,n+s,k},s\in \mathbb{Z}^{+}\}$, there exists $s\in
\{1,...,M^{\prime }\}$ such that%
\begin{eqnarray}
&&\sum_{\{j:m_{j}=m^{\prime }\}}\sum_{k=1}^{m_{j}-1}c_{j,k}\mathbf{S}%
_{j,n+s,k}+\sum_{\{j:m_{j}<m^{\prime }\}}\sum_{k=0}^{m_{j}-1}c_{j,k}\mathbf{S%
}_{j,n+s,k}  \notag \\
&\leq &-C^{\prime \prime }(\sum_{\{j:m_{j}=m^{\prime
}\}}\sum_{k=1}^{m_{j}-1}||\mathbf{\overrightarrow{\mathbf{S}}}%
_{j,n,k}||+\sum_{\{j:m_{j}<m^{\prime }\}}\sum_{k=0}^{m_{j}-1}||\mathbf{%
\overrightarrow{\mathbf{S}}}_{j,n,k}||).  \label{assumption-induction}
\end{eqnarray}%
Thus if{\small 
\begin{equation}
\sum_{\{j:m_{j}=m^{\prime }\}}\left\Vert \mathbf{\overrightarrow{\mathbf{S}}}%
_{j,n,0}\right\Vert \leq \frac{C^{\prime \prime }(\sum_{\{j:m_{j}=m^{\prime
}\}}^{{}}\sum_{k=1}^{m_{j}-1}||\mathbf{\overrightarrow{\mathbf{S}}}%
_{j,n,k}||+\sum_{\{j:m_{j}<m^{\prime }\}}\sum_{k=1}^{m_{j}-1}||\mathbf{%
\overrightarrow{\mathbf{S}}}_{j,n,k}||)}{4\left( \ell
\max_{j}\{c_{j,0}\}P_{M^{\prime }}(m^{\prime })\vee C^{\prime \prime
}\right) },  \label{condition-T-j-n-1}
\end{equation}%
}by (\ref{eq-T-j-n+s-k}) and (\ref{assumption-induction}) it is easy to
verify that this lemma holds by letting $C^{\prime }=C^{\prime \prime }/4$.
Now we consider the case when (\ref{condition-T-j-n-1}) doesn't hold. By
Lemma \ref{lem-rotation-stretched-exponential} we see that there exist $%
M^{\prime \prime }$ and $C^{\prime \prime \prime }>0$ such that for any $%
K_{1}\in \mathbb{Z}^{+}$ we can always find $s\in
\{K_{1},...,K_{1}+M^{\prime \prime }\}$ with%
\begin{equation}
\sum_{\{j:m_{j}=m^{\prime }\}}c_{j,m^{\prime }+1}[1,0](\mathsf{R}_{\theta
_{j}})^{s}\mathsf{R}_{\theta _{j,m^{\prime }}-\theta _{j,1}}\mathbf{%
\overrightarrow{\mathbf{S}}}_{j,n,0}<-C^{\prime \prime \prime
}\sum_{\{j:m_{j}=m^{\prime }\}}\left\Vert \mathbf{\overrightarrow{\mathbf{S}}%
}_{j,n,0}\right\Vert .  \label{bound-lemma-rotation-case-2}
\end{equation}%
Note that from (\ref{eq-def-P-s}) we have that $P_{s}(m^{\prime
})/P_{s}(k)\rightarrow \infty $ as $s\rightarrow \infty $ for any $%
k<m^{\prime }$. Thus by (\ref{eq-s-iter}) and the fact that (\ref%
{condition-T-j-n-1}) doesn't hold, there exists $K\in \mathbb{Z}^{+}$ only
depending on $M^{\prime }$ and $c_{j,k}$'s such that when $s>K${\small 
\begin{equation}
\frac{P_{s}(m^{\prime })\sum_{\{j:m_{j}=m^{\prime }\}}\left\Vert \mathbf{%
\overrightarrow{\mathbf{S}}}_{j,n,0}\right\Vert }{||\sum_{j=1}^{\ell
}\sum_{k=0}^{m_{j}-1}c_{j,k}\mathbf{\overrightarrow{\mathbf{S}}}%
_{j,n+s,k}-P_{s}(m^{\prime })\sum_{\{j:m_{j}=m^{\prime }\}}c_{j,m^{\prime
}+1}(\mathsf{R}_{\theta _{j}})^{s}\mathsf{R}_{\theta _{j,m^{\prime }}-\theta
_{j,1}}\mathbf{\overrightarrow{\mathbf{S}}}_{j,n,0}||}>\frac{4}{C^{\prime
\prime \prime }}.  \label{bound-lemma-rotation-case-2-2}
\end{equation}%
} Choosing $K_{1}=K$, then by (\ref{bound-lemma-rotation-case-2}) and (\ref%
{bound-lemma-rotation-case-2-2}) we can always find $s\in
\{K_{1},...,K_{1}+M^{\prime \prime }\}$ such that{\small 
\begin{equation*}
\sum_{j=1}^{\ell }\sum_{k=0}^{m_{j}-1}c_{j,k}\mathbf{S}_{j,n+s,k}\leq -\frac{%
3}{4}P_{s}(m^{\prime })C^{\prime \prime \prime }\sum_{\{j:m_{j}=m^{\prime
}\}}\left\Vert \mathbf{\overrightarrow{\mathbf{S}}}_{j,n,0}\right\Vert
<-C_{1}\sum_{j=1}^{\ell }\sum_{k=0}^{m_{j}-1}||\mathbf{\overrightarrow{%
\mathbf{S}}}_{j,n,k}||,
\end{equation*}%
}for some $C_{1}>0$, where the rightmost inequality is again by the fact
that (\ref{condition-T-j-n-1}) doesn't hold. Letting $M=M^{\prime }\vee
M^{\prime \prime }$, the proof is completed.
\end{proof}

\begin{lemma}
\label{lem-Jordan-matrix} For a polynomial $P$, let $\Lambda $ be its zero
set and for $\lambda \in \Lambda $ denote by $m(\lambda )$ the multiplicity
of the eigenvalue $\lambda $. Define an $\ell \times \ell $ matrix $\mathsf{A%
}=\mathsf{A}(P)$ by%
\begin{equation}
\mathsf{A}_{1,j}=a_{j}\mbox{ for }1\leq j\leq \ell ;\quad \mathsf{A}%
_{i,i-1}=1\mbox{
for }2\leq i\leq \ell ;\quad \mathsf{A}_{i,j}=0\mbox{ otherwise}.
\end{equation}%
Then there exist a basis $\{\mathbf{v}_{\lambda ,j}:\lambda \in \Lambda 
\mbox{
and }j=1,\ldots ,m(\lambda )\}$ in $\mathbb{R}^{L}$ which contains all the
eigenvectors of $\mathsf{A}$ such that for any $\mathbf{x}=\sum_{\lambda \in
\Lambda }\sum_{j=1}^{m(\lambda )}c_{\lambda ,j}\mathbf{v}_{\lambda ,j}$, we
have%
\begin{equation}
\mathsf{A}^{n}\mathbf{x}=\sum_{\lambda \in \Lambda }\sum_{r=1}^{m(\lambda
)}\sum_{j=0}^{m(\lambda )-r}\binom{n}{j}\lambda ^{n-j}c_{\lambda ,r+j}%
\mathbf{v}_{\lambda ,r}\,.  \label{eq-lemma-Hordan-matrix}
\end{equation}%
Furthermore, we have $\bar{\mathbf{v}}_{\lambda ,j}=\mathbf{v}_{\bar{\lambda}%
,j}$ and $\bar{c}_{\lambda ,j}=c_{\bar{\lambda},j}$.
\end{lemma}

\begin{proof}
With $a_{L}\neq 0$ we know that $0\not\in \Lambda $. Also, the special form
of $\mathsf{A}$ results with each $\lambda \in \Lambda $ having geometric
multiplicity one (as $\mathbf{x}\in \mathrm{ker}(\mathsf{A}-\lambda \mathsf{I%
})$ implies that $x_{k}=\lambda ^{-1}x_{k+1}$ for $k=1,\ldots ,L-1$). Recall
that $\mathsf{A}=\mathsf{V}\mathsf{J}\mathsf{V}^{-1}$, with $\mathsf{J}$ the
Jordan normal form of $\mathsf{A}$, consisting here of one (maximal) block
of dimension $m(\lambda )$ per $\lambda \in \Lambda $. Specifically, the
block $\mathsf{J}_{\lambda }\overset{\scriptscriptstyle\triangle }{=}\lambda 
\mathsf{I}+N(m(\lambda ))$ with nilpotent $N_{i,k}(m)=\mathbf{1}_{k=i+1}$
for $i=1,\ldots ,m-1$ such that $N^{r}(m)=\mathbf{1}_{k=i+r}$ for $%
i=1,\ldots ,m-r$ for $1\leq r\leq m$. The columns of $\mathsf{V}=\{\mathbf{v}%
_{\lambda ,j}:j=1,\ldots ,m(\lambda ),\lambda \in \Lambda \}$ are such that $%
\mathbf{v}_{\lambda ,1}$ is the eigenvector of $\mathsf{A}$ for eigenvalue $%
\lambda $ and $(\mathsf{A}-\lambda \mathsf{I})\mathbf{v}_{\lambda ,j}=%
\mathbf{v}_{\lambda ,j-1}$ for $j=2,\ldots ,m(\lambda )$. It is easy to
verify that the power $\mathsf{J}^{n}$ of such a Jordan normal form consists
of block diagonal upper-triangular matrices of dimension $m(\lambda )$,
where $\mathsf{J}_{\lambda }^{n}=\binom{n}{j}\lambda ^{n-j}$ at positions $%
(r,r+j)$ for $1\leq r\leq m(\lambda )$ and $1\leq j\leq m(\lambda )-r$.
Thus, $\mathsf{A}^{n}\mathbf{x}=\mathsf{V}\mathsf{J}^{n}(\mathsf{V}^{-1}%
\mathbf{x})$, so writing $\mathbf{x}=\sum_{\lambda ,j}c_{\lambda ,j}\mathbf{v%
}_{\lambda ,j}$, we have that $\mathsf{V}^{-1}\mathbf{x}=[c_{\lambda
,j}]_{\lambda \in \Lambda ,1\leq j\leq m(\lambda )}$. This then completes
the proof of the lemma.Before the proof we need the following lemma.
\end{proof}

\begin{lemma}
\label{lem-polynomial-approximation}Assume $L\geq m>0$. Then there exists a $%
C>0$, such that for any polynomial $g(x)=\sum_{j=0}^{m-1}c_{j+1}x^{j}$, we
can find a $y\in \{1/L,2/L,...,1\}$ such that%
\begin{equation*}
\left\vert g(y)\right\vert \geq C\max_{1\leq j\leq m}|c_{j}|.
\end{equation*}
\end{lemma}

\begin{proof}
For any $i\in \lbrack m]$, we can find $a_{i,k}$, $k=1,2,...,L$, such that
for $j\in \lbrack m]$, $\sum_{k=1}^{L}a_{i,k}(k/L)^{j-1}=\delta _{ij}$.
Therefore we have that%
\begin{equation}
\sum_{j=1}^{L}a_{i,k}g((k-1)/L)=c_{i}.  \label{eq-lem-add-polynomial}
\end{equation}%
Denote by $M:=\max_{1\leq i,k\leq m}|a_{i,k}|$. \abbr{WLOG}
we assume that $c_{i+1}=\max_{1\leq j\leq m}|c_{j}|$, thus by (\ref%
{eq-lem-add-polynomial}) obviously we can find $y\in \{1/L,2/L,...,1\}$ such
that $\left\vert g(y)\right\vert \geq \max_{1\leq j\leq m}|c_{j}|/(LM)$.
\end{proof}

{\small 
\bibliographystyle{plain}
\bibliography{persistence}
}

\end{document}